%% file: correlators.tex
\documentclass[10pt,a4paper,pointlessnumbers]{scrartcl}
\usepackage[T1]{fontenc}

\usepackage{amssymb}
\usepackage{amsmath}
\usepackage{amsfonts}
\usepackage{bbm}
\usepackage{amsthm}
\usepackage{lmodern}
\usepackage{mathrsfs}
\usepackage[hidelinks]{hyperref}\usepackage{color}
\usepackage[margin=2.5cm]{geometry}
\usepackage[all,cmtip]{xy}
\usepackage[utf8]{inputenc}
\usepackage{graphicx}

\usepackage{varwidth}
\usepackage{overpic}

\makeatletter
\DeclareRobustCommand{\em}{%
	\@nomath\em \if b\expandafter\@car\f@series\@nil
	\normalfont \else \slshape \fi}
\makeatother
\usepackage{upgreek}	
\usepackage{rotating}
\usepackage{tikz}
\usetikzlibrary{matrix,arrows,decorations.pathmorphing,shapes.geometric}
\usepackage{tikz-cd}
\usepackage{needspace}
\newcommand{\spaceplease}{\needspace{5\baselineskip}}
\usetikzlibrary{decorations.markings}

\usetikzlibrary{backgrounds}

\newcommand{\SkAlg}{\catf{SkAlg}}
\usepackage{todonotes}

\newcommand{\HOM}{\underline{\catf{Hom}}} \newcommand{\END}{\underline{\catf{\End}}}

\usepackage{tikz-cd}

\usetikzlibrary{decorations.markings}
\usetikzlibrary{backgrounds}

\usepackage{etex}

\tikzstyle{tikzfig}=[baseline=-0.25em,scale=0.5]
\tikzstyle{tikzfigs}=[baseline=-0.25em,scale=0.4]
\pgfkeys{/tikz/tikzit fill/.initial=0}
\pgfkeys{/tikz/tikzit draw/.initial=0}
\pgfkeys{/tikz/tikzit shape/.initial=0}
\pgfkeys{/tikz/tikzit category/.initial=0}

\pgfdeclarelayer{edgelayer}
\pgfdeclarelayer{nodelayer}
\pgfsetlayers{background,edgelayer,nodelayer,main}

\tikzstyle{none}=[inner sep=0mm]

\newcommand{\tikzfig}[1]{%
	{\tikzstyle{every picture}=[tikzfig]
		\IfFileExists{#1.tikz}
		{\input{#1.tikz}}
		{%
			\IfFileExists{./figures/#1.tikz}
			{\input{./figures/#1.tikz}}
			{\tikz[baseline=-0.5em]{\node[draw=red,font=\color{red},fill=red!10!white] {\textit{#1}};}}%
	}}%
}

\newcommand{\tikzfigs}[1]{%
	{\tikzstyle{every picture}=[tikzfigs]
		\IfFileExists{#1.tikz}
		{\input{#1.tikz}}
		{%
			\IfFileExists{./figures/#1.tikz}
			{\input{./figures/#1.tikz}}
			{\tikz[baseline=-0.5em]{\node[draw=red,font=\color{red},fill=red!10!white] {\textit{#1}};}}%
	}}%
}

\tikzstyle{every loop}=[]

\usepackage{tikzit}
\input{mydiagrams.tikzstyles}
\usetikzlibrary{decorations.pathreplacing,decorations.markings}
\tikzset{
	on each segment/.style={
		decorate,
		decoration={
			show path construction,
			moveto code={},
			lineto code={
				\path [#1]
				(\tikzinputsegmentfirst) -- (\tikzinputsegmentlast);
			},
			curveto code={
				\path [#1] (\tikzinputsegmentfirst)
				.. controls
				(\tikzinputsegmentsupporta) and (\tikzinputsegmentsupportb)
				..
				(\tikzinputsegmentlast);
			},
			closepath code={
				\path [#1]
				(\tikzinputsegmentfirst) -- (\tikzinputsegmentlast);
			},
		},
	},
	mid arrow/.style={postaction={decorate,decoration={
				markings,
				mark=at position .7 with {\arrow[#1]{stealth}}
	}}},
}
\tikzset{%
	link/.style    = { white, double = black, line width = 1.8pt,
		double distance = 0.4pt },
	channel/.style = { white, double = black, line width = 0.8pt,
		double distance = 0.8pt },
}

\tikzset{%
	blink/.style    = { white, double = blue, line width = 2pt,
		double distance = 1pt },
	channel/.style = { white, double = blue, line width = 2pt,
		double distance = 1pt },
}

\tikzstyle{tikzfig}=[baseline=-0.25em,scale=0.5]

\pgfkeys{/tikz/tikzit fill/.initial=0}
\pgfkeys{/tikz/tikzit draw/.initial=0}
\pgfkeys{/tikz/tikzit shape/.initial=0}
\pgfkeys{/tikz/tikzit category/.initial=0}

\pgfdeclarelayer{edgelayer}
\pgfdeclarelayer{nodelayer}
\pgfsetlayers{background,edgelayer,nodelayer,main}

\tikzstyle{none}=[inner sep=0mm]

\tikzstyle{every loop}=[]

\usepackage{mathtools}
\mathtoolsset{showonlyrefs}
\usepackage{epstopdf}

\newtheoremstyle{mytheorem}
{\topsep}
{\topsep}
{\slshape}
{0pt}
{\bfseries}
{.}
{ }
{\thmname{#1}\thmnumber{ #2}\thmnote{ {\normalfont\slshape(#3)}}}

\newtheoremstyle{mydefinition}
{\topsep}
{\topsep}
{\normalfont}
{0pt}
{\bfseries}
{.}
{ }
{\thmname{#1}\thmnumber{ #2}\thmnote{ {\normalfont\slshape(#3)}}}

\theoremstyle{mytheorem}
\newtheorem{theorem}{Theorem}[section]

\makeatletter
\newtheorem*{rep@theorem}{\rep@title}
\newcommand{\newreptheorem}[2]{%
	\newenvironment{rep#1}[1]{%
		\def\rep@title{#2 \ref{##1}}%
		\begin{rep@theorem}}%
		{\end{rep@theorem}}}
\makeatother

\newreptheorem{theorem}{Theorem}
\newreptheorem{corollary}{Corollary}
\newtheorem{lemma}[theorem]{Lemma}
\newtheorem{proposition}[theorem]{Proposition}
\newtheorem{corollary}[theorem]{Corollary}
\theoremstyle{mydefinition}
\newtheorem{definition}[theorem]{Definition}
\newtheorem{exx}[theorem]{Example}
\newenvironment{example}
{\pushQED{\qed}\exx}
{\popQED\endexx}
\newtheorem{remm}[theorem]{Remark}
\newenvironment{remark}
{\pushQED{\qed}\remm}
{\popQED\endremm}
\numberwithin{equation}{section}
\usepackage{enumitem}

\newenvironment{pnum}{\begin{enumerate}[topsep=2pt,parsep=2pt,partopsep=2pt,itemsep=0pt,label={(\roman{*})}]}{\end{enumerate}}

\DeclareMathSymbol{\Phiit}{\mathalpha}{letters}{"08}\let\Phi\undefined\newcommand{\Phi}{\Phiit}
\DeclareMathSymbol{\Psiit}{\mathalpha}{letters}{"09}\let\Psi\undefined\newcommand{\Psi}{\Psiit}
\DeclareMathSymbol{\Sigmait}{\mathalpha}{letters}{"06}\let\Sigma\undefined\newcommand{\Sigma}{\Sigmait}
\DeclareMathSymbol{\Xiit}{\mathalpha}{letters}{"04}
\DeclareSymbolFont{extraup}{U}{zavm}{m}{n}
\DeclareMathSymbol{\vardiamond}{\mathalpha}{extraup}{87}
\DeclareMathSymbol{\Lambdait}{\mathalpha}{letters}{"03}\let\Lambda\undefined\newcommand{\Lambda}{\Lambdait}
\DeclareMathSymbol{\Piit}{\mathalpha}{letters}{"05}\let\Pi\undefined\newcommand{\Pi}{\Piit}
\DeclareMathSymbol{\Gammait}{\mathalpha}{letters}{"00}\let\Gamma\undefined\newcommand{\Gamma}{\Gammait}
\DeclareMathSymbol{\Omegait}{\mathalpha}{letters}{"0A}\let\Omega\undefined\newcommand{\Omega}{\Omegait}
\DeclareMathSymbol{\Upsilonit}{\mathalpha}{letters}{"07}\let\Upsilon\undefined\newcommand{\Upsilon}{\Upilonit}
\DeclareMathSymbol{\Thetait}{\mathalpha}{letters}{"02}\let\Theta\undefined\newcommand{\Theta}{\Thetait}

\def\Hom{\catf{Hom}}
\let\O\undefined\newcommand{\O}{\catf{O}}

\def\End{\catf{End}}
\def\id{\mathrm{id}}
\def\SL{\catf{SL}}

\def\dim{\mathrm{dim}}
\let\to\undefined\newcommand{\to}{\longrightarrow}
\let\mapsto\undefined\newcommand{\mapsto}{\longmapsto}
\newcommand{\catf}[1]{\mathsf{#1}}
\newcommand{\Proj}{\operatorname{\catf{Proj}}}

\newcommand{\Map}{\catf{Map}}

\newcommand{\dg}{^{\normalfont\textbf{\dag}}}

\def\op{\mathrm{op}}
\newcommand{\bzero}{\catf{B}^\circ}\newcommand{\psizero}{\psi^\circ}
\newcommand{\bulk}{\catf{B}}

\newcommand{\qss}{\cat{O}_\Sigma^\cat{A}}
\newcommand{\qssp}{\cat{O}_{\Sigma'}^\cat{A}}

\newcommand{\LO}{\mathscr{L}}

\newcommand{\Ao}{\cat{A}_{\text{\normalfont \bfseries !}}}

\newcommand{\ra}[1]{\xrightarrow{\   #1    \ }}

\newcommand{\Lexf}{\catf{Lex}^\catf{f}}
\newcommand{\Rexf}{\catf{Rex}^\catf{f}}

\newcommand{\cupx}{\stackrel{\times}{\cup}}
\newcommand{\trace}{\catf{t}}

\newcommand{\moduli}{\mathfrak{a}}

\newcommand{\act}{\triangleright}
\newcommand{\actcor}{\blacktriangleright}

\newcommand{\FA}{\mathfrak{F}_{\! \cat{A}}}
\newcommand{\FAA}{\mathfrak{F}_{\! \bar{\cat{A}}\boxtimes \cat{A}}}

\newcommand{\cor}{\xi}
\newcommand{\precor}{\lambda}

\newcommand{\Fbar}{\mathring{F}}
\newcommand{\Fhb}{F^\vardiamond}

\newcommand{\vect}{\catf{vect}}

\newcommand{\spr}[1]{\left\langle #1\right\rangle}
\newcommand{\cat}[1]{\mathcal{#1}}

\newcommand{\Ext}{\catf{Ext}}

\newcommand{\Surf}{\catf{Surf}}

\newcommand{\SkA}{\catf{Sk}_{\! \cat{A}}}

\newcommand{\Legs}{\catf{Legs}}
\newcommand{\Add}{\catf{Add}}

\newcommand{\PhiA}{\Phi_{\! \cat{A}}}
\newcommand{\elliptic}{\mathcal{D}_{\! \cat{A}}}

\newcommand{\skA}{\catf{sk}_\cat{A}}

\newcommand{\Fmod}{F\catf{-mod}_\cat{A}}\newcommand{\rmod}{\catf{mod}_{\cat{A}}\!-\!}
\definecolor{Blue}  {rgb} {0.282352,0.239215,0.803921}
\definecolor{Green} {rgb} {0.133333,0.545098,0.133333}
\definecolor{Red}   {rgb} {0.803921,0.000000,0.000000}
\definecolor{Violet}{rgb} {0.580392,0.000000,0.827450}

\usepackage{dingbat}

\makeatletter
\newcommand{\monthyeardate}{%
	\DTMenglishmonthname{\@dtm@month}, \@dtm@year
}
\makeatother

\setkomafont{caption}{\small\slshape}

\setkomafont{section}{\rmfamily\large}
\setkomafont{subsection}{\normalfont\slshape}
\setkomafont{subsubsection}{\rmfamily}
\setkomafont{subparagraph}{\normalfont\scshape}

\newtheorem*{theorem*}{Theorem}
\newtheorem*{corollary*}{Corollary}

\makeatletter
\renewcommand\section{\@startsection {section}{1}{\z@}%
	{-3.5ex \@plus -1ex \@minus -.2ex}%
	{2.3ex \@plus.2ex}%
	{\normalfont\scshape\centering}}
\makeatother

\usepackage{titlesec}
\titleformat{\subsection}[runin]
{\slshape}
{\thesubsection}
{0.5em}
{}
[.]
\renewcommand\thepart{\Alph{part}}
\titleformat{\part}
{\Large\bfseries\centering}
{\thepart}
{0.5em}
{}
[]
\usepackage{titletoc}
\dottedcontents{part}[0.5cm]{\rmfamily\bfseries}{0.5cm}{0.2cm}
\dottedcontents{section}[1.5cm]{\rmfamily}{0.5cm}{0.2cm}

\setkomafont{disposition}{\normalfont\bfseries}

\begin{document}

	\vspace*{-0.5cm}	\begin{center}	\textbf{\large{The Construction of Correlators in \\[0.5ex]
				Finite Rigid Logarithmic Conformal Field Theory}}\\	\vspace{1cm}{\large Lukas Woike }\\ 	\vspace{5mm}{\slshape  Université Bourgogne Europe\\ CNRS\\ IMB UMR 5584\\ F-21000 Dijon\\ France }\end{center}	\vspace{0.3cm}	
	\begin{abstract}\noindent 
		For logarithmic conformal field theories whose monodromy data is given by a not necessarily semisimple modular category, we solve the problem of constructing and classifying the consistent systems of correlators. The correlator construction given in this article applies to the open-closed sector and generalizes the well-known one for rational conformal field theories given by Fuchs-Runkel-Schweigert roughly twenty years ago and solves conjectures of Fuchs, Gannon, Schaumann and Schweigert. The strategy is, even in the rational special case, entirely different. The correlators are constructed using the extension procedures that can be devised by means of the modular microcosm principle. It is shown that, as in the rational case, the correlators admit a holographic description, with the main difference that the holographic principle is phrased in terms of factorization homology. The latter description is used to prove that the coefficients obtained by the evaluation of the torus partition function at projective objects are non-negative integers. Moreover, we show that the derived algebra of local operators associated to a consistent system of correlators carries a Batalin-Vilkovisky structure. We prove that it is equivalent to the Batalin-Vilkovisky structure on the Hochschild cohomology of the pivotal module category of boundary conditions, for the notion of pivotality due to Schaumann and Shimizu. This proves several expectations formulated by Kapustin-Rozansky and Fuchs-Schweigert for general conformal field theories.
\end{abstract}

	\tableofcontents
	\normalsize

	\section{Introduction and summary}
	The monodromy data of a two-dimensional conformal field theory can be described by providing for 
	 each compact oriented surface with a label attached to each boundary component a vector space, called the \emph{space of conformal blocks} for that surface.
	This vector space comes with an action of an extension of the mapping class group of that surface. 
	The spaces of conformal blocks are organized into what is called a \emph{modular functor} \cite{Segal,ms89,turaev,tillmann,baki,jfcs}. Restriction to genus zero endows the label set with the structure of  a certain linear balanced braided monoidal category $\cat{A}$ with a weak form of duality~\cite{baki,cyclic}. 
The term \emph{monodromy data} is sometimes used for the modular functor in its entirety or just the category $\cat{A}$. 

	In this article, we are interested in the situation in which $\cat{A}$ is 
a \emph{modular category}, a linear balanced braided monoidal category that is subject to finiteness conditions and a non-degeneracy condition on the braiding and is moreover  \emph{rigid}, i.e.\ all objects in $\cat{A}$ have duals. If $\cat{A}$ is additionally semisimple, then it is called a \emph{modular fusion category}. By a  result of Huang~\cite{huang} such a category arises for example by taking modules over a \emph{(strongly) rational} vertex operator algebra; for this reason, this situation is often 
 referred to as the `rational case'.
Not necessarily  semisimple
modular categories are set to be the monodromy data for \emph{finite rigid logarithmic conformal field theories} in \cite{jfcs,fspivotal}, see Section~\ref{secmodel} for an independent justification for why this is the correct input datum. 
Dropping the requirement of semisimplicity is a vast generalization that comes with tremendous challenges. One still has, however, finiteness and rigidity as remaining guardrails. Sources for non-semisimple modular categories include suitable vertex operator algebras or Hopf algebras; some modular categories arise from both sources at once. Some relatively recent works in this very active field include~\cite{lo,glo,creutziggainutdinovrunkel,creutziglentnerrupert,mcrae,gannonnegron}.

The modular functor $\FA$ associated to a
not necessarily semisimple modular category $\cat{A}$
was constructed by Lyubashenko \cite{lyubacmp,lyu,lyulex,kl}. For a surface $\Sigma$ (in this article always smooth, compact and oriented) with $n\ge 0$ boundary components, the modular functor gives us a left exact functor $\FA(\Sigma;-): \cat{A}^{\boxtimes n}\to \vect$ taking as input the $n$ boundary labels, where $\boxtimes$ is the Deligne product and $\vect$ is the category of finite-dimensional vector spaces over our fixed algebraically closed ground field $k$. 
The functor $\FA(\Sigma;-)$ comes with a projective action of the mapping class group of $\Sigma$.
Moreover, the assignment $\Sigma \mapsto \FA(\Sigma;-)$ is compatible with gluing. 
It is justified to call $\FA$ `the' modular functor associated to $\cat{A}$ because $\FA$ is, up to equivalence, the only modular functor extending the genus zero spaces of conformal blocks given by the morphism spaces of $\cat{A}$. This is a consequence of the results of~\cite{brochierwoike}. The definition of the notion of a modular functor is recalled in Section~\ref{secmf}.

	The construction and study of the spaces of conformal blocks is a part of \emph{chiral conformal field theory}. It is only the first layer of the construction.
	The next step is to pass to 
	 \emph{full conformal field theory}.
	This has the mathematically precise meaning of finding all \emph{consistent systems of correlators}, which are vectors in the spaces of conformal blocks 
	invariant under the mapping class group action and compatible with the gluing of surfaces. 
Exhibiting the consistent systems of correlators for a modular functor describing the monodromy data of a conformal field theory can be understood as a mathematically rigorous incarnation of finding the `solutions' of the theory.	

	To describe the notion of a consistent system of correlators more precisely, we follow the encyclopedia entry~\cite{algcften}: For a modular category $\cat{A}$, consider the balanced braided category $\bar{\cat{A}}$ that is the same monoidal category as $\cat{A}$, but with inverted braiding and balancing. It is again a modular category, and so is $\bar{\cat{A}} \boxtimes \cat{A}$. The shift of interest from $\cat{A}$ to  $\bar{\cat{A}} \boxtimes \cat{A}$ comes, physically speaking, from the fact that we need to combine \emph{left and right movers}; this is a part of the \emph{process of holomorphic factorization}.

	A \emph{consistent system of correlators for the conformal field theory with monodromy data $\cat{A}$} has an underlying object $B \in \bar{\cat{A}} \boxtimes \cat{A}$, called the \emph{bulk field}, and vectors $\cor_\Sigma^B \in \FAA(\Sigma;B,\dots,B)$ for each surface $\Sigma$ with all boundary components labeled by $B$ such that the $\cor_\Sigma^B$ are mapping class group invariant and  compatible with the gluing along boundary components. The latter is expressed by saying that the $\cor_\Sigma^B$ satisfy the \emph{sewing constraints}.

	Actually, we will ask for a little more: The modular functor $ \FAA$ is 
	 an \emph{open-closed modular functor}~\cite{sn}, i.e.\
	it is also consistently defined on surfaces with marked intervals in the boundary.
	In this situation, we will not only ask for a bulk field $B \in \bar{\cat{A}}\boxtimes\cat{A}$ attached to the boundary circles, but also a \emph{boundary field} $F\in\cat{A}$ attached to the marked intervals, see Figure~\ref{figocsurface}. Again, one demands invariance under the mapping class group and gluing.
	
	\begin{figure}[h]
		\centering
		\begin{tikzpicture}[scale=0.4]
			\begin{pgfonlayer}{nodelayer}
				\node [style=none] (0) at (-10, 3) {};
				\node [style=none] (1) at (-10, 1) {};
				\node [style=none] (2) at (-10, -1) {};
				\node [style=none] (3) at (-10, -3) {};
				\node [style=none] (4) at (-7.5, 0.75) {};
				\node [style=none] (5) at (-6.5, 0.75) {};
				\node [style=none] (6) at (-7.75, 1.25) {};
				\node [style=none] (7) at (-6.25, 1.25) {};
				\node [style=none] (8) at (-5.5, -0.75) {};
				\node [style=none] (9) at (-4.5, -0.75) {};
				\node [style=none] (10) at (-5.75, -0.25) {};
				\node [style=none] (11) at (-4.25, -0.25) {};
				\node [style=none] (12) at (-9.25, 2.5) {};
				\node [style=none] (13) at (-9.25, 1.5) {};
				\node [style=none] (14) at (-10, 3) {};
				\node [style=none] (15) at (-10.75, 1.5) {};
				\node [style=none] (16) at (-11.5, -2) {$B$};
				\node [style=none] (17) at (-11.5, 2.5) {$F$};
				\node [style=none] (18) at (-8.5, 2) {$F$};
				\node [style=none] (19) at (-11.5, -2.5) {};
				\node [style=none] (20) at (-10, -4.5) {};
				\node [style=none] (21) at (-10, -5) {\small \slshape bulk object $B$ associated to every parametrized boundary circle};
				\node [style=none] (22) at (-10, 4) {};
				\node [style=none] (23) at (-11.25, 3) {};
				\node [style=none] (24) at (-8.5, 2.5) {};
				\node [style=none] (25) at (-10, 5.75) {\small \slshape boundary object $F$ associated to every parametrized boundary interval};
				\node [style=none] (26) at (-10, 5.25) {};
				\node [style=none] (27) at (-2, 0) {};
				\node [style=none] (28) at (1, 0) {};
				\node [style=none] (29) at (-2, 0.25) {};
				\node [style=none] (30) at (-2, -0.25) {};
				\node [style=none] (31) at (6.5, 0) {$\FA(\Sigma;F,F,B) \curvearrowleft \Map(\Sigma)$};
				\node [style=none] (32) at (-7.5, -1.75) {$\Sigma$};
				\node [style=none] (33) at (6.5, -1) {space of conformal blocks};
			\end{pgfonlayer}
			\begin{pgfonlayer}{edgelayer}
				\draw [bend left=90, looseness=1.50] (0.center) to (1.center);
				\draw [bend right=90, looseness=1.50] (0.center) to (1.center);
				\draw [bend left=90, looseness=3.00] (1.center) to (2.center);
				\draw [bend left=90, looseness=4.00] (0.center) to (3.center);
				\draw [bend left=90, looseness=1.25] (4.center) to (5.center);
				\draw [bend right=90, looseness=1.50] (6.center) to (7.center);
				\draw [bend left=90, looseness=1.25] (8.center) to (9.center);
				\draw [bend right=90, looseness=1.50] (10.center) to (11.center);
				\draw [style=open, in=120, out=-180] (14.center) to (15.center);
				\draw [style=open, bend left] (12.center) to (13.center);
				\draw [style=open, bend left=90, looseness=1.75] (2.center) to (3.center);
				\draw [style=open, bend right=90, looseness=1.75] (2.center) to (3.center);
				\draw [style=end arrow, in=-90, out=90, looseness=0.75] (20.center) to (19.center);
				\draw [style=end arrow, bend left] (22.center) to (24.center);
				\draw [style=end arrow, bend right, looseness=1.25] (22.center) to (23.center);
				\draw (26.center) to (22.center);
				\draw [style=end arrow] (27.center) to (28.center);
				\draw (29.center) to (30.center);
			\end{pgfonlayer}
		\end{tikzpicture}
		\label{figocsurface}
		\caption{An open-closed surface with parametrized intervals (`open boundary') and parametrized boundary circles (`closed boundary'),
			 both printed in blue. To the open boundary components, we attach the boundary object $F$ while the closed boundary components are labeled with the bulk object $B$.}
		\end{figure}
	\spaceplease
	Finding any consistent system of correlators, let alone all, is a rather difficult mathematical problem.
	For modular fusion categories, i.e.\ in the rational case,
	the construction of correlators has been achieved roughly twenty years ago in a series of articles of Fuchs, Runkel and Schweigert, partially with Fjelstad
	\cite{frs1,frs2,frs3,frs4,ffrs,ffrsunique}, building on additional works of some of these authors with Felder and Fröhlich~\cite{fffs0,fffs}, see also the  paper~\cite{correspondences} building algebraic foundations. In these works, consistent systems of correlators are constructed
	from special symmetric Frobenius algebras in the modular fusion category $\cat{A}$
	 using the Reshetikhin-Turaev topological field theory~\cite{rt1,rt2,turaev} associated to $\cat{A}$.
	 Phrased in a slightly more pointed way, one uses three-dimensional topological field theory to solve two-dimensional rational conformal field theory. This is an instance of the \emph{holographic principle}~\cite{kapustinsaulina}. An early account of the correspondence between rational conformal field theory and three-dimensional topological field theory (however not yet including full conformal field theory) is \cite{kontsevich,witten}.
	 
	 Beyond the semisimple case, the tools of \cite{rt1,rt2,turaev} are not available, and the construction of correlators in this generality is open. The main result of this article is the following:

	\spaceplease
		\begin{reptheorem}{thmmainshort}
		Let $\cat{A}$ be a modular category. Then any special symmetric Frobenius algebra $F\in\cat{A}$ gives rise to a consistent system of open-closed correlators for the modular functor $\FAA$.
	\end{reptheorem}
	
	In the rational case, the construction of \cite{frs1,frs2,frs3,frs4,ffrs,ffrsunique} is recovered.
	
	If one asks for the open-closed correlators a compatibility between bulk and boundary similar to \cite{ffrsunique}, then all consistent systems of open-closed correlators are of this form, so that Theorem~\ref{thmmainshort} is not only a construction, but a classification (Corollary~\ref{corclassification}).
	The solution is completely constructive, and the bulk algebra is Morita invariant in the input algebra $F$, see Corollary~\ref{cormoritainvarianz}.
	
	Let us comment on the relation to the non-semisimple correlator constructions that are available in special cases:
	 The already highly non-trivial task to set up the framework to treat correlators in the non-semisimple case is achieved in \cite{jfcs,jfcslog,fspivotal}.
	 The case in which the Frobenius algebra $F$ in Theorem~\ref{thmmainshort} is the monoidal unit is called the \emph{Cardy case} investigated in \cite{cardycartan,cardycase}.
	 If $\cat{A}$ is given by finite-dimensional modules over a ribbon factorizable Hopf algebra $H$, the Cardy case, including a version that is twisted by a ribbon automorphism of $H$ is solved in \cite{fuchsschweigertstignerhopf,fuchsschweigertstignermcg}.
	 Beyond the Hopf-algebraic case, there is an analysis of the genus zero situation in \cite{jfcs} with a condition for the extension beyond genus zero. Unfortunately, it has not been possible so far to check this condition for the conjectured bulk field candidate in the Cardy case \cite{jfcs,cardycase}. Theorem~\ref{thmmainshort} solves in particular this conjecture.

	 \subsection*{The holographic principle\label{secholographicintro}}
	 In order to prove Theorem~\ref{thmmainshort}, we 
	  construct for a modular category $\cat{A}$ and a special symmetric Frobenius algebra $F\in\cat{A}$ vectors in $\FAA(\Sigma;\bulk(F),\dots,\bulk(F))$ for a surface $\Sigma$, with the bulk object $\bulk(F) \in \bar{\cat{A}}\boxtimes\cat{A}$ that we will discuss in detail in Section~\ref{seccyl}. 
	 To this end, we will first construct vectors in $\FAA(\Sigma;\bzero(F),\dots,\bzero(F))$, with the coend $\bzero(F)=\int^{X \in \cat{A}} X^\vee \boxtimes F \otimes X \in \bar{\cat{A}}\boxtimes \cat{A}$ being a `pre-version' of the bulk object that will have  to be refined afterwards
	 to obtain the actual bulk object. These vectors are constructed first on the open sector using the extension procedures within the framework of the modular microcosm principle~\cite{microcosm} and will then be successively extended building on the methods developed in \cite{cyclic,mwansular,brochierwoike,envas,reflection}.

	 Actually,
	 $\FAA(\Sigma;\bzero(F),\dots,\bzero(F))$ is isomorphic to $\FA(    \bar\Sigma \cupx_{\partial \Sigma} \Sigma;F,\dots,F)$
	 if $\Sigma$ is connected with $n\ge 1$ boundary components, where $\bar \Sigma \cupx_{\partial \Sigma} \Sigma$ is the surface obtained by gluing $\Sigma$ and $\bar{\Sigma}$ together along  cylinders with an additional boundary component on them, see Figure~\ref{figdouble}; this is a version of the familiar \emph{double} of a surface.

	 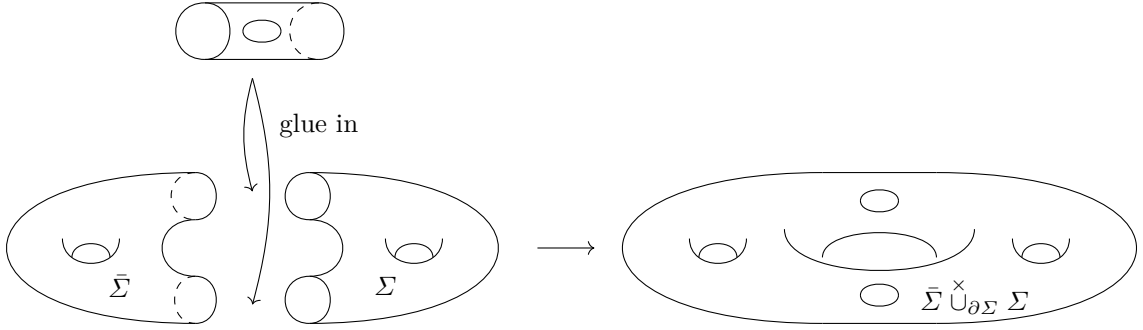
\begin{figure}[h]
	 	\begin{tikzpicture}[scale=0.5]
	 		\begin{pgfonlayer}{nodelayer}
	 			\node [style=none] (0) at (3.75, -0.25) {};
	 			\node [style=none] (1) at (4.75, -0.25) {};
	 			\node [style=none] (2) at (3.5, 0.25) {};
	 			\node [style=none] (3) at (5, 0.25) {};
	 			\node [style=none] (4) at (1.5, 2) {};
	 			\node [style=none] (5) at (2.5, 0.5) {};
	 			\node [style=none] (6) at (1.5, -2) {};
	 			\node [style=none] (7) at (-1.5, 2) {};
	 			\node [style=none] (8) at (-2.5, 0.5) {};
	 			\node [style=none] (9) at (-1.5, -2) {};
	 			\node [style=none] (10) at (-4.75, -0.25) {};
	 			\node [style=none] (11) at (-3.75, -0.25) {};
	 			\node [style=none] (12) at (-5, 0.25) {};
	 			\node [style=none] (13) at (-3.5, 0.25) {};
	 			\node [style=none] (14) at (-1.5, -0.25) {};
	 			\node [style=none] (15) at (1.5, -0.25) {};
	 			\node [style=none] (16) at (-0.5, -1.25) {};
	 			\node [style=none] (17) at (0.5, -1.25) {};
	 			\node [style=none] (19) at (-0.5, 1.25) {};
	 			\node [style=none] (20) at (0.5, 1.25) {};
	 			\node [style=none] (28) at (-18, 2) {};
	 			\node [style=none] (29) at (-18, 0.75) {};
	 			\node [style=none] (30) at (-18, -2) {};
	 			\node [style=none] (31) at (-21.25, -0.25) {};
	 			\node [style=none] (32) at (-20.25, -0.25) {};
	 			\node [style=none] (33) at (-21.5, 0.25) {};
	 			\node [style=none] (34) at (-20, 0.25) {};
	 			\node [style=none] (35) at (-18, -0.75) {};
	 			\node [style=none] (36) at (-15, 2) {};
	 			\node [style=none] (37) at (-15, 0.75) {};
	 			\node [style=none] (38) at (-15, -2) {};
	 			\node [style=none] (39) at (-12.75, -0.25) {};
	 			\node [style=none] (40) at (-11.75, -0.25) {};
	 			\node [style=none] (41) at (-13, 0.25) {};
	 			\node [style=none] (42) at (-11.5, 0.25) {};
	 			\node [style=none] (43) at (-15, -0.75) {};
	 			\node [style=none] (44) at (-14.75, 6.5) {};
	 			\node [style=none] (45) at (-14.75, 5) {};
	 			\node [style=none] (46) at (-17.75, 6.5) {};
	 			\node [style=none] (47) at (-17.75, 5) {};
	 			\node [style=none] (50) at (-16.75, 5.75) {};
	 			\node [style=none] (51) at (-15.75, 5.75) {};
	 			\node [style=none] (52) at (-16.5, 4.5) {};
	 			\node [style=none] (53) at (-16.5, 1.5) {};
	 			\node [style=none] (54) at (-16.5, -1.5) {};
	 			\node [style=none] (55) at (-14.75, 3.25) {glue in};
	 			\node [style=none] (56) at (-9, 0) {};
	 			\node [style=none] (57) at (-7.5, 0) {};
	 			\node [style=none] (58) at (-13, -1) {$\Sigma$};
	 			\node [style=none] (59) at (-20, -1) {$\bar\Sigma$};
	 			\node [style=none] (60) at (2.5, -1.25) {$\bar \Sigma \cupx_{\partial \Sigma} \Sigma$};
	 		\end{pgfonlayer}
	 		\begin{pgfonlayer}{edgelayer}
	 			\draw [bend left=90, looseness=1.25] (0.center) to (1.center);
	 			\draw [bend right=90, looseness=1.50] (2.center) to (3.center);
	 			\draw [bend right=90, looseness=4.50] (6.center) to (4.center);
	 			\draw [in=180, out=-180, looseness=4.50] (7.center) to (9.center);
	 			\draw [bend left=90, looseness=1.25] (10.center) to (11.center);
	 			\draw [bend right=90, looseness=1.50] (12.center) to (13.center);
	 			\draw (7.center) to (4.center);
	 			\draw (9.center) to (6.center);
	 			\draw [bend right=90, looseness=0.75] (8.center) to (5.center);
	 			\draw [bend left=90, looseness=0.75] (14.center) to (15.center);
	 			\draw [bend left=90] (16.center) to (17.center);
	 			\draw [bend right=90] (16.center) to (17.center);
	 			\draw [bend left=90] (19.center) to (20.center);
	 			\draw [bend right=90] (19.center) to (20.center);
	 			\draw [in=180, out=-180, looseness=4.25] (28.center) to (30.center);
	 			\draw [bend left=90, looseness=1.25] (31.center) to (32.center);
	 			\draw [bend right=90, looseness=1.50] (33.center) to (34.center);
	 			\draw [bend right=90, looseness=2.00] (29.center) to (35.center);
	 			\draw [bend left=90, looseness=1.50] (35.center) to (30.center);
	 			\draw [bend left=90, looseness=1.50] (28.center) to (29.center);
	 			\draw [in=0, out=0, looseness=4.25] (36.center) to (38.center);
	 			\draw [bend left=90, looseness=1.25] (39.center) to (40.center);
	 			\draw [bend right=90, looseness=1.50] (41.center) to (42.center);
	 			\draw [bend left=90, looseness=2.00] (37.center) to (43.center);
	 			\draw [bend right=90, looseness=1.75] (36.center) to (37.center);
	 			\draw [bend right=90, looseness=1.75] (43.center) to (38.center);
	 			\draw [bend left=90, looseness=1.50] (43.center) to (38.center);
	 			\draw [bend left=90, looseness=1.50] (36.center) to (37.center);
	 			\draw [style=mydotsblack, bend right=90, looseness=1.75] (28.center) to (29.center);
	 			\draw [style=mydotsblack, bend right=90, looseness=1.50] (35.center) to (30.center);
	 			\draw [bend left=90, looseness=1.50] (44.center) to (45.center);
	 			\draw [bend right=90, looseness=1.75] (46.center) to (47.center);
	 			\draw [bend left=90, looseness=1.50] (46.center) to (47.center);
	 			\draw [style=mydotsblack, bend right=90, looseness=1.75] (44.center) to (45.center);
	 			\draw (46.center) to (44.center);
	 			\draw (47.center) to (45.center);
	 			\draw [bend left=90] (50.center) to (51.center);
	 			\draw [bend right=90] (50.center) to (51.center);
	 			\draw [style=end arrow, bend right=15] (52.center) to (53.center);
	 			\draw [style=end arrow, bend left=15] (52.center) to (54.center);
	 			\draw [style=end arrow] (56.center) to (57.center);
	 		\end{pgfonlayer}
	 	\end{tikzpicture}
	 	
	 	\caption{The construction of $\bar \Sigma \cupx_{\partial \Sigma} \Sigma$ for a genus one surface with two boundary components.}
	 	\label{figdouble}
	 \end{figure}
	 
	 One then proves \begin{align}\FAA(\Sigma;\bzero(F),\dots,\bzero(F))\cong\FA (\bar \Sigma \cupx_{\partial \Sigma} \Sigma;F,\dots,F) \cong \Hom_{\int_\Sigma \cat{A}}(\qss,F^{\boxtimes n}\act \qss) \ , \label{eqnfaasigma}\end{align}
	 where $\int_\Sigma \cat{A}$ is the factorization homology for $\Sigma$ with coefficients in $\cat{A}$, $\qss \in \int_\Sigma \cat{A}$ is the \emph{quantum structure sheaf} from~\cite{bzbj},
	  and $\act : \cat{A}^{\boxtimes n} \boxtimes \int_\Sigma \cat{A}\to\int_\Sigma \cat{A}$ denotes the $\cat{A}^{\boxtimes n}$-action on factorization homology induced by the boundary.
	 The homotopy fixed point structure of the quantum structure sheaf endows $\Hom_{\int_\Sigma \cat{A}}(\qss,F^{\boxtimes n}\act \qss)$ with a mapping class group action that makes~\eqref{eqnfaasigma} equivariant.
	 By \eqref{eqnfaasigma} the correlator 
	  is obtained from  a 
	  map $\qss \to F^{\boxtimes n} \act \qss$ invariant under the mapping class group $\Map(\Sigma)$ or,
	 equivalently, to a $\Map(\Sigma)$-invariant map $F^{\boxtimes n} \to \moduli_\Sigma$, where $\moduli_\Sigma := \END_{\int_\Sigma \cat{A}}(\qss)\in\cat{A}^{\boxtimes n}$ is the inner endomorphism algebra of the quantum structure sheaf carrying again a mapping class group action.
	 This algebra, introduced in this form in 	\cite{bzbj,bzbj2,skeinfin}, generalizes the moduli algebras from
	 \cite{alekseevmoduli,agsmoduli,asmoduli,br1,br2}.

	 If $\Sigma$ is the torus $\mathbb{T}_1^2$ with one boundary component, the correlator comes from a map
	  $F \to \elliptic$ from $F$ to the \emph{elliptic double} $\elliptic$ defined by Brochier-Jordan~\cite{brochierjordane}. This is just a morphism in $\cat{A}$, not necessarily an algebra map.
	 It is invariant under the mapping class group of the torus with one boundary component, the braid group $B_3$ on three strands. 
	 The elliptic double $\elliptic$ is given by two tensor copies $\mathbb{F}\otimes\mathbb{F}$ of the coend algebra $\mathbb{F}=\int^{X\in\cat{A}} X^\vee \otimes X$, with both copies of $\mathbb{F}$ as subalgebras and additional cross relations. 
	 If $\cat{A}$ is given by finite-dimensional modules over a ribbon factorizable Hopf algebra $H$, then $\mathbb{F}$ is the coadjoint representation $H_\text{coadj}^*$, i.e.\ $\elliptic = H_\text{coadj}^* \otimes H_\text{coadj}^*$ as objects. 
	 Precomposition with the unit $I\to F$ of $F$ gives us a morphism $I\to \elliptic$,\label{refelementell} i.e.\ an invariant of the elliptic double which is just an element of the skein algebra of the torus $\mathbb{T}_1^2$ with one boundary component. 
	 It will turn out that this element will later induce the torus partition function
	 that we calculate in Section~\ref{sectorus}.

	 The skein-theoretic perspective on~\eqref{eqnfaasigma} is also quite instructive:
	 Consider the cylinder $\Sigma \times [0,1]$.
	 To this three-dimensional manifold, we apply the functor $\SkA(-)$ for the skein module; since we are working in the non-semisimple setting, this needs to be the cocompleted version $\SkA$~\cite{brownhaioun,mwskein} of the 
	 admissible skein module construction $\skA$ of Costantino-Geer-Patureau-Mirand~\cite{asm}. 
	 For the $\Sigma \times \{1\}$-copy, we insert $F$ near the boundary, but not for the $\Sigma \times \{0\}$-copy, see Figure~\ref{fig3d}.

	 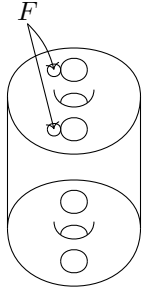
\begin{figure}[h]
	 	\centering	\begin{tikzpicture}[scale=0.35]
	 		\begin{pgfonlayer}{nodelayer}
	 			\node [style=none] (0) at (-2.5, -2) {};
	 			\node [style=none] (1) at (2.5, -2) {};
	 			\node [style=none] (2) at (-0.5, -2.25) {};
	 			\node [style=none] (3) at (0.5, -2.25) {};
	 			\node [style=none] (4) at (-0.75, -1.75) {};
	 			\node [style=none] (5) at (0.75, -1.75) {};
	 			\node [style=none] (6) at (-2.5, 3) {};
	 			\node [style=none] (7) at (2.5, 3) {};
	 			\node [style=none] (8) at (-0.5, 2.75) {};
	 			\node [style=none] (9) at (0.5, 2.75) {};
	 			\node [style=none] (10) at (-0.75, 3.25) {};
	 			\node [style=none] (11) at (0.75, 3.25) {};
	 			\node [style=none] (12) at (-0.5, -1) {};
	 			\node [style=none] (13) at (0.5, -1) {};
	 			\node [style=none] (14) at (-1.75, 6.25) {$F$};
	 			\node [style=none] (15) at (-0.5, -3.25) {};
	 			\node [style=none] (16) at (0.5, -3.25) {};
	 			\node [style=none] (18) at (-0.5, 4) {};
	 			\node [style=none] (19) at (0.5, 4) {};
	 			\node [style=none] (21) at (-0.5, 1.75) {};
	 			\node [style=none] (22) at (0.5, 1.75) {};
	 			\node [style=none] (27) at (0, -1.5) {};
	 			\node [style=none] (28) at (-1, 1.75) {};
	 			\node [style=none] (29) at (-0.5, 1.75) {};
	 			\node [style=none] (30) at (-1, 4) {};
	 			\node [style=none] (31) at (-0.5, 4) {};
	 			\node [style=none] (32) at (-1.75, 5.75) {};
	 			\node [style=none] (33) at (-0.75, 4) {};
	 			\node [style=none] (34) at (-0.75, 1.75) {};
	 		\end{pgfonlayer}
	 		\begin{pgfonlayer}{edgelayer}
	 			\draw [bend right=90, looseness=1.50] (0.center) to (1.center);
	 			\draw [bend left=90, looseness=1.25] (0.center) to (1.center);
	 			\draw [bend left=90, looseness=1.25] (2.center) to (3.center);
	 			\draw [bend right=90, looseness=1.50] (4.center) to (5.center);
	 			\draw [bend right=90, looseness=1.50] (6.center) to (7.center);
	 			\draw [bend left=90, looseness=1.25] (6.center) to (7.center);
	 			\draw [bend left=90, looseness=1.25] (8.center) to (9.center);
	 			\draw [bend right=90, looseness=1.50] (10.center) to (11.center);
	 			\draw (6.center) to (0.center);
	 			\draw (7.center) to (1.center);
	 			\draw [bend left=90, looseness=1.50] (12.center) to (13.center);
	 			\draw [bend right=90, looseness=1.50] (12.center) to (13.center);
	 			\draw [bend left=90, looseness=1.50] (15.center) to (16.center);
	 			\draw [bend right=90, looseness=1.50] (15.center) to (16.center);
	 			\draw [bend left=90, looseness=1.50] (18.center) to (19.center);
	 			\draw [bend right=90, looseness=1.50] (18.center) to (19.center);
	 			\draw [bend left=90, looseness=1.50] (21.center) to (22.center);
	 			\draw [bend right=90, looseness=1.50] (21.center) to (22.center);
	 			\draw [bend left=90, looseness=1.50] (28.center) to (29.center);
	 			\draw [bend right=90, looseness=1.75] (28.center) to (29.center);
	 			\draw [bend left=90, looseness=1.50] (30.center) to (31.center);
	 			\draw [bend right=90, looseness=1.75] (30.center) to (31.center);
	 			\draw [style=end arrow] (32.center) to (34.center);
	 			\draw [style=end arrow, bend left=15] (32.center) to (33.center);
	 		\end{pgfonlayer}
	 	\end{tikzpicture}
	 	
	 	\caption{The cylinder over $\Sigma$, with the boundary labels only kept in the upper copy.}
	 	\label{fig3d}
	 \end{figure}

	 Therefore, \eqref{eqnfaasigma} is telling us that the space of conformal blocks for the above-mentioned version of the double of the surface with $F$ as boundary label is canonically isomorphic to the skein module for the manifold in Figure~\ref{fig3d}:
	 \begin{align}\label{eqnholprinintro}
	 	\FA\left(\!\!\!\!\!\!\!\!\!\begin{array}{c}
	 	\begin{tikzpicture}[scale=0.35]
	 		\begin{pgfonlayer}{nodelayer}
	 			\node [style=none] (0) at (3.75, -0.25) {};
	 			\node [style=none] (1) at (4.75, -0.25) {};
	 			\node [style=none] (2) at (3.5, 0.25) {};
	 			\node [style=none] (3) at (5, 0.25) {};
	 			\node [style=none] (4) at (1.5, 2) {};
	 			\node [style=none] (5) at (2.5, 0.5) {};
	 			\node [style=none] (6) at (1.5, -2) {};
	 			\node [style=none] (7) at (-1.5, 2) {};
	 			\node [style=none] (8) at (-2.5, 0.5) {};
	 			\node [style=none] (9) at (-1.5, -2) {};
	 			\node [style=none] (10) at (-4.75, -0.25) {};
	 			\node [style=none] (11) at (-3.75, -0.25) {};
	 			\node [style=none] (12) at (-5, 0.25) {};
	 			\node [style=none] (13) at (-3.5, 0.25) {};
	 			\node [style=none] (14) at (-1.5, -0.25) {};
	 			\node [style=none] (15) at (1.5, -0.25) {};
	 			\node [style=none] (16) at (-0.5, -1.25) {};
	 			\node [style=none] (17) at (0.5, -1.25) {};
	 			\node [style=none] (18) at (0, -1.25) {$F$};
	 			\node [style=none] (19) at (-0.5, 1.25) {};
	 			\node [style=none] (20) at (0.5, 1.25) {};
	 			\node [style=none] (21) at (0, 1.25) {$F$};
	 		\end{pgfonlayer}
	 		\begin{pgfonlayer}{edgelayer}
	 			\draw [bend left=90, looseness=1.25] (0.center) to (1.center);
	 			\draw [bend right=90, looseness=1.50] (2.center) to (3.center);
	 			\draw [bend right=90, looseness=4.50] (6.center) to (4.center);
	 			\draw [in=180, out=-180, looseness=4.50] (7.center) to (9.center);
	 			\draw [bend left=90, looseness=1.25] (10.center) to (11.center);
	 			\draw [bend right=90, looseness=1.50] (12.center) to (13.center);
	 			\draw (7.center) to (4.center);
	 			\draw (9.center) to (6.center);
	 			\draw [bend right=90, looseness=0.75] (8.center) to (5.center);
	 			\draw [bend left=90, looseness=0.75] (14.center) to (15.center);
	 			\draw [bend left=90, looseness=1.50] (16.center) to (17.center);
	 			\draw [bend right=90, looseness=1.50] (16.center) to (17.center);
	 			\draw [bend left=90, looseness=1.50] (19.center) to (20.center);
	 			\draw [bend right=90, looseness=1.50] (19.center) to (20.center);
	 		\end{pgfonlayer}
	 	\end{tikzpicture}
	 	\end{array} \!\!\!\!\!\!\!\!\!\right) 
	 	\cong
	 	\catf{Sk}_\cat{A} \left(\begin{array}{c} \begin{tikzpicture}[scale=0.35]
	 		\begin{pgfonlayer}{nodelayer}
	 			\node [style=none] (0) at (-2.5, -2) {};
	 			\node [style=none] (1) at (2.5, -2) {};
	 			\node [style=none] (2) at (-0.5, -2.25) {};
	 			\node [style=none] (3) at (0.5, -2.25) {};
	 			\node [style=none] (4) at (-0.75, -1.75) {};
	 			\node [style=none] (5) at (0.75, -1.75) {};
	 			\node [style=none] (6) at (-2.5, 3) {};
	 			\node [style=none] (7) at (2.5, 3) {};
	 			\node [style=none] (8) at (-0.5, 2.75) {};
	 			\node [style=none] (9) at (0.5, 2.75) {};
	 			\node [style=none] (10) at (-0.75, 3.25) {};
	 			\node [style=none] (11) at (0.75, 3.25) {};
	 			\node [style=none] (12) at (-0.5, -1) {};
	 			\node [style=none] (13) at (0.5, -1) {};
	 			\node [style=none] (14) at (-1.75, 6.25) {$F$};
	 			\node [style=none] (15) at (-0.5, -3.25) {};
	 			\node [style=none] (16) at (0.5, -3.25) {};
	 			\node [style=none] (18) at (-0.5, 4) {};
	 			\node [style=none] (19) at (0.5, 4) {};
	 			\node [style=none] (21) at (-0.5, 1.75) {};
	 			\node [style=none] (22) at (0.5, 1.75) {};
	 			\node [style=none] (27) at (0, -1.5) {};
	 			\node [style=none] (28) at (-1, 1.75) {};
	 			\node [style=none] (29) at (-0.5, 1.75) {};
	 			\node [style=none] (30) at (-1, 4) {};
	 			\node [style=none] (31) at (-0.5, 4) {};
	 			\node [style=none] (32) at (-1.75, 5.75) {};
	 			\node [style=none] (33) at (-0.75, 4) {};
	 			\node [style=none] (34) at (-0.75, 1.75) {};
	 		\end{pgfonlayer}
	 		\begin{pgfonlayer}{edgelayer}
	 			\draw [bend right=90, looseness=1.50] (0.center) to (1.center);
	 			\draw [bend left=90, looseness=1.25] (0.center) to (1.center);
	 			\draw [bend left=90, looseness=1.25] (2.center) to (3.center);
	 			\draw [bend right=90, looseness=1.50] (4.center) to (5.center);
	 			\draw [bend right=90, looseness=1.50] (6.center) to (7.center);
	 			\draw [bend left=90, looseness=1.25] (6.center) to (7.center);
	 			\draw [bend left=90, looseness=1.25] (8.center) to (9.center);
	 			\draw [bend right=90, looseness=1.50] (10.center) to (11.center);
	 			\draw (6.center) to (0.center);
	 			\draw (7.center) to (1.center);
	 			\draw [bend left=90, looseness=1.50] (12.center) to (13.center);
	 			\draw [bend right=90, looseness=1.50] (12.center) to (13.center);
	 			\draw [bend left=90, looseness=1.50] (15.center) to (16.center);
	 			\draw [bend right=90, looseness=1.50] (15.center) to (16.center);
	 			\draw [bend left=90, looseness=1.50] (18.center) to (19.center);
	 			\draw [bend right=90, looseness=1.50] (18.center) to (19.center);
	 			\draw [bend left=90, looseness=1.50] (21.center) to (22.center);
	 			\draw [bend right=90, looseness=1.50] (21.center) to (22.center);
	 			\draw [bend left=90, looseness=1.50] (28.center) to (29.center);
	 			\draw [bend right=90, looseness=1.75] (28.center) to (29.center);
	 			\draw [bend left=90, looseness=1.50] (30.center) to (31.center);
	 			\draw [bend right=90, looseness=1.75] (30.center) to (31.center);
	 			\draw [style=end arrow] (32.center) to (34.center);
	 			\draw [style=end arrow, bend left=15] (32.center) to (33.center);
	 		\end{pgfonlayer}
	 	\end{tikzpicture}\end{array} \right)
	 \end{align}

	 The vector that later induces the correlator is equivalently a vector in the skein module of the cylinder and therefore has a three-dimensional origin. 
	 As in \cite{frs1,frs2,frs3,frs4,ffrs,ffrsunique}, this yields a \emph{holographic interpretation of the correlators}.
	 We should however warn the reader that the construction in this paper differs, even in the semisimple case, substantially 
	  from the known semisimple one using the Reshetikhin-Turaev topological field theory.
	 The holographic principle is rather an a posteriori interpretation; for the mere construction, it is not needed.

	\subsection*{Torus partition function} The torus partition function is extracted by carefully generalizing existing methods in Section~\ref{sectorus}:
	 In the semisimple case, one may use the isomorphism
	 $\FA(\bar\Sigma) \cong \FA(\Sigma)^*$ 
	 originating from the description via the three-dimensional topological field theory
	  to see the correlator for the torus $\mathbb{T}^2$ as an endomorphism of $\FA(\mathbb{T}^2)$ whose matrix elements $\mathcal{Z}^F_{i,j}$ with respect to the canonical basis of $\FA(\mathbb{T}^2)$ indexed by the simple objects of $\cat{A}$ are the coefficients of the partition function.
	  These are non-negative integers by the results of~\cite{frs1}, a fact that is extremely crucial for the physical interpretation of these quantities.
	   
	 In the non-semisimple case, the situation is more complicated, but we use the factorization homology description of the correlators to see 
	 that the correlator for the torus gives rise to an element in the skein algebra of the torus 
	 and therefore to an endomorphism $\mathcal{Z}^F$ of $\FA(\Sigma)$ intertwining the action of the modular group. We propose this endomorphism
	 as the partition function.
	 We may then extract a coefficient $\mathcal{Z}_{P,Q}^F$ of the partition function for two \emph{projective} objects $P,Q \in \Proj \cat{A}$ roughly as follows:
	 By definition $\mathcal{Z}^F$
	  can act on  a 	vector inside the space of conformal blocks for the torus that we can represent by skeins labeled by endomorphisms of projective objects in $\mathbb{S}^1 \times \mathbb{D}^2$. We denote this action by $\mathcal{Z}^F \actcor -$. After acting on the identity of $Q \in \Proj \cat{A}$ and taking the union with the identity of $P^\vee \in \Proj \cat{A}$ in $\mathbb{S}^1 \times \mathbb{D}^2$ with opposite orientation, we obtain a link in $(\overline{\mathbb{S}^1 \times \mathbb{D}^2}) \cup_{\mathbb{T}^2} (\mathbb{S}^1 \times \mathbb{D}^2)=\mathbb{S}^2 \times \mathbb{S}^1$ that we turn into a numerical invariant by generalizing the standard methods of Reshetikhin-Turaev with the help of the theory of admissible skeins~\cite{asm,brownhaioun,mwskein}.
	 On a technical level, the coefficients of the partition function
	  will be realized as Hattori-Stallings traces on factorization homology of the two-dimensional sphere.
	 A symbolic overview is given in Figure~\ref{figpartition},
	 see Section~\ref{sectorus} for the  details.
	 \begin{figure}[h]
	 	\centering
	 	\begin{tikzpicture}[scale=0.5]
	 		\begin{pgfonlayer}{nodelayer}
	 			\node [style=none] (0) at (-13.25, 6.5) {};
	 			\node [style=none] (1) at (-11.25, 6.5) {};
	 			\node [style=none] (2) at (-12.75, 5.75) {};
	 			\node [style=none] (3) at (-11.75, 5.75) {};
	 			\node [style=none] (4) at (-15.25, 6) {};
	 			\node [style=none] (5) at (-9.25, 6) {};
	 			\node [style=none] (6) at (-14.25, 6) {};
	 			\node [style=none] (7) at (-10.25, 6) {};
	 			\node [style=none] (8) at (-12.25, 5) {$P$};
	 			\node [style=none] (9) at (-1, 6.5) {};
	 			\node [style=none] (10) at (1, 6.5) {};
	 			\node [style=none] (11) at (-0.5, 5.75) {};
	 			\node [style=none] (12) at (0.5, 5.75) {};
	 			\node [style=none] (13) at (-3, 6) {};
	 			\node [style=none] (14) at (3, 6) {};
	 			\node [style=none] (15) at (-2, 6) {};
	 			\node [style=none] (16) at (2, 6) {};
	 			\node [style=none] (17) at (0, 5) {$Q$};
	 			\node [style=none] (18) at (-4.25, 6) {$\mathcal{Z}^F \actcor$};
	 			\node [style=none] (19) at (-7, 4.5) {$\mathbb{T}^2$};
	 			\node [style=none] (20) at (-8, 6.5) {};
	 			\node [style=none] (21) at (-6, 6.5) {};
	 			\node [style=none] (25) at (0, 3) {$\mathbb{S}^1 \times \mathbb{D}^2$};
	 			\node [style=none] (26) at (-12.25, 3) {$\overline{\mathbb{S}^1 \times \mathbb{D}^2}$};
	 			\node [style=none] (27) at (4, 6) {};
	 			\node [style=none] (28) at (7, 6) {};
	 			\node [style=none] (29) at (10.5, 3) {link in $\mathbb{S}^2 \times \mathbb{S}^1$};
	 			\node [style=none] (30) at (10.5, 2) {};
	 			\node [style=none] (31) at (10.5, 0.25) {};
	 			\node [style=none] (32) at (7, -0.5) {coefficient of partition function $\mathcal{Z}_{P,Q}^F \in \mathbb{Z}_{\ge 0}$};
	 			\node [style=none] (33) at (5.25, 1.25) {turn into numerical invariant};
	 			\node [style=none] (36) at (10.75, 9.75) {};
	 			\node [style=none] (37) at (9.75, 6.5) {};
	 			\node [style=none] (41) at (9.75, 9.75) {};
	 			\node [style=none] (42) at (9.25, 7) {$F$};
	 			\node [style=none] (43) at (11.25, 7) {$Q$};
	 			\node [style=none] (44) at (7.75, 6) {};
	 			\node [style=none] (45) at (12, 6) {};
	 			\node [style=none] (46) at (10.75, 6.5) {};
	 			\node [style=none] (47) at (9.75, 4.25) {};
	 			\node [style=none] (48) at (10, 4.75) {$\mathbb{S}^2$};
	 			\node [style=none] (49) at (8.5, 6.5) {};
	 			\node [style=none] (50) at (8.5, 11.5) {};
	 			\node [style=none] (51) at (7.75, 8.25) {$P$};
	 			\node [style=none] (53) at (10.75, 8.5) {};
	 			\node [style=none] (54) at (10.75, 9.75) {};
	 			\node [style=none] (55) at (10.75, 11.5) {};
	 			\node [style=none] (56) at (9.75, 8.5) {};
	 			\node [style=none] (57) at (9.75, 8.75) {};
	 			\node [style=none] (59) at (9.75, 10.75) {};
	 			\node [style=none] (63) at (9.75, 11.5) {};
	 			\node [style=none] (64) at (9.75, 9.75) {};
	 			\node [style=none] (65) at (9.75, 9.75) {};
	 			\node [style=none] (66) at (11.25, 9.75) {};
	 		\end{pgfonlayer}
	 		\begin{pgfonlayer}{edgelayer}
	 			\draw [bend right=90, looseness=1.50] (0.center) to (1.center);
	 			\draw [bend left=45, looseness=1.25] (2.center) to (3.center);
	 			\draw [bend left=90, looseness=1.25] (4.center) to (5.center);
	 			\draw [bend right=90, looseness=1.25] (4.center) to (5.center);
	 			\draw [style=red, bend left=90, looseness=1.25] (6.center) to (7.center);
	 			\draw [style=red, bend right=90, looseness=1.25] (6.center) to (7.center);
	 			\draw [bend right=90, looseness=1.50] (9.center) to (10.center);
	 			\draw [bend left=45, looseness=1.25] (11.center) to (12.center);
	 			\draw [bend left=90, looseness=1.25] (13.center) to (14.center);
	 			\draw [bend right=90, looseness=1.25] (13.center) to (14.center);
	 			\draw [style=red, bend left=90, looseness=1.25] (15.center) to (16.center);
	 			\draw [style=red, bend right=90, looseness=1.25] (15.center) to (16.center);
	 			\draw [bend right=90, looseness=2.50] (20.center) to (21.center);
	 			\draw [style=end arrow] (27.center) to (28.center);
	 			\draw [style=end arrow] (30.center) to (31.center);
	 			\draw [bend left=90, looseness=1.75] (44.center) to (45.center);
	 			\draw [bend right=90, looseness=1.75] (44.center) to (45.center);
	 			\draw [style=mydots, bend right=15] (44.center) to (45.center);
	 			\draw [style=mover] (49.center) to (50.center);
	 			\draw (56.center) to (57.center);
	 			\draw [style=open] (57.center) to (56.center);
	 			\draw [style=open, bend left=90, looseness=1.25] (57.center) to (59.center);
	 			\draw [style=open] (57.center) to (65.center);
	 			\draw (53.center) to (54.center);
	 			\draw [style=bover, in=-90, out=0] (57.center) to (66.center);
	 			\draw [style=open, in=0, out=90] (66.center) to (59.center);
	 			\draw [style=over] (54.center) to (55.center);
	 			\draw [style=open] (56.center) to (65.center);
	 			\draw [style=bover] (65.center) to (63.center);
	 			\draw [style=bover] (37.center) to (56.center);
	 			\draw [style=over] (46.center) to (53.center);
	 		\end{pgfonlayer}
	 	\end{tikzpicture}
	 	\caption{On the extraction of the torus partition function, see Section~\ref{sectorus} for further details.}\label{figpartition}
	 	\end{figure}
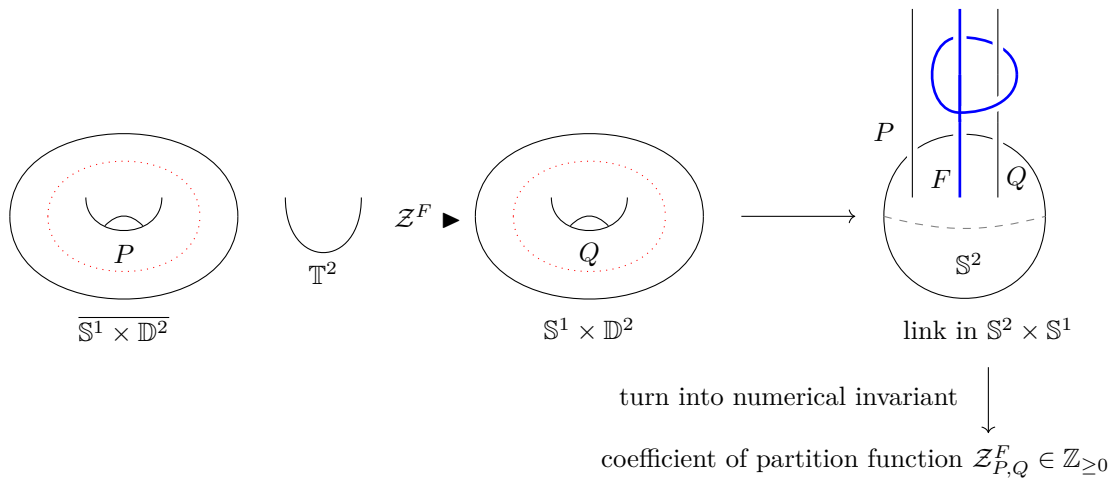

	 We then prove that these coefficients of the torus partition function
	  are again non-negative integers (in the following statement, we denote the morphism spaces of $\cat{A}$ by $\cat{A}(-,-)$):
	 
	 \begin{reptheorem}{thmpartition} Let $F\in \cat{A}$ be a special Frobenius algebra in a modular category $\cat{A}$. 
	 	The coefficient $\mathcal{Z}_{P,Q}^F$ 
	 	of the
	 	torus partition function
	 	for $P,Q \in \Proj \cat{A}$ is a non-negative integer, and $\mathcal{Z}_{P,Q}^F \le \dim\, \cat{A}(P^\vee,F\otimes P)$. In the Cardy case $F=I$, the torus partition function for $P,Q \in \Proj \cat{A}$ is the entry $C_{P^\vee,Q}=\dim \, \cat{A}(P^\vee,Q)$ of the Cartan matrix of $\cat{A}$.   
	 \end{reptheorem}
	 
	 In the Hopf-algebraic Cardy case, the 
	  fact that the partition function leads to the Cartan matrix was proven in~\cite{cardycartan}; this partition function is known as the \emph{Cardy-Cartan modular invariant}. The non-semisimple generalization of the partition function in \cite{cardycartan} follows a priori a different proposed definition. Theorem~\ref{thmpartition} asserts in particular that the notion of partition function used in this article is in line with \cite{cardycartan}. 
	 
We should however warn the reader that the identities at 
projective objects do not, as in the rational case, 
form a basis for the space of conformal blocks of the torus. 
As an alternative, the factorization homology description of correlators suggests directly a
 tangible algebraic datum capturing the full information of the torus partition function that we refer to as \emph{elliptic class function} (Section~\ref{secclassfunctions} \& \ref{secellicf}).
It amounts to a morphism $\zeta_{1,1}^F:\mathbb{A}^{\otimes 2}\to I$
(namely the one dual to the map $I\to \elliptic$ mentioned on page~\pageref{refelementell}), where  $\mathbb{A}=\int_{X\in\cat{A}} X^\vee \otimes X$ is the canonical end. This map is invariant under the $B_3$-action on $\mathbb{A}^{\otimes 2}$ from \cite{brochierjordane}.
If $\cat{A}$ is given by finite-dimensional modules over a ribbon factorizable Hopf algebra $H$, the elliptic class function is a $H$-module map $H_\text{adj} \otimes H_\text{adj}\to k$. We express it via a closed formula in Corollary~\ref{corellcfhopf}.	

	\spaceplease
	\subsection*{Local operators and Hochschild cohomology the category of boundary conditions\label{secbulkfieldintro}}
	For a special symmetric Frobenius algebra $F \in \cat{A}$ in a modular category $\cat{A}$, consider the associated consistent system of correlators. After genus zero restriction, we must obtain a symmetric commutative Frobenius algebra 
	  in $\bar{\cat{A}}\boxtimes\cat{A}$~\cite{jfcs}.
	  The image under the canonical equivalence $\bar{\cat{A}}\boxtimes\cat{A}\to Z(\cat{A})$~\cite{eno-d}
	  to the Drinfeld center $Z(\cat{A})$
	  of $\cat{A}$
	   is
	  a symmetric commutative Frobenius algebra $\Fhb \in Z(\cat{A})$, and its image
	  under the forgetful functor $U:Z(\cat{A})\to\cat{A}$ is an algebra in $\cat{A}$ whose homotopy invariants, i.e.\ the right derived hom space \begin{align} \LO(\cat{A},F):=\mathbb{R}\Hom_ \cat{A}(I,U\Fhb) \ , 
	  	\end{align} give us a differential graded algebra. This algebra should be interpreted as the \emph{algebra of local operators} for the full conformal field theory $(\cat{A},F)$
	  	because the product on $\Fhb$ describes the operator product expansion. 
	  	
	From the consistent system of correlators for $F$, the algebra $\LO(\cat{A},F)$ inherits the structure of a framed $E_2$-algebra (Proposition~\ref{prophi}), which means that its cohomology is a Batalin-Vilkovisky algebra. Such a structure should be expected by the principles of
	\cite[Section~1.2]{bbzbdn}, but in the context of this article, this structure is far from automatic. We obtain it using the methods of \cite{homotopyinvariants}. 
	The main result about $\LO(\cat{A},F)$ is the following:
	
	\begin{reptheorem}{thmhhfmod}
		Let $F$ be a special symmetric Frobenius algebra in a modular category $\cat{A}$.
		Then there is an equivalence 
		\begin{align} \LO(\cat{A};F) \simeq CH^*(\Fmod)
		\end{align} of differential graded framed $E_2$-algebras
		between the local operators for the conformal field theory $\cat{A}$ with correlator built from $F$ and the Hochschild cochains of the category $\Fmod$ of $F$-modules in $\cat{A}$, equipped with the framed $E_2$-structure coming from the cyclic Deligne conjecture.
	\end{reptheorem}

	Seeing that $\LO(\cat{A};F)$ and $ CH^*(\Fmod)$
	 agree as framed $E_2$-algebras is relatively involved, especially because the multiplicative structure on $\LO(\cat{A};F)$ uses the ribbon structure on $\cat{A}$ while $CH^*(\Fmod)$ just sees the linear category $\Fmod$. As explained in \cite{verlindesw}, transforming an $E_2$-structure into Deligne's $E_2$-structure on Hochschild cohomology is an analogue of a \emph{diagonalization} of the multiplication.
	 Such a diagonalization should be expected because of the observations in \cite[Section~3.3]{fspivotal} pertaining to the underived case. 
	 
	The strategy for the proof of Theorem~\ref{thmhhfmod} is the following: Clearly, the symmetric commutative Frobenius algebra $\Fhb \in Z(\cat{A})$ should be the center of $F$ for any notion of center that  deserves its name, but one needs relatively sophisticated algebraic methods to make a precise statement. One needs to use first that $\Fmod$ comes with the structure of a \emph{pivotal module category} over $\cat{A}$ in the sense of \cite{schaumannpiv,relserre}. Informed by this observation, the notion of center suggested in \cite{internal} is the end $\int_{M\in\Fmod} \HOM_F(M,M) \in Z(\cat{A})$ over internal endomorphism algebras. In the same article and additionally~\cite{fspivotal}, $\Fmod$ is proposed as \emph{category of boundary conditions} for a conformal field theory with monodromy data $\cat{A}$, thereby making $\int_{M\in\Fmod} \HOM_F(M,M) \in Z(\cat{A})$ a bulk field candidate. 
	The key step for the proof of Theorem~\ref{thmhhfmod} is to prove that $\Fhb \in Z(\cat{A})$ is isomorphic to 
	$\int_{M\in\Fmod} \HOM_F(M,M) \in Z(\cat{A})$ as symmetric Frobenius algebra. 
	As expected as this is (again: to the circle we \emph{have} to assign some sort of center --- there is no way around it), the proof is involved and requires a rather explicit calculation of the bulk algebra in terms of relations (Section~\ref{secrelbulk}).
	The comparison to \cite{internal} might also be a bit surprising because the correlator construction in this article does not start with the pivotal module category $\Fmod$ as an input datum. 
	In other words, \cite[Assumption~2]{fspivotal} is not one of our axioms, but it still satisfied.
	Another consequence is that   the correlators built from $F$, at least in hindsight, extend to higher genus
	the genus zero correlators that one would be able to build from $\int_{M\in\Fmod} \HOM_F(M,M)$ via the commutative symmetric Frobenius structure.

	There is also a physical interpretation:
	Since $\Fmod$ is the category of boundary conditions, Theorem~\ref{thmhhfmod} tells us that the algebra of local operators agrees with the Hochschild cochains of the category of boundary conditions, as differential graded framed $E_2$-algebra (or Batalin-Vilkovisky algebra once we pass to cohomology), and should therefore be described as deformations of the category of boundary conditions. 
	This is a conjecture made  by Kapustin-Rozansky in \cite{kapustinrozansky}. Theorem~\ref{thmhhfmod} makes precise and proves this conjecture in the context of finite rigid logarithmic conformal field theory.

		\vspace*{0.2cm}\textsc{Acknowledgments.} 
		I would like to thank Christoph Schweigert for countless explanations on correlators and extremely insightful comments on this project.
			Moreover, I would like to thank
			Jorge Becerra,
	Adrien Brochier, Lukas Müller, Yang Yang and Deniz Yeral
		 for  helpful discussions related to this article.
		LW  gratefully acknowledges support
		by the ANR project CPJ n°ANR-22-CPJ1-0001-01 at the Institut de Mathématiques de Bourgogne (IMB).
		The IMB receives support from the EIPHI Graduate School (ANR-17-EURE-0002).\enlargethispage*{0.5cm}

	\needspace{15\baselineskip}
\part{The construction and classification of correlators}
The first part of this article contains the construction and classification of the consistent systems of correlators for the conformal field theory whose monodromy data is a modular category.
The emphasis is on brevity; it does not contain a lot of applications or explanations working out the connection to previously known special cases. Instead, this is treated in the second part.

\section{Preliminaries}
In this section, we collect some general tools needed throughout the article.
Other, more specialized tools will be recalled later directly in the place where they are needed.

\subsection{Modular operads and their algebras\label{secmodop}}
First we give a brief reminder on modular operads and their algebras. The description focuses on the idea without the background on how these ideas are technically implemented. The reader who prefers a slightly longer, dense account that gives more insight on technicalities is referred to \cite[Section~2]{reflection}.

According to the definitions in \cite{gkmod,costello,cyclic},
 a \emph{category-valued modular operad} $\cat{O}$ consists of categories $\cat{O}(T)$, where $T$ is a finite disjoint union of corollas, i.e.\ graphs with one vertex and finitely many legs attached to it (the corolla can have zero legs).
Whenever we have two such disjoint unions of corollas $S$ and $T$, the modular operads comes with canonical equivalences $\cat{O}(S)\times\cat{O}(T)\simeq \cat{O}(S \sqcup T)$. One refers to $\cat{O}(T)$ as the object of operations of arity $n-1$ or total arity $n$ if $T$ has $n$ legs.
The operadic composition takes the following form:
Suppose that $\Gamma$ is a finite graph. Denote by $\nu(\Gamma)$ the disjoint union of corollas obtained by cutting $\Gamma$ at all internal edges and by $\pi_0(\Gamma)$ the disjoint union of corollas obtained by contracting all internal edges, see Figure~\ref{figgraphsexample}. 
Suppose that we are now given two disjoint unions of corollas $T$ and $T'$ and isomorphisms $\nu(\Gamma)\cong T$ and $\pi_0(\Gamma)\cong T'$ of graphs, then $\cat{O}$ provides us with a functor $\cat{O}(T)\to \cat{O}(T')$ describing the operadic composition. 	Moreover, $\cat{O}$ comes equipped with an operation of total arity two behaving neutrally with respect to operadic composition, the so-called \emph{operadic identity}.
All of this data has to be complemented with a lot of coherence isomorphisms subject to even more conditions. This is developed in \cite[Section~2.1]{cyclic} based on the description of modular operads in \cite{costello}.

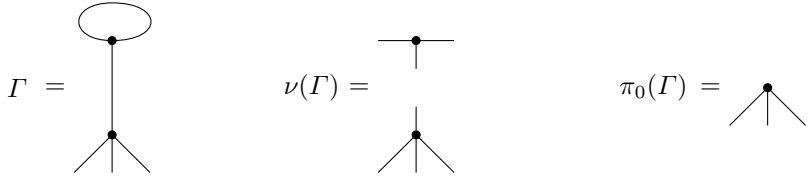
\begin{figure}[h]
	\centering	\begin{tikzpicture}[scale=0.5]
		\begin{pgfonlayer}{nodelayer}
			\node [style=none] (0) at (-7, 2.5) {};
			\node [style=none] (1) at (-7, 0) {};
			\node [style=none] (2) at (-8, -1) {};
			\node [style=none] (3) at (-7, -1) {};
			\node [style=none] (4) at (-6, -1) {};
			\node [style=none] (5) at (-7, 3.5) {};
			\node [style=none] (6) at (-8.5, 1.25) {$=$};
			\node [style=none] (7) at (-9.5, 1.25) {$\Gamma$};
			\node [style=none] (8) at (1, 2.5) {};
			\node [style=none] (9) at (1, 0) {};
			\node [style=none] (10) at (0, -1) {};
			\node [style=none] (11) at (1, -1) {};
			\node [style=none] (12) at (2, -1) {};
			\node [style=none] (14) at (-0.5, 1.25) {$=$};
			\node [style=none] (15) at (-1.75, 1.25) {$\nu(\Gamma)$};
			\node [style=none] (16) at (0, 2.5) {};
			\node [style=black dot] (17) at (-7, 0) {};
			\node [style=black dot] (18) at (-7, 2.5) {};
			\node [style=black dot] (19) at (1, 2.5) {};
			\node [style=black dot] (20) at (1, 0) {};
			\node [style=none] (21) at (2, 2.5) {};
			\node [style=none] (22) at (1, 1.75) {};
			\node [style=none] (23) at (1, 0.75) {};
			\node [style=none] (24) at (10.25, 1.25) {};
			\node [style=none] (25) at (9.25, 0.25) {};
			\node [style=none] (26) at (10.25, 0.25) {};
			\node [style=none] (27) at (11.25, 0.25) {};
			\node [style=black dot] (28) at (10.25, 1.25) {};
			\node [style=none] (29) at (8.75, 1.25) {$=$};
			\node [style=none] (30) at (7.25, 1.25) {$\pi_0(\Gamma)$};
		\end{pgfonlayer}
		\begin{pgfonlayer}{edgelayer}
			\draw (0.center) to (1.center);
			\draw (1.center) to (2.center);
			\draw (1.center) to (3.center);
			\draw (1.center) to (4.center);
			\draw [bend left=90, looseness=3.00] (0.center) to (5.center);
			\draw [bend left=90, looseness=3.50] (5.center) to (0.center);
			\draw (9.center) to (10.center);
			\draw (9.center) to (11.center);
			\draw (9.center) to (12.center);
			\draw (16.center) to (21.center);
			\draw (19) to (22.center);
			\draw (23.center) to (20);
			\draw (24.center) to (25.center);
			\draw (24.center) to (26.center);
			\draw (24.center) to (27.center);
		\end{pgfonlayer}
	\end{tikzpicture}
	
	\caption{For the definition of $\nu(\Gamma)$ and $\pi_0(\Gamma)$.
		The operadic composition for this example of $\Gamma$ describes the composition of operations of total arity three and four to an operation of total arity three.}
	\label{figgraphsexample}
\end{figure}

Given a category-valued modular operad $\cat{O}$, we can consider a modular $\cat{O}$-algebra inside a symmetric monoidal bicategory.
The case of interest in this article is the symmetric monoidal bicategory $\Lexf$ of finite (linear) categories~\cite{etingofostrik,fss} over an algebraically closed field $k$ that we fix throughout (linear abelian categories with finite-dimensional morphisms spaces, enough projective objects, finitely many simple objects up to isomorphism, finite length for every object), left exact functors (functors preserving finite limits) and linear natural transformations. The monoidal product is the Deligne product $\boxtimes$. The monoidal unit is the category $\vect$ of finite-dimensional vector spaces.

Concretely, such a $\Lexf$-valued modular $\cat{O}$-algebra is a finite category $\cat{A} \in \Lexf$ and, for every $o\in\cat{O}(T)$ for a corolla $T$, a left exact functor $\cat{A}_o:\cat{A}^{\boxtimes \Legs(T)}\to\vect$ with further coherence data and conditions that in particular express the compatibility with operadic composition, see \cite[Section~2.4]{cyclic}.
In the case in which $o$ is the operadic identity, we obtain a left exact functor $\kappa : \cat{A}\boxtimes\cat{A}\to\vect$ that gives us a non-degenerate symmetric pairing on $\cat{A}$ that can equivalently be described by an equivalence $D:\cat{A}^\op \to \cat{A}$ with
$\kappa(X,Y)\cong \cat{A}(DX,Y)$ for $X,Y\in\cat{A}$, where $\cat{A}(-,-)$ denotes the morphisms spaces of $\cat{A}$. Because of the symmetry of $\kappa$, the anti-equivalence $D$ squares to the identity up to a canonical isomorphism $\id_\cat{A}\ra{\cong} D^2$. 
We describe modular algebras through left exact functors
$\cat{A}_o:\cat{A}^{\boxtimes \Legs(T)}\to\vect$, but we could actually use the pairing $\kappa$ to move all copies of $\cat{A}$ to the right to obtain a left exact functor $\vect \to \cat{A}^{\boxtimes \Legs(T)}$ (which is an object in $\cat{A}^{\boxtimes \Legs(T)}$).
If we do this for the pairing itself, we obtain the \emph{coevaluation object} $\Delta \in \cat{A}\boxtimes\cat{A}$ that by \cite[Proposition 2.26]{cyclic} is given by the coend \begin{align}\label{eqncoev}\Delta = \int^{X\in\cat{A}} DX\boxtimes X \ . 
	\end{align}
The compatibility of the modular algebra with gluing tells us that, if an operation $o'$ arises from an operation $o \in \cat{O}(T)$ by taking the operadic composition over two legs of $T$, then there is a canonical isomorphism
\begin{align}\label{eqnexcision} \cat{A}_{o'} \cong \cat{A}_o (...,\Delta',\dots, \Delta'',\dots) \ , 
	\end{align} where $\Delta = \Delta' \boxtimes \Delta''$ (in Sweedler notation) is inserted into the slots of $\cat{A}_o$ affected by the gluing.
The gluing property expressed by~\eqref{eqnexcision} is often referred to as \emph{excision}.

If we restrict in the definition of a modular operad the graphs describing the operadic compositions to graphs which are \emph{forests} (disjoint unions of contractible graphs),
we obtain cyclic operads~\cite{gk} and their cyclic algebras in the description of \cite{costello}, see \cite[Section~2]{cyclic} for details.

\subsection{Modular functors\label{secmf}}
The notion of a modular operad is best illustrated by looking at the motivating example in \cite{gkmod}, the \emph{modular surface operad} $\Surf$ whose category-valued version~\cite{cyclic,brochierwoike} will be one of the key objects of this paper:
For a corolla $T$, the category of operations $\Surf(T)$ is a groupoid.
Its objects are connected compact smooth oriented surfaces $\Sigma$ together with an orientation-preserving diffeomorphism $(\mathbb{S}^1)^{\sqcup \Legs(T)}\ra{\cong} \partial \Sigma$ (for some fixed standard orientation on the circle) that is referred to as \emph{boundary parametrization}. 
A compact oriented smooth surface with parametrized boundary will just be referred to as \emph{surface} in the sequel.
Morphisms in $\Surf(T)$ are isotopy classes of diffeomorphisms preserving the orientation and the boundary parametrization; in short, the morphisms are mapping classes. In particular, the automorphism group of $\Sigma$ in $\Surf$ is the (pure) mapping class group $\Map(\Sigma)$ of $\Sigma$, see \cite{farbmargalit} for a textbook on mapping class groups.

For the modular surface operad, a theory of suitable extensions $\cat{Q}\to\Surf$ is developed in \cite[Section~3.1]{brochierwoike}
(suitable meaning that the extension is relative to genus zero and satisfies an insertion of vacua property),
 and a \emph{modular functor} is a pair $(\cat{Q},\cat{B})$ of such an extension and a modular $\cat{Q}$-algebra $\cat{B}$.
While this is inspired by the classical definitions~\cite{Segal,ms89,turaev,tillmann,baki,jfcs} (that are partly in mutual disagreement), this definition covers a much greater generality. Luckily, the subtleties of the definition of a modular functor are not terribly important
for this paper because
we will only treat one well-understood classical class of examples, namely the \emph{Lyubashenko modular functor}~\cite{lyubacmp,lyu,lyulex,kl} built from a not necessarily semisimple modular category $\cat{A}$.
Recall that a \emph{modular category} $\cat{A}$ is a finite tensor category~\cite{etingofostrik} (a finite category with bilinear rigid monoidal product and simple unit; we denote the rigid duality by $-^\vee$) that additionally comes equipped with a braiding $c$ and a balancing $\theta$ that are subject to two conditions:
\begin{itemize}
	\item The balancing $\theta$ is compatible with the rigid duality in the sense that $\theta_{X^\vee}=\theta_X^\vee$ for $X\in\cat{A}$. This
	makes $\cat{A}$ a so-called \emph{finite ribbon category}~\cite{egno}. Because of this, $-^\vee$ is automatically a two-sided duality and it comes with a \emph{pivotal structure}, i.e.\ a monoidal isomorphism $\id_\cat{A}\ra{\cong}-^{\vee\vee}$. 
	\item The braiding $c$ of $\cat{A}$ is non-degenerate, i.e.\ $c_{Y,X}\circ c_{X,Y}=\id_{X,Y}$ for all $Y\in\cat{A}$ implies $X\cong I^{\oplus n}$ for some $n\ge 0$.
\end{itemize}
A semisimple modular category is called a \emph{modular fusion category}.
In the modular fusion case, the Lyubashenko modular functor agrees with the one built by Reshetikhin-Turaev~\cite{rt1,rt2,turaev}.

For a modular category $\cat{A}$, we denote by $\FA$ its Lyubashenko modular functor.
To a surface $\Sigma$ with $n$ boundary components, it assign a left exact functor
$\FA(\Sigma):\cat{A}^{\boxtimes n}\to\vect$ on which a central extension
\begin{align} 1 \to k^\times \to \widetilde {\Map(\Sigma)} \to \Map(\Sigma) \to 1 \ , \label{eqnextensions}
\end{align} of $\Map(\Sigma)$ acts through natural transformations. The extensions showing up here are well understood and arise from the so-called \emph{framing anomaly}, see e.g.~\cite{gilmermasbaum} for more background.
An explicit description of the projective mapping class group representation on $\FA(\Sigma)$ at all genera is possible in terms of generators and relations of the mapping class groups, but it is notoriously difficult.
We will content ourselves with the following implicit description:
If $\Sigma$ has genus zero, then \begin{align}\label{eqngenuszeroblocks} \FA(\Sigma)(X_1,\dots, X_n)\cong \cat{A}(I,X_1\otimes\dots\otimes X_n)\end{align} with the mapping class group (which is the ribbon braid group in this case) acting in the usual way through the balanced braided structure of $\cat{A}$. 
The Lyubashenko modular functor $\FA$ is now characterized completely by the fact that, up to a contractible choice, it is the unique modular functor extending this prescription in genus zero to all surfaces \cite[Section~8.1]{brochierwoike}.

\subsection{Symmetric Frobenius algebras}
A symmetric Frobenius algebra $F$ in a pivotal finite tensor category $\cat{A}$ has an underlying 
object $F\in \cat{A}$, an associative multiplication $\mu : F \otimes F \to I$ with unit $\eta : I \to F$ and a non-degenerate, invariant and symmetric pairing $\beta :F\otimes F \to I$. For more background on symmetric Frobenius algebras, see e.g.~\cite{fuchsstigner}.
The non-degeneracy means that the induced map $\psi  = \kappa \otimes \id_{F^\vee} \circ (\id_F \otimes b_F):F\to F^\vee$ defined using the coevaluation $b_F:I\to F\otimes F^\vee$ is an isomorphism.
The invariance of the pairing amounts to $\kappa(-,\mu)=\kappa(\mu,-)$.
The symmetry of $\beta$ means that $F\ra{\omega_F}F^{\vee\vee} \ra{\psi^\vee} F^\vee$ agrees with $\psi$.
In other words, $\psi$ is a self-duality in the sense of \cite[Section~4]{microcosm}, see also Section~\ref{secbulkfield} below.

We will need later for any symmetric Frobenius algebra $F\in \cat{A}$ a \emph{canonical operation \begin{align} \label{eqntarn}v	_n : I\to F^{\otimes n}\end{align} of total arity $n\ge 0$}:\begin{itemize}
	\item For $n=0$, it is the identity of $I$. \item For $n=1$, it is the unit $I\to F$. \item For $n=2$, it is the copairing $I\to F\otimes F$.
	\item For $n\ge 3$, it is given inductively by
	\begin{align}
		v_n : I   \ra{v_{n-1}} F^{\otimes (n-1)} \ra{  \delta \otimes \id_{F^{\otimes (n-2)}}   } F^{\otimes n}  \ , 
	\end{align}
	where $\delta : F \to F\otimes F$ is the coproduct.
\end{itemize}

\begin{definition}[$\text{\cite[Definition 2.22]{correspondences}}$]\label{defspecial}
	A symmetric Frobenius algebra $F$ 
	in a pivotal finite tensor category $\cat{A}$
	with product $\mu :F \otimes F\to F$, unit $\eta :I\to F$, coproduct $\delta:F\to F\otimes F$ and counit $\varepsilon :F\to I$  is called \emph{special} if
	\begin{align}
		\varepsilon \circ \eta & \neq 0 \ , \\
		\mu \circ \delta &= \id _F \ . 
	\end{align}
\end{definition}

\begin{remark} Sometimes one requires instead 
	\begin{align}
		\varepsilon \circ \eta &= \beta_I\cdot  \id_I \, \\
		\mu \circ \delta &= \beta_F\cdot  \id_F 
	\end{align}
	for invertible scalars $\beta_I$ and $\beta_F$, but in this situation, it is standard to renormalize using the following convention:
	The two conditions imply that $\beta_I \beta_F$ agrees with the dimension $\dim\, F\in k$ defined by 
	$
	\dim\, F \cdot \id_I = \varepsilon \circ \mu \circ \delta \circ \eta
	$ of $F$, which hence is non-zero.
	We then renormalize the Frobenius form $\varepsilon$ 
	 such that the conditions from Definition~\ref{defspecial} hold.
\end{remark}

\begin{remark}\label{remh}
	For a symmetric Frobenius algebra $F$, one calls $H:= \mu \circ \delta$ also the \emph{handle operator}, see e.g.~\cite[page~128]{kocktft}, and $h:= \mu \circ \delta \circ \eta =H\circ \eta$ the \emph{handle element}.
	Since $H=\mu(h,-)$, we have 
	\begin{align}
		H=\id_F \quad \Longleftrightarrow \quad h=\eta \ . \label{eqnh}
	\end{align}
\end{remark}

\begin{example}\label{exspecialsymfrob}
	The monoidal unit $I\in\cat{A}$ of any finite tensor category is obviously a special symmetric Frobenius algebra.
	A different class of examples can be extracted from~\cite[Proposition~6.2]{laugwitzwalton}:
	Let $K$ be a finite-dimensional commutative non-semisimple Hopf algebra such that its Drinfeld double $D(K)$ is ribbon,
	and let $N$ be a finite abelian group such that the characteristic of our ground field $k$ does not divide $|N|$.
	If one endows $H=K\otimes k[N]$ with the tensor product Hopf algebra structure ($k[N]$ has the standard Hopf algebra structure on the group algebra), then the Drinfeld center $Z(H\catf{-mod})$ is a non-semisimple modular category, and
	$k[N]$ carries the structure of a special symmetric Frobenius algebra in $Z(H\catf{-mod})$.
	In fact, in \cite{laugwitzwalton}, it is proved that this Frobenius algebra is special and has a trivial twist. By \cite[Proposition 2.25~(ii)]{correspondences} it is then also symmetric.
\end{example}

	\section{A description of correlators via the microcosm principle\label{secmicrocosm}}
The notion of a consistent system of correlators is standard and was recalled in the introduction: It is a collection of vectors 
 in the spaces of conformal blocks invariant under the mapping class group and under gluing.
As intuitive as this description may be, it is full of technical subtleties, particularly beyond the semisimple case. 
These will be addressed in this section.

\subsection{Holomorphic factorization}
Correlators are mapping class group invariant vectors in the spaces of conformal blocks.
Clearly, we will not ask for invariance under the extension $\widetilde {\Map(\Sigma)}$. This would not make any sense because non-zero vectors are not invariant under the multiplication with invertible scalars different from $1$.
But how does one ask for invariance under $\Map(\Sigma)$? Na\"ively, one might want to lift an element $f\in \Map(\Sigma)$ to $\widetilde {\Map(\Sigma)}$ and ask for invariance, but two such lifts will differ by a scalar. So if we have invariance for one choice, we will not have invariance for a different one. 

Luckily, we are not really in this situation because if the conformal field theory has monodromy data $\cat{A}$, we will not search for invariant elements in the spaces of conformal blocks for $\cat{A}$,
 but for $\bar{\cat{A}}\boxtimes \cat{A}$. 
This is the principle of \emph{holomorphic factorization}, see e.g.~\cite{algcften} for more background.

If $\cat{A}$ is a modular fusion category, then $\FA(\Sigma)$ is the value of the Reshetikhin-Turaev topological field theory on $\Sigma$ and
$\FA(\bar\Sigma)\cong \FA(\Sigma)^*$ by a canonical $\Map(\Sigma)$-equivariant isomorphism which means that $\FAA(\Sigma)$ is $\Map(\Sigma)$-equivariantly isomorphic to $\End\, \FA(\Sigma)$. This implies that there is actually in a canonical way a non-projective $\Map(\Sigma)$-action on $\FAA(\Sigma)$, and this is the action under which the correlators should be invariant.
In that way, the construction of \cite{frs1,frs2,frs3,frs4,ffrs,ffrsunique} uses the duality in the three-dimensional bordism category to eliminate the anomaly for $\FAA$. 

Beyond semisimplicity, without the Reshetikhin-Turaev topological field theory in the background, this becomes more difficult. 
There exists a definition of correlators in the non-semisimple situation~\cite{jfcs,jfcslog}, but this assumes in crucial places that mapping class group elements act as actual operators on the spaces of conformal blocks. This leads to invariance equations as in \cite[Definition 4.9]{jfcs} that depend on this choice. 	
That is not necessarily a problem, but in order to apply the	 definitions of \cite{jfcs}, we need to specialize to $\bar{\cat{A}} \boxtimes \cat{A}$ and provide a  splitting of the extensions~\eqref{eqnextensions} in that situation. 
The correlators then need to be constructed \emph{in reference to that splitting}.
The technical framework that needs to be set up for this is unfortunately rather involved.

\subsection{The modular microcosm principle for the surface operad\label{secbulkfield}}
Suppose that $\cat{O}$ is a category-valued modular operad and $\cat{B}$ a $\Lexf$-valued modular $\cat{O}$-algebra, then the microcosm principle for modular operad set up in \cite{microcosm} as a generalization of the microcosm principle of Baez-Dolan~\cite{baezdolanmicrocosm} allows us the define modular $\cat{O}$-algebras (with coefficients) in $\cat{B}$.

Let us recall this concept for a $\Surf$-algebra in a $\Surf$-algebra $\cat{B}$ (the \emph{macrocosmic} $\Surf$-algebra).
By abuse of notation we denote the underlying category of $\cat{B}$ also by $\cat{B}$.
Then a modular $\Surf$-algebra in $\cat{B}$ has an underlying object $B\in\cat{B}$ (the \emph{microcosmic} $\Surf$-algebra).
By \cite[Theorem 7.17]{cyclic} $\cat{B}$ inherits from the topology a ribbon Grothendieck-Verdier structure in the sense of Boyarchenko-Drinfeld~\cite{bd}, i.e.\ a balanced braided category with a weak form of duality $D:\cat{A}^\op \to \cat{A}$; this is the anti-equivalence encountered in Section~\ref{secmodop}. For the purpose of this article, it is actually enough to imagine that $D$ is the rigid duality.

With respect to $D$, the object $B\in\cat{B}$ must be self-dual in the sense of \cite[Section~4]{microcosm}, i.e.\ it is equipped with an isomorphism $\psi: B\cong DB$ such that $B\ra{\cong} D^2 B\ra{D\psi} DB$ agrees with $\psi$. This self-duality amounts to a non-degenerate symmetric pairing $B\otimes B \to K$ with $K=DI$. Again, the rigid case is enough. Then $K=I$, and we will find a non-degenerate symmetric pairing in the usual sense. Beyond that, similar notions are available in \cite{fsswfrobenius}.

By \cite[Section~3]{microcosm} this entails that for a surface $\Sigma$ the vector spaces $\cat{B}(\Sigma;B,\dots,B)$ form a
functor $\int \Surf \to \vect$ from the Grothendieck construction $\int \Surf$
 (the category obtained by gluing all spaces of operations of $\Surf$ together) to $\vect$. This can be seen as 
a flat vector bundle over the moduli stack of Riemann surfaces. 

A \emph{modular $\Surf$-algebra in a $\Lexf$-valued $\Surf$-algebra $\cat{B}$} with underlying self-dual object $B \in \cat{B}$ now consists of vectors $\cor_\Sigma^B \in \cat{B}(\Sigma;B,\dots,B)$ invariant under the mapping class group and gluing such that 
$\cor_{\mathbb{D}^2 \times \mathbb{S}^1}^B \in \cat{B}(\mathbb{D}^2 \times \mathbb{S}^1;B,B)=\cat{B}(B\otimes B,K)$ for the cylinder is the pairing of $B$.
 The latter is called the \emph{unitality condition}.

\subsection{Bulk field correlators via reflection equivariance\label{secbulkfield}}
Finally, we will give in this subsection a description of bulk field correlators equivalent to the one from \cite{jfcs}.
	The latter definition is developed a bit further \cite[Section~4.2]{jfcslog} where it is suggested to see the self-duality of the Frobenius algebras typically appearing in the description of bulk field correlators 
	as a shadow of the rigid duality under the microcosm principle.
	Various incarnations of this seem to be known to experts: A correspondence between Frobenius structures and Grothendieck-Verdier duality via the microcosm principle is mentioned without further explanation in \cite[Section~2]{nlab:microcosm_principle}. It  is moreover implicit e.g.\ in \cite[Definition~2.3.2]{egger}.
	
The framework of \cite{microcosm}, that treats ansular and open correlators only, is not enough to turn the remark in \cite[Section~4.2]{jfcslog} into a result. Instead, a crucial ingredient that we are still missing is \emph{reflection equivariance}~\cite{reflection}:
Let $\cat{A}$ be a modular category.
The prescription~\eqref{eqngenuszeroblocks} gives us a genus zero modular functor, i.e.\ a cyclic framed $E_2$-algebra $\cat{A}$~\cite[Section~5]{cyclic} in $\Lexf$. By dualizing the morphism spaces and the labels we obtain a new cyclic framed $E_2$-algebra $\cat{A}\dg$~\cite[Section~3]{reflection}. 
By \cite[Theorem 5.11]{reflection}
the choice of a two-sided \emph{modified trace}~\cite{geerpmturaev,mtrace1,mtrace2,mtrace3,mtrace} on $\Proj \cat{A}$ provides an equivalence
\begin{align}
	\cat{A}\dg \ra{\simeq} \bar{\cat{A}}
	\end{align}
of cyclic framed $E_2$-algebras (that is the identity when forgetting the cyclic structure).
This structure is referred to as \emph{cyclic reflection equivariance}; it is a homotopy $\mathbb{Z}_2$-fixed point structure with respect to the operation that dualizes the genus zero spaces of conformal blocks and acts by orientation reversal on genus zero surfaces.

Suppose now that we are given a conformal field theory whose monodromy data is a modular category $\cat{A}$, then we clearly recover the notion of a bulk field correlator in line with the definitions in \cite{jfcs} if we consider the $\Surf$-algebras in $\FAA$, see also the explanations in \cite[Section~8]{microcosm}.
But for this,
 $\FAA$ needs to be seen as $\Surf$-algebra. As explained above, the definition cannot be made without this detail.
We achieve this by choosing a cyclic reflection equivariant structure for $\cat{A}$. Then 
\cite[Theorem 5.11]{reflection} provides us with an equivalence
\begin{align}\label{eqnequivmf} \FAA \ra{\simeq} \mathfrak{F}_{\! \cat{A}\dg \boxtimes \cat{A}}\end{align}
 of modular functors.
Thanks to \cite[Lemma 6.19]{reflection}, $\FAA$ then lives canonically over $\Surf$ once a cyclic reflection equivariance is chosen.
In other words, $\FAA$ is anomaly-free in a canonical way.
In terms of string-nets, this can be deduced from \cite[Corollary~8.3]{sn}, but this does not really help us because we will not formulate the correlators in terms of string-nets.

With \emph{this} structure of a $\Surf$-algebra for $\FAA$, the groupoid $\catf{ModAlg}(\Surf;\FAA)$ of $\Surf$-algebras in $\FAA$~\cite[Remark 5.3]{microcosm} describes the bulk field correlators for the conformal field theory with monodromy data $\cat{A}$. We suppress the chosen reflection equivariant structure in the notation, and this is justified:
The groupoid $\catf{ModAlg}(\Surf;\FAA)$ does not depend on the chosen reflection equivariant structure because the $\Surf$-structure on $\FAA$ is always chosen such that there is a canonical equivalence $\catf{ModAlg}(\Surf;\FAA) \ra{\simeq} \catf{ModAlg}(\Surf;\mathfrak{F}_{\! \cat{A}\dg \boxtimes \cat{A}})$. 
Let us summarize:

\begin{proposition}\label{propmicrocosmcor}
	A bulk field correlator for a conformal field theory whose monodromy data is a modular category $\cat{A}$ can be equivalently described as a
	modular $\Surf$-algebra with coefficients in the modular $\Surf$-algebra $\bar{\cat{A}}\boxtimes\cat{A}$, where the latter $\Surf$-algebra structure is inherited from $\cat{A}\dg \boxtimes \cat{A}$ through the choice of a reflection equivariant structure.
\end{proposition}

Note that in order to apply the microcosm principle and to obtain the correct notion of correlators, 
it seems that we really need the concept of reflection equivariance and the invertibility 
of the generalized skein modules from \cite{reflection} as both enter into the cancellation of the anomaly. 
These rely in turn heavily on Grothendieck-Verdier duality and factorization homology, and we are not aware of	 a construction that circumvents this.
In the semisimple case, all of these issues, in particular those related to the anomaly, are implicitly resolved by means of the Reshetikhin-Turaev topological field theory.

\needspace{10\baselineskip}
	\section{The modular extension procedure 
		for the underlying pivotal finite tensor category
		\label{secopen}}
The modular functor for $\bar{\cat{A}}\boxtimes\cat{A}$
is actually an open-closed modular functor~\cite[Corollary 8.7]{sn},
and the correlators that we want to build should be compatible with this additional structure.
To this end, we treat in this section the open sector first.

\subsection{The open modular functor of a pivotal finite tensor category}
A \emph{categorified open two-dimensional topological field theory}~\cite{envas}, that in this text we will mostly refer to as \emph{open modular functor} is a modular algebra over the open surface operad. The definition is inspired by the concept of an open field theory by Lazaroiu~\cite{Lazaroiu} and Moore-Segal~\cite{mooresegal}.
The open surface operad is  a groupoid-valued modular operad $\O$ whose operations in total arity $n$ are given by connected compact oriented surfaces with at least one boundary component and $n\ge 0$ parametrized intervals in the boundary of the surface, sometimes referred to  as \emph{marked intervals}. The rest of the boundary is not parametrized and called \emph{free boundary}. We refer to Figure~\ref{figexopen} for an example.
The morphisms in the groupoids of operations are mapping classes, more precisely isotopy classes of diffeomorphisms
preserving the orientation and the boundary parametrization.
The operadic composition is gluing along intervals.
For more details, we refer to \cite[Section~2 \& 3]{sn}.

The	cyclic operad of disks with marked intervals in their boundary
is  the cyclic associative operad. Therefore, the restriction of an open modular functor in $\Lexf$
to disks with marked intervals in their boundary is a cyclic associative algebra in $\Lexf$.
A cyclic associative algebra in $\Lexf$ has in particular an underlying non-cyclic associative algebra in $\Lexf$, i.e.\ a monoidal category $\cat{A}$
whose monoidal product we denote by $\otimes$.
In addition, we have the duality $D: \cat{A}\to\cat{A}^\op$ coming from the symmetric non-degenerate pairing 
and the symmetry isomorphisms that give us a natural isomorphism $\omega :\id_\cat{A}\to D^2$ as explained in Section~\ref{secmodop}.
By \cite[Theorem~4.12]{cyclic} being a cyclic associative algebra means exactly that this data forms a so-called \emph{pivotal Grothendieck-Verdier category} in the sense of~\cite{bd}. This means the following things:
\begin{itemize}
	\item  $(\cat{A},\otimes,D)$ is a \emph{Grothendieck-Verdier category} (using the dual conventions of \cite{cyclic}), i.e.\ the equivalence $D: \cat{A}\to\cat{A}^\op$ makes the functors 
	$\cat{A}(K,X\otimes -)$ for $X\in \cat{A}$ and the \emph{dualizing object} $K:=DI$
	representable by means of isomorphisms $\cat{A}(K,X\otimes -)\cong \cat{A}(DX,-)$
	(the equivalence $D$ is called a \emph{Grothendieck-Verdier duality});
	
	\item the isomorphism $\omega:\id_\cat{A}\to D^2$ is monoidal, and  $\omega_K$ is the canonical isomorphism $K\cong D^2 K$. 
\end{itemize}
An example of a pivotal Grothendieck-Verdier category is a pivotal finite tensor category in the sense of \cite{etingofostrik,egno}.
By \cite[Theorem 2.2 \& Corollary~4.4]{envas} the restriction map from open modular functors to pivotal Grothendieck-Verdier categories is an equivalence.

\begin{figure}[h]
	\centering	\begin{tikzpicture}[scale=0.5]
		\begin{pgfonlayer}{nodelayer}
			\node [style=none] (0) at (7.5, 11) {};
			\node [style=none] (1) at (7.5, 9) {};
			\node [style=none] (2) at (7.5, 8) {};
			\node [style=none] (3) at (7.5, 6) {};
			\node [style=none] (4) at (10, 7.5) {};
			\node [style=none] (5) at (12, 7.5) {};
			\node [style=none] (6) at (10.5, 6.75) {};
			\node [style=none] (7) at (11.5, 6.75) {};
			\node [style=none] (8) at (11, 8.75) {$\Sigma$};
			\node [style=none] (9) at (7.5, 5) {};
			\node [style=none] (10) at (7.5, 3) {};
			\node [style=none] (11) at (8.25, 10.5) {};
			\node [style=none] (12) at (8.25, 9.5) {};
			\node [style=none] (13) at (8.25, 4.5) {};
			\node [style=none] (14) at (8.25, 3.5) {};
			\node [style=none] (15) at (6, 7) {};
			\node [style=none] (16) at (0.5, 5) {};
			\node [style=none] (17) at (0.5, 4.5) {free boundary component};
			\node [style=none] (18) at (6.25, 4) {};
			\node [style=none] (19) at (3.75, 2.5) {};
			\node [style=none] (20) at (3.75, 2) {a marked interval};
		\end{pgfonlayer}
		\begin{pgfonlayer}{edgelayer}
			\draw [bend right=90, looseness=1.50] (1.center) to (0.center);
			\draw [bend right=270, looseness=1.50] (1.center) to (0.center);
			\draw [bend left=90, looseness=1.50] (2.center) to (3.center);
			\draw [bend right=90, looseness=1.50] (2.center) to (3.center);
			\draw [bend right=90, looseness=1.50] (4.center) to (5.center);
			\draw [bend left=45, looseness=1.25] (6.center) to (7.center);
			\draw [bend right=90, looseness=1.50] (10.center) to (9.center);
			\draw [bend right=270, looseness=1.50] (10.center) to (9.center);
			\draw [bend left=90, looseness=4.50] (1.center) to (2.center);
			\draw [bend left=90, looseness=4.50] (3.center) to (9.center);
			\draw [bend left=90, looseness=3.00] (0.center) to (10.center);
			\draw [style=open, bend left] (11.center) to (12.center);
			\draw [style=open, bend left] (13.center) to (14.center);
			\draw [style=open, bend right=90, looseness=1.50] (9.center) to (10.center);
			\draw [style=end arrow, in=-180, out=90] (16.center) to (15.center);
			\draw [style=end arrow, in=-180, out=90, looseness=1.25] (19.center) to (18.center);
		\end{pgfonlayer}
	\end{tikzpicture}
	\caption{An operation of total arity 3 in the open surface operad. Marked intervals are printed in blue.}
	\label{figexopen}
\end{figure}
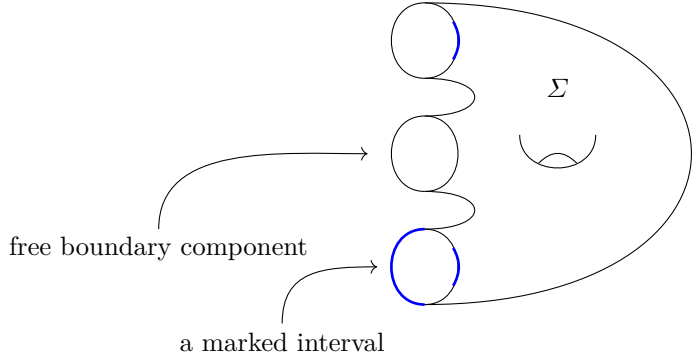

Let $\cat{A}$ be a ribbon category in $\Lexf$.
By what we just explained,	there is a unique open modular functor $\Ao$ whose restriction to disks with marked boundary intervals is given by $\cat{A}$ as pivotal finite tensor category. 
This means that the open modular functor is characterized by the fact that 
on a disk $D$ with $n\ge 0$ boundary intervals it is given by
\begin{align} \Ao(D ; X_1,\dots,X_n) \cong \cat{A}(I,X_1\otimes \dots \otimes X_n) \ ,  \label{eqndisk}\end{align}
where the needed cyclic symmetry is the usual cyclic symmetry of the hom spaces of pivotal tensor categories.

\subsection{The construction of open correlators and implications\label{secopencor}}
Let $\cat{A}$ be a pivotal tensor category.
In \cite[Theorem 8.3]{microcosm} the modular microcosm principle is used to build a consistent system of open correlators from any symmetric Frobenius algebra $F\in\cat{A}$, i.e.\
vectors \begin{align}\label{omicrocosmeqn}
	\lambda_\Sigma^F \in \Ao(\Sigma;F,\dots,F)\quad \text{with}\quad \Sigma \in \O(T)\end{align} that are mapping class group invariant and preserved by the gluing along intervals (this is the \emph{open correlator}).
All consistent systems of open correlators with monodromy data $\cat{A}$ are of this form.
The construction in \cite{microcosm} relies on Costello's modular envelope and the description of the latter in the case of the cyclic associative operad \cite{giansiracusa,envas} in terms of ribbon graphs.
It should be thought of as a generalization of the usual construction principles in quantum topology involving skein theory, triangulations or state-sums.
It directly extends a cyclic associative algebra in $\cat{A}$ (a symmetric Frobenius algebra) to the entire open surface operad.
We refer to this as \emph{modular extension}.

For a modular category $\cat{A}$ with fixed reflection equivariant structure, we will now rewrite this result in terms of the modular functor $\mathfrak{F}_{\bar{\cat{A}}\boxtimes\cat{A}}$:
To this end, we will use 
the definition
\begin{align}
	\bzero  (Y):= \int^{X \in \cat{A}} X^\vee \boxtimes Y\otimes X\in\bar{\cat{A}}\boxtimes\cat{A} \quad \text{for all}\quad Y\in\cat{A} \ . 
\end{align}
Under the equivalence \begin{align}\label{equivG} G:\bar{\cat{A}}\boxtimes\cat{A}\ra{\simeq}  Z(\cat{A})\end{align} from~\cite[Section~4]{eno-d},
$\bzero  (Y)$ is sent to $LY \in Z(\cat{A})$, where $L:\cat{A}\to Z(\cat{A})$ is the left adjoint to the forgetful functor $U:Z(\cat{A})\to \cat{A}$, see \cite[Section~2.2]{cardycase}.
With respect to the rigid duality $X \boxtimes Y \mapsto X^\vee \boxtimes Y^\vee$ of $\bar{\cat{A}}\boxtimes\cat{A}$, any self-duality $Y\cong Y^\vee$ induces a self-duality $\psizero:\bzero  (Y)\cong \bzero  (Y)^\vee$
(this uses that we have fixed a reflection equivariant structure and hence a trivialization of both Nakayama functors~\cite[Theorem~4.12]{reflection}, and moreover \cite[eq.~(3.52)]{fss}).
In particular, $\bzero  (F)$ for a symmetric Frobenius algebra $F\in\cat{A}$ is canonically a self-dual object.

\begin{theorem}\label{thmsym}
	Let $\cat{A}$ be a modular category and let $F$ be a 
	symmetric Frobenius algebra $F$ in $\cat{A}$.
	Then via the modular extension procedure, $F$ gives rise to 
	vectors $\lambda_\Sigma^F \in \mathfrak{F}_{\bar{\cat{A}}\boxtimes\cat{A}}(\Sigma;\bzero(F),\dots,\bzero(F))$ 
	in the spaces of conformal blocks for the Lyubashenko modular functor for the modular category $\bar{\cat{A}}\boxtimes\cat{A}$ at all surfaces $\Sigma\in \Surf(T)$ with at least one boundary component per connected component.  These vectors are mapping class group invariant and solve the gluing constraints for gluing along intervals. 
\end{theorem}

\begin{proof}
	For the construction of the vectors,  we can assume that $\Sigma$ is connected.
	Each boundary component $S$ of $\Sigma$ is parametrized, so it comes with a diffeomorphism $\mathbb{S}^1\ra{\cong} S$. We model $\mathbb{S}^1$ as the unit circle in $\mathbb{C}$ and  place a small marked interval on $S$ around the image of $1 \in \mathbb{S}^1$ under the $\mathbb{S}^1\ra{\cong} S$. This gives us a surface $\Sigma_*$ that is an operation
	 in the open surface operad; we have $\Map(\Sigma_*)=\Map(\Sigma)$ since it makes no difference for the mapping class group whether we fix a boundary component or a marked interval. 
	
	The symmetric Frobenius algebra $F$ gives us by~\eqref{omicrocosmeqn} 
	a vector in $\lambda_{\Sigma_*}^F \in \Ao(\Sigma_* ; F,\dots,F)$ invariant under the mapping class group and compatible with the gluing along intervals.
	Thanks to \cite[Theorem 4.3]{envas}, we have  \begin{align} \Ao(\Sigma_* ; F,\dots,F)\cong \mathfrak{F}_{Z(\cat{A})}(\Sigma;LF,\dots,LF)\label{eqnAoZ}\end{align} by a canonical isomorphism respecting the mapping class group action and the gluing along intervals; we denote the image of the just-constructed vector 
	by $\lambda_\Sigma^F$.
	The statement now follows from the above-mentioned fact that the equivalence~\eqref{equivG} sends $\bzero(F)$ to $LF$, see also \cite[Corollary 4.4]{envas}.
\end{proof}

\needspace{10\baselineskip}
\section{The correlator construction}
In this section, we finish
the correlator construction. Essentially, we will take the vectors from Theorem~\ref{thmsym} and prove the compatibility under gluing.

\subsection{Invariance under partial traces}
For  a pivotal  finite tensor category  $\cat{A}$ and  a connected surface $\Sigma$ with at least one boundary component
and $n\ge 1$ marked boundary intervals, denote 
by $\Sigma \setminus d$ the result of cutting out a small disk $d$  from
$\Sigma$, thereby adding a free boundary component to the surface.
We constructed in \cite[Section~6]{envas}
so-called \emph{partial trace maps}
\begin{align} \Ao (\Sigma \setminus d; X_1,\dots,X_n) \to  \Ao(\Sigma;X_1,\dots,X_n) \ . 
\end{align}

\begin{lemma}\label{lemmapt}
	Let $\cat{A}$ a modular category and $\Sigma\in \O(T)$ a connected surface with at least one boundary component
	and $n\ge 1$ marked boundary intervals. 
	Then for a  special symmetric Frobenius algebra $F$ in $\cat{A}$
	any partial trace map
	\begin{align} \Ao (\Sigma \setminus d; F,\dots,F) \to \Ao(\Sigma;F,\dots,F) 	\label{eqnpartialtrace}
	\end{align} preserves the vectors $\lambda^F$ from~\eqref{omicrocosmeqn}.
\end{lemma}

\begin{proof}
	\begin{pnum} \item We first treat the case $\Sigma =D$ in which $\Sigma$ is itself a disk:	Under the identifications
		\begin{align}
			\Ao(D;F,\dots,F) &\cong \cat{A}(I,F^{\otimes n}) \ , \\
			\Ao(D \setminus d ; F ,\dots, F) &\cong \cat{A}(I,F ^{\otimes n}\otimes \mathbb{F})  \ , 
		\end{align}
		that can be extracted from~\cite[Section~9]{microcosm}, 
		the map \eqref{eqnpartialtrace} is the  map
		\begin{align}
			\label{eqnpartialtrace0} 	\cat{A}(I,F ^{\otimes n}\otimes \mathbb{F}) \to \cat{A}(I,F^{\otimes n})
		\end{align} induced by the map $\tau:\mathbb{F}=\int^{X \in \cat{A}} X^\vee \otimes X\to I$  that  comes from the evaluations $ X^\vee \otimes X\to I$. 
		The vectors $\lambda^F$ are:
		\begin{itemize}
			\item The vector
			\begin{align}
				\lambda_{D\setminus d}^F:	
				I \ra{v_{n+2}} F^{\otimes n} \otimes F^{\otimes 2} \ra{\text{self-duality}} F^{\otimes n} \otimes F^\vee \otimes F \ra{\text{structure map for the coend}} F^{\otimes n}\otimes \mathbb{F} 
			\end{align}in $	\cat{A}(I,F^{\otimes n} \otimes \mathbb{F})$, where $v_n:I\to F^{\otimes n} $ was defined in~\eqref{eqntarn}.
			\item The vector
			\begin{align}
				\lambda_D^F=	v_n:I\to  F^{\otimes n}
			\end{align}
			in $\cat{A}(I,F^{\otimes n})$.
		\end{itemize}
		Therefore, it remains to prove
		\begin{align}
			\label{eqnspecial1}(\id\otimes \tau)\circ	\lambda_{D\setminus d}^F  = v_n \quad \text{for}\quad n\ge 1 \ .  
		\end{align}
		Note that $F^\vee \otimes F \to \mathbb{F} \ra{\tau} I$ is the coevaluation of $F$; in combination with the self-duality, it yields the pairing $\beta : F\otimes F\to I$, so that \eqref{eqnspecial1}, with the inductive definition of $v_n$, reduces to
		\begin{align} (\id_{F^{\otimes n}} \otimes (\beta \circ \delta))\circ v_{n+1}=v_n \ , 
			\end{align}
		and this equation holds indeed because $F$ is special.\label{ptstep1}
		
		\item Next we treat the general case:
		 We can obtain $\Sigma$ by gluing a disk $D$ along intervals. We then obtain $\Sigma \setminus d$ by gluing a suitable disk $D\setminus d$ with a hole along intervals. 
		By the topological construction of the partial trace maps the following diagram commutes:
		\begin{equation}\label{eqndiagramdisk}
			\begin{tikzcd}
			 \Ao(D\setminus d;F,\dots,F)	  \ar[]{rrr}{p_D}  \ar[swap]{dd}{G_d} &&&  \Ao(D;F,\dots,F)  \ar{dd}{G} \\
				\\ \Ao(\Sigma\setminus d;F,\dots,F)    \ar[]{rrr}{p_\Sigma} &&& 	\Ao(\Sigma;F,\dots,F) \ . 
			\end{tikzcd} 
		\end{equation}
		The horizontal arrows are the partial trace maps while
		the vertical ones are the gluing maps. 
		
		The statement that remains to be shown
		is $ p_\Sigma \precor_{\Sigma\setminus d}^F = \precor_\Sigma^F$. 
		Since the $\lambda^F$ are compatible with the gluing along intervals, $\precor_{\Sigma \setminus d}^F = G_d \precor_{D \setminus d}^F $
		and $\lambda_\Sigma^F=G\lambda_D^F$. 
		This leaves us with
		\begin{align} p_\Sigma \precor_{\Sigma\setminus d}^F =p_\Sigma G_d \precor_{D \setminus d}^F \stackrel{\eqref{eqndiagramdisk}}{=} Gp_D \precor_{D \setminus d}^F \stackrel{\text{\ref{ptstep1}}}{=} G\precor_D^F =\precor_\Sigma^F  
		\end{align}
		and finishes the proof.
	\end{pnum}
\end{proof}

\subsection{Gluing along boundary circles}
For the vectors constructed in Theorem~\ref{thmsym}, we can now prove further gluing properties along boundary \emph{circles}:
	
\begin{proposition}\label{propgluingspecial}
	The vectors $\precor_\Sigma^F \in \mathfrak{F}_{\bar{\cat{A}}\boxtimes\cat{A}}(\Sigma;\bzero(F),\dots,\bzero(F))$
	built in Theorem~\ref{thmsym} for a modular category $\cat{A}$ and a special symmetric Frobenius algebra $F\in\cat{A}$ are preserved by gluing along boundary circles if the surface obtained via gluing still has at least one boundary component per connected component.
\end{proposition}

The gluing maps with respect to which the $\precor^F$ are preserved originate from the fact that $\FAA$ is a modular functor
and \cite[Proposition~4.7]{microcosm}. This needs a self-duality of $\bzero(F)$, namely the one $\psizero:\int^{X\in\cat{A}} X^\vee \boxtimes F\otimes X\cong \int_{X\in\cat{A}}X^\vee \boxtimes F^\vee \otimes X$ coming from reflection equivariance and the self-duality of $F$,
see the explanations before Theorem~\ref{thmsym}.

\begin{proof}[\slshape Proof of Proposition~\ref{propgluingspecial}]
	Suppose that we perform a gluing $s : \Sigma \to \Omega$ along a pair of boundary circles such that $\Omega$ still satisfies the hypotheses of Theorem~\ref{thmsym} (at least one boundary component per connected component).
	The gluing along a boundary circle can be decomposed into two steps \cite[Theorem~6.3]{envas}:
	\begin{pnum}
		\item The gluing along a pair of intervals on the boundary. This creates an unparametrized boundary component, i.e.\ a removed disk. This first operation preserves the $\precor^F$. This is already included in Theorem~\ref{thmsym}. 
		\item The closing of the hole via a partial trace map (we can assume that $\Omega$ is connected)
		\begin{align} \Ao(\Omega\setminus d;F,\dots,F) \to \Ao(\Omega;F,\dots,F) \end{align}
		that under the isomorphisms $\Ao(-;F,\dots,F)\cong \FAA(-;\bzero(F),\dots,\bzero(F))$
		amounts to a map
		\begin{align} \mathfrak{F}_{\bar{\cat{A}}\boxtimes\cat{A}}(\Omega\setminus d ;\bzero(F),\dots,\bzero(F)) \to \mathfrak{F}_{\bar{\cat{A}}\boxtimes\cat{A}}(\Omega ;\bzero(F),\dots,\bzero(F)) \ . 
		\end{align}
		These partial trace maps also preserve the $\precor^F$ by Lemma~\ref{lemmapt} because $F$ is assumed to be special.
	\end{pnum}
	Therefore, the gluing $s$ preserves the $\precor^F$.
\end{proof}

\subsection{The cylinder idempotent\label{seccyl}}
Proposition~\ref{propgluingspecial} gives us for a special symmetric Frobenius algebra $F$ in a modular category $\cat{A}$ and any surface without closed components 	vectors $\precor_\Sigma^F \in \mathfrak{F}_{\bar{\cat{A}}\boxtimes\cat{A}}(\Sigma;\bzero(F),\dots,\bzero(F))$ that are invariant under the mapping class group and gluing. 
Most of the work is done, but there remains an issue:	The vector for the cylinder, when seen as endomorphism of $\bzero(F)$ is not the identity, as it should be for the unitality of the microcosmic $\Surf$-algebra describing the correlator, see Section~\ref{secbulkfield}. 
Let us explain this in detail: For a special symmetric Frobenius algebra $F$ in a modular category $\cat{A}$,
the vector $\precor_{\mathbb{S}^1 \times [0,1]}^F \in \FAA(\mathbb{S}^1 \times [0,1];\bzero(F),\bzero(F))$
from Proposition~\ref{propgluingspecial}
is by construction a symmetric copairing $\precor_{\mathbb{S}^1 \times [0,1]}^F:I\to \bzero(F) \otimes \bzero(F) $ in $\bar{\cat{A}}\boxtimes\cat{A}$ and hence a map $\bzero(F) \to \bzero(F)^\vee$ that however does not need to be an isomorphism because the pairing  could be degenerate.
By the self-duality of $\bzero(F)$ coming from reflection equivariance and self-duality of $F$, the map 
$\bzero(F) \to \bzero(F)^\vee$ gives rise to an endomorphism $C_F:\bzero(F)\to\bzero(F)$ in $\bar{\cat{A}}\boxtimes\cat{A}$. 
The compatibility with gluing along boundary circles from Proposition~\ref{propgluingspecial} implies $C_F \circ C_F=C_F$.
For this reason, we refer to $C_F$ as the \emph{cylinder idempotent}. 

This means that the unitality condition from Section~\ref{secbulkfield} is not necessarily satisfied.
This is a standard problem and there are standard methods to resolve this, see \cite[Section~3.2]{rcftsn} in the semisimple setting.
We will develop here a similar approach in the non-semisimple setting:
We will adjust $\bzero(F)$ such that $C_F$ is `forced' to become the identity. More precisely, we define $\bulk(F)\in\bar{\cat{A}}\boxtimes\cat{A}$ as the image of $C_F$, i.e.\ the equalizer
\begin{equation}\label{eqneq}
	\begin{tikzcd}
		\displaystyle 	\bulk(F)   	\ar[r,"j"]& 	\displaystyle	\bzero(F) \ar[r, shift left=2,"C_F"] \ar[r, swap,shift right=2,"\id"]
		& \displaystyle 	\bzero(F)  \ .      
	\end{tikzcd}
\end{equation} 
It is also the coequalizer
\begin{equation}\label{eqncoeq}
	\begin{tikzcd}
		\displaystyle	\bzero(F) \ar[r, shift left=2,"C_F"] \ar[r, swap,shift right=2,"\id"]
		& \displaystyle 	\bzero(F)  \ar[r,"p"]& \displaystyle 	\bulk(F) \ ,      
	\end{tikzcd}
\end{equation} 
where $p$ is induced by $C_F:\bzero(F)\to\bzero(F)$ and the universal property of $\bulk(F)$ thanks to~\eqref{eqneq}.
The idempotent is recovered as $j\circ p=C_F$; moreover, $p\circ j=\id_{	\bulk(F)}$.

The symmetry of $\precor_{\mathbb{S}^1 \times [0,1]}^F$ tells us that, under self-duality identification $\psizero:\bzero(F) \ra{\cong} \bzero(F)^\vee$ mentioned after Proposition~\ref{propgluingspecial},
the idempotent $C_F$ satisfies \begin{align} \psizero \circ C_F=C_F^\vee\circ  {\psizero}^\vee=C_F^\vee\circ \psizero  \ , \end{align} where ${\psizero}^\vee$ is seen as map $\bzero(F)\to\bzero(F)^\vee$ (we suppress the pivotal structure). 
With~\eqref{eqneq} and~\eqref{eqncoeq}, this entails that $\bulk(F)$ is self-dual via an isomorphism $\psi:\bulk(F) \ra{\cong} \bulk(F)^\vee$, namely the one satisfying
\begin{align} p^\vee \circ \psi = \psizero \circ j \ . 
\end{align}
To see that $\psi$ is again self-dual, i.e.\ $\psi = \psi^\vee$ as maps $\bulk(F)\to\bulk(F)^\vee$, observe
\begin{align}
	p^\vee \circ \psi \circ p= \psizero \circ j\circ p=\psizero \circ C_F=C_F^\vee \circ {\psizero}^\vee=(j\circ p)^\vee \circ {\psizero}^\vee =p^\vee \circ ({\psizero} \circ j)^\vee=p^\vee \circ \psi^\vee \circ p
\end{align}
and use that $p$ is an epimorphism while $p^\vee$ is a monomorphism.
Moreover, from
\begin{align}\psi \circ p\circ {\psizero}^{-1} = (p^\vee \circ \psi)^\vee \circ {\psizero}^{-1} = (\psizero \circ j)^\vee \circ {\psizero}^{-1}=j^\vee\ ,  
\end{align}
we deduce
\begin{align} \psi^{-1}\circ j^\vee = p \circ {\psizero}^{-1} 
	\label{eqnbzeroinv}
\end{align}
and therefore, after dualization, the relation
\begin{align}
	\label{eqnbzeroinv2} p^\vee \circ \psi = \psi^\circ\circ j 
\end{align}
that we will need later.

We finish the investigation of the cylinder idempotent with the following observation:

\begin{corollary}[Generalized idempotence]\label{corgi}
	For a special symmetric Frobenius algebra $F\in\cat{A}$ in a modular category and any surface $\Sigma$ with at least one boundary component per connected component,
	the endomorphism
	\begin{align}
		{C_F}_* : \FAA(\Sigma;\bzero(F),\dots,\bzero(F)) \to \FAA(\Sigma;\bzero(F),\dots,\bzero(F))
	\end{align} that applies $C_F:\bzero(F)\to\bzero(F)$ in any or even of the slots preserves the vectors
	\begin{align}
		\precor_\Sigma^F \in \FAA(\Sigma;\bzero(F),\dots,\bzero(F))\end{align} from Theorem~\ref{thmsym}.
\end{corollary}

If $\Sigma$ is a cylinder, this amounts to the usual idempotence of $C_F$, hence the name of the result.

\begin{proof}Without loss of generality, we can treat the case in which $C_F$ is applied to one boundary component.
	By definition the triangle
	\begin{equation}
		\begin{tikzcd}
			\FAA(\Sigma;\bzero(F),\dots,\bzero(F)) \ar[swap]{dd}{{C_F}_*}  \ar[]{drr}{\id\otimes \precor^F_{\mathbb{S}^1 \times [0,1]}}    &&	 \\ && \FAA(\Sigma;\bzero(F),\dots,\bzero(F)) \otimes \FAA(\mathbb{S}^1 \times [0,1];\bzero(F),\bzero(F))  \ar[]{lld}{\text{gluing $\Sigma \cup (\mathbb{S}^1 \times [0,1])\cong \Sigma$}}
			\\ \FAA(\Sigma;\bzero(F),\dots,\bzero(F))    
		\end{tikzcd} 
	\end{equation}
	commutes. 
	Therefore, the statement follows from Proposition~\ref{propgluingspecial}.
\end{proof}

\subsection{Extension to closed surfaces\label{secextclosedsurfaces}}
We have not made assignments for closed surfaces so far. To this end, we will later use the following extension result:

\begin{proposition}\label{propextclosed}
	Let $\cat{A}$ be a modular category and $\FAA$ the modular functor associated to $\bar{\cat{A}}\boxtimes\cat{A}$ with the $\Surf$-structure coming from reflection equivariance.
	Suppose that for a self-dual object $B \in \bar{\cat{A}} \boxtimes\cat{A}$
	the structure of a $\Surf$-algebra on $B$ with coefficients in $\FAA$ is only partially defined, namely
	through vectors $\cor_\Sigma \in \FAA(\Sigma;B,\dots,B)$ in the spaces of conformal blocks for surfaces with at least one boundary component on every connected component such that the invariance under the mapping class group and gluing holds whenever the $\xi$ are defined on the surface resulting from the gluing. We also demand the unitality condition on the cylinder.
	Then the vectors $\xi$ extend uniquely to a $\Surf$-algebra in $\FAA$, i.e.\ a consistent system of bulk field 
	correlators for the conformal field theory with monodromy data $\cat{A}$.
\end{proposition}

\begin{proof}
	The uniqueness is clear because every closed surface can be obtained by gluing surfaces with boundary components together.
	
	The existence is a bit more subtle:
	By Proposition~\ref{propmicrocosmcor} we can replace $\bar{\cat{A}}$ with $\cat{A}\dg$. 
	Let us now restrict the vectors $\xi$ to genus zero:	Then the vectors $\cor_\Sigma \in \mathfrak{F}_{\! \cat{A}\dg \boxtimes\cat{A}}(\Sigma;B,\dots,B)$ with $\Sigma$ running over all genus zero surfaces with at least one boundary component give
	 us a cyclic framed $E_2$-algebra in $\cat{A}\dg \boxtimes \cat{A}$. 
	By \cite[Proposition~10.2]{microcosm} $B$ becomes therefore a symmetric  commutative Frobenius algebra in $\cat{A}\dg\boxtimes \cat{A}$. 
	We want to extend this symmetric  commutative Frobenius algebra to a consistent system of correlators.

	To this end, denote by $\vartheta_H \in \left(  \widehat{\cat{A}\dg} \boxtimes \widehat{\cat{A}} \right)(H;B,\dots,B)$ the ansular correlator to which the genus zero restriction of $\xi$ extends thanks to \cite[Theorem~10.6]{microcosm}, i.e.\ the system of handlebody group invariant vectors in the vector spaces assigned to handlebodies by the ansular functor $\widehat{\cat{A}\dg} \boxtimes \widehat{\cat{A}}$, see \cite{mwansular} for more background on ansular functors.

	The isomorphism
	\begin{align}  \label{eqnH} \mathfrak{F}_{\! \cat{A}\dg \boxtimes \cat{A}}(\Sigma;B,\dots;B)  \ra{\cong} 
		\left( \widehat{\cat{A}\dg} \boxtimes \widehat{\cat{A}}\right)(H;B,\dots,B) \ , 
	\end{align} 
	that we obtain for the choice of any handlebody $H$ for any surface $\Sigma=\partial H$ sends $\cor_\Sigma$ to $\vartheta_H$ if $\Sigma$ has at least one boundary component per connected component (otherwise the $\cor_\Sigma$ are not defined yet).
	This is true because the $\vartheta_H$ and the $\cor_\Sigma$ both satisfy the compatibility under gluing. 
	
	The ansular correlator $\vartheta_H$ extends to a correlator over all of $\Surf$ if and only if its image under~\eqref{eqnH} is mapping class group invariant in the case in which $\Sigma$ is a torus with one boundary component. This follows from the model for the surface operad given in \cite[Proposition~6.21]{reflection} (note how this is essentially a version of the criterion \cite[Theorem 4.10]{jfcs} taking additionally reflection equivariance into account).
	But this image is $\xi_{\mathbb{T}_1^2}$, and its mapping class group 
	invariance holds by assumption.
\end{proof}

\subsection{The main result\label{secmainresult}} We can now finally prove the main result of this paper:

\begin{theorem}\label{thmmainshort}
	Let $\cat{A}$ be a modular category. Then any special symmetric Frobenius algebra $F\in\cat{A}$ gives rise to a consistent system of open-closed correlators for the modular functor $\FAA$.
\end{theorem}
This is the short version from the introduction. We will prove the following more elaborate version that treats the closed sector, i.e.\ the bulk field correlators. The open-closed nature follows from how the bulk part is constructed from the open part and will be discussed in Corollary~\ref{coropenclosed}. 

\begin{theorem}\label{thmmainlong}
	Let $\cat{A}$ be a modular category.
	Then any special symmetric Frobenius algebra $F$ in $\cat{A}$ gives rise to 
	vectors $\cor_\Sigma^F \in \mathfrak{F}_{\bar{\cat{A}}\boxtimes\cat{A}}(\Sigma;\bulk(F),\dots,\bulk(F))$ in the spaces of conformal blocks for the Lyubashenko modular functor for the modular category $\bar{\cat{A}}\boxtimes\cat{A}$ at all surfaces.  
	These vectors are mapping class group invariant and solve the sewing constraints for gluing along boundary circles. 
\end{theorem}

\begin{proof}[\slshape Proof of Theorem~\ref{thmmainlong}]
	Theorem~\ref{thmsym} provides
	 the mapping class group invariant vectors	\begin{align} \label{eqnvectorsxi} \precor_\Sigma^F \in \FAA(\Sigma;\bzero(F),\dots,\bzero(F))\end{align}  as long as $\Sigma$ has at least one boundary component per connected component. Since $F$ is special, these vectors
	are preserved under the gluing along boundary circles by Proposition~\ref{propgluingspecial} provided that the surface resulting from the gluing still has at least one boundary component per connected component.

	Now we restrict to the image of the cylinder idempotent (Section~\ref{seccyl}):
	We  define 
	 the vector
	$
		\cor^F_\Sigma\in \FAA(\Sigma;\bulk(F),\dots,\bulk(F)) 
	$ as the image of $\precor_\Sigma^F$ under the map that applies $p:\bzero(F)\to\bulk(F)$ from~\eqref{eqncoeq} to all labels.
	By construction these vectors are mapping class group invariant 
	because the mapping class group acts by natural transformations on the $\FAA(\Sigma;-)$.
	
	They are also preserved under gluing $\Sigma \to \Omega$ provided that the $\lambda^F$ are still defined on $\Omega$ as we will show next:
	Without loss of generality, we consider a gluing affecting the first two slots 
	of $\FAA(\Sigma;-)$.
	We will see the compatibility with gluing by means of the following diagram (all parts of the diagram will be explained afterwards):
	
	\footnotesize
	\begin{equation}\label{eqndiagramres}
		\begin{tikzcd}
		\ar[swap,bend right=20]{dddddd}{G} \ar[swap,bend right=20]{ddddddrr}{\ell: \bulk(F)\boxtimes\bulk(F)\to \Delta_2} 	\FAA(\Sigma;\bulk(F),\dots,\bulk(F))  &&&&  \ar[swap]{ddll}{\text{apply $C_F$ to the first copy of $\bzero(F)$}} \ar[swap]{llll}{p_\Sigma} \FAA(\Sigma;\bzero(F),\dots,\bzero(F))  \\ \\
			&& \FAA(\Sigma;\bzero(F),\dots,\bzero(F)) \ar[bend left=20]{ddddrr}{g} \ar[]{dd}{\ell^\circ:  \bzero(F)\boxtimes\bzero(F)\to \Delta_2}   && k \ar[]{ll}{\precor_\Sigma^F}   \ar[]{dddd}{\precor_\Omega^F} \ar[swap]{uu}{\precor_\Sigma^F} \\ \\
			&& \FAA(\Sigma;\Delta_2,\bzero(F),\dots,\bzero(F)) \ar[bend left=20]{ddrr}{\cong} \ar[]{dd}{p_\Sigma \ \text{on the remaining slots}}  \\ \\
			\FAA(\Omega;\bulk(F),\dots,\bulk(F))   	&& \ar[]{ll}{\cong}  \FAA(\Sigma;\Delta_2,\bulk(F),\dots,\bulk(F))   && \ar[bend left=20]{llll}{p_\Omega} \FAA(\Omega;\bzero(F),\dots,\bzero(F))   
		\end{tikzcd}
	\end{equation}
	\normalsize
	
	Here $g$ and $G$ are the respective gluing maps. Therefore, the commutativity of the outer diagram gives us indeed the compatibility of the $\xi^F$ with gluing.
	Let us now explain the inner diagrams and prove their commutativity:
	\begin{itemize}
		
		\item The map $\ell:  \bulk(F)\boxtimes\bulk(F)\to \Delta_2$, with the coevaluation object $\Delta_2$ of $\bar{\cat{A}}\boxtimes\cat{A}$, is the map associated to the self-duality of $\bulk(F)$. With the structure map $\iota_{\bulk(F)}:\bulk(F)^\vee \boxtimes \bulk(F)\to \Delta_2$ of the coend $\Delta_2$, it is
		\begin{align}
			\ell:	\bulk(F) \boxtimes\bulk(F) \ra{\psi\boxtimes \id}    \bulk(F)^\vee \boxtimes \bulk(F) 
		\ra{\iota_{\bulk(F)}}	\Delta_2
			 \ , 
		\end{align}
		where the isomorphism $\psi:\bulk(F)\to\bulk(F)^\vee$ is from Section~\ref{seccyl}.
		Analogously, $\psi^\circ :\bzero(F)\to\bzero(F)^\vee$ induces the map $\ell^\circ: \bzero(F)^\vee \boxtimes\bzero(F) \to \Delta_2  $.
		Through the definition of the gluing maps~\cite[Proposition 4.7]{microcosm}, the two triangles involving $\ell$ and $\ell^\circ$
		 commute by construction.
		
		\item The unlabeled isomorphisms `$\cong$' come from excision, and naturality of excision in the uninvolved boundary labels therefore gives us the commutativity of the lowest tetragon.
		
		\item On the right, the upper right triangle commutes by Corollary~\ref{corgi}, and the triangle below it by Proposition~\ref{propgluingspecial}.
		
		\item We are left with the pentagon in the middle whose commutativity follows from the commutativity of
		\begin{equation}\label{eqncommpentagon}
			\begin{tikzcd}
				\Delta_2     &&& 	\bzero(F)\boxtimes\bzero(F) \ar[swap]{lll}{\ell^\circ}  \\
				\\ \ar[]{uu}{\ell}  \bulk(F)\boxtimes\bulk(F)  &&& \bzero(F) \boxtimes \bzero(F) \ar[]{lll}{p\boxtimes p} \ar[swap]{uu}{C_F\boxtimes \id} \ . 
			\end{tikzcd} 
		\end{equation}
		To see the commutativity of~\eqref{eqncommpentagon}, observe with $C_F=j\circ p : \bzero(F)\to\bzero(F)$ for the monomorphism $j:\bulk(F)\to \bzero(F)$ and the epimorphism $p:\bzero(F)\to\bulk(F)$ from Section~\ref{seccyl} that
		\begin{align}
		\ell^\circ \circ (	C_F \boxtimes \id)   =  \iota _{\bzero(F)} \circ
		\left(\left(\psizero\circ j \circ p\right) \boxtimes \id\right) \stackrel{\eqref{eqnbzeroinv2}}{=}  \iota _{\bzero(F)}\circ   \left(\left(p^\vee \circ \psi \circ p \right) \boxtimes \id\right) 
		\end{align}
		With the universal property of the coend, this gives us
		\begin{align}
		\ell \circ (	p\boxtimes p)  =  \iota_{\bulk(F)} \circ   \left(\left(\psi\circ p  \right) \boxtimes p\right)  = \iota _{\bzero(F)}\circ   \left(\left(p^\vee \circ \psi \circ p \right) \boxtimes \id\right) =\ell^\circ \circ (	C_F \boxtimes \id) \ . 
		\end{align}
		This prove the commutativity of~\eqref{eqncommpentagon} and finishes the proof of the compatibility with gluing.
	\end{itemize}

	We summarize: The vectors $\cor^F$
	\begin{itemize}
		\item are invariant under the mapping class group and under gluing of boundary circles (at least to the extent that they are defined yet because we have not made assignments for closed surfaces so far),
		\item and the vector in $\FAA(\mathbb{S}^1 \times \mathbb{D}^2;\bulk(F),\bulk(F))$ selects the non-degenerate symmetric pairing for $\bulk(F)$ by construction, see Section~\ref{seccyl}. 
	\end{itemize} The application of Proposition~\ref{propextclosed} gives us the desired unique extension to closed surfaces and finishes the proof that 
	the $\cor^F\in \mathfrak{F}_{\bar{\cat{A}}\boxtimes\cat{A}}(\Sigma;\bulk(F),\dots,\bulk(F))$ are a consistent system of bulk field correlators.
\end{proof}

A  observation following mostly from the proof above is the following:

\begin{corollary}\label{coropenclosed}
The correlators from Theorem~\ref{thmmainlong}
 are part of a consistent system of open-closed correlators for the conformal field theory with monodromy data $\cat{A}$. The boundary object is $F$ while the bulk object is $\bulk(F)$.
	\end{corollary}

	This  means that for any surface $\Sigma$ with parametrized boundary circles and parametrized intervals, i.e.\ an operation in the two-colored open-closed surface operad~\cite[Section~3]{sn}, we obtain vectors \begin{align} \xi_\Sigma^F \in \FAA(\Sigma; \bulk(F),\dots,\bulk(F),F,\dots,F) \ , \label{eqnopenclosed}
\end{align} where all parametrized boundary circles are labeled with the bulk object $\bulk(F)$ while all parametrized intervals are labeled by the boundary object $F$. These vectors are mapping class group invariant and preserved under the gluing along parametrized boundary circles and parametrized boundary intervals. 

\begin{proof}[\slshape Proof of Corollary~\ref{coropenclosed}]
	Essentially, this holds by the construction of the vectors $\lambda^F$. Let us spell this out:  The same methods from the proof of Theorem~\ref{thmsym} based on \cite{envas,microcosm} give us vectors
	\begin{align}
		\lambda_\Sigma^F \in \FAA(\Sigma; \bzero(F),\dots,\bzero(F),F,\dots,F)
	\end{align}
	where $\Sigma$ is an `open-closed' surface all of whose parametrized boundary circles are labeled by $\bzero(F)$ and all of whose parametrized boundary intervals are labeled by $F$. 
	The desired vectors $\xi_\Sigma^F$ from~\eqref{eqnopenclosed} are obtained by projecting all labels $\bzero(F)$ to $\bulk(F)$ while the remaining labels are left unchanged.
	As above, the so-constructed $\xi_\Sigma^F$ are mapping class group invariant because the mapping class group action is natural in the arguments of $\FAA$.
	They are also preserved by the gluing along intervals because the gluing maps are natural in arguments of $\FAA$ not involved in the gluing.
	For the gluing along boundary circles, the argument reduces to the one given in the proof of Theorem~\ref{thmmainlong}.
	\end{proof}

	\begin{remark}[String-net	 methods]
	For rational conformal field theories, there exists a correlator construction using string-nets~\cite{rcftsn,fsy-sn}
	that allows one to obtain correlators directly through graphs drawn on surfaces. 
	In the non-semisimple case, the string-net description of the spaces of conformal blocks is still possible~\cite{sn}, but
	a direct translation
	of \cite{rcftsn} to \cite{sn} does not work at present.
	The main issue is that for a modular category $\cat{A}$ the modular functor for $Z(\cat{A})\simeq \bar{\cat{A}}\boxtimes \cat{A}$ is described through a string-net construction involving just the projective objects of $\cat{A}$ (and this is presently the only version of the string-net construction that recovers the Lyubashenko modular functor for $Z(\cat{A})$). 
	This seems incompatible with the fact that the underlying object of the special symmetric Frobenius algebra that we would like to built correlators from is generally not projective.
	Actually, the open modular functor $\Ao$ is equivalent to the open part of the string-net construction for $\cat{A}$~\cite[Theorem~4.2]{envas}, but in a rather non-trivial way.
	The key advantage of $\Ao$, as modular extension of the cyclic associative algebra $\cat{A}$, 
	is that we can apply the modular microcosm principle from \cite{microcosm}.
	A version of the construction on the string-net side would be highly desirable.
\end{remark} 

\subsection{The classification problem for open-closed correlators\label{secclass}}
For a modular category $\cat{A}$,
Theorem~\ref{thmmainshort} tells us that we get for every special symmetric Frobenius algebra $F\in\cat{A}$ a consistent system of open-closed correlators.
The obvious question is now whether all consistent systems of correlators for the monodromy data $\cat{A}$
are of this form. The answer  to this question generalizes the one given in the semisimple case in \cite{ffrsunique}: The answer is `yes' if we are looking for open-closed correlators in the \emph{strong sense}, which means that we ask for the following conditions
\begin{enumerate}
	\item[(S)] that are non-zero on the sphere, \label{conditionS}
	\item[(O)]
	and if we ask the closed part to be induced by the open part. \label{conditionO}
	\end{enumerate}
While~(S) is clear, let us explain what~(O) means:
If we are given a consistent system of open-closed correlators, we can restrict to the open sector. This open part of the data is equivalent to a symmetric Frobenius algebra $F\in \cat{A}$ by \cite[Theorem~8.3]{microcosm}.
In a canonical way, $F$ induces already vectors  $\precor_\Sigma^F\in \mathfrak{F}_{\bar{\cat{A}}\boxtimes\cat{A}}(\Sigma;\bzero(F),\dots,\bzero(F))$  as we have seen in the proof of
by Theorem~\ref{thmsym} for all surfaces in the closed sector with at least one boundary component per connected component.
A system of  open-closed correlators
being induced by the open part means that for the closed part we use these vectors that the open part already provides, i.e.\ we ask these vectors to be compatible with gluing and then pass to the image of the cylinder idempotent.

Let us remark that we do not mention as assumptions the non-degeneracy of the two-point function of the disk or the sphere. For the sphere, this is built into the definition in \cite{jfcs} for bulk field correlators and in \cite{microcosm} for the open part. We consider this as part of the definition of a consistent system of correlators.

\begin{corollary}\label{corclassification}
	Let $\cat{A}$ be a modular category.
	Then all consistent systems of open-closed correlators for the conformal field theory
	with monodromy data $\cat{A}$
	that are open-closed correlators in the strong sense (satisfying the conditions~(S) and~(O)) arise from special symmetric Frobenius algebras in $\cat{A}$ through the construction of Theorem~\ref{thmmainshort}. 
	\end{corollary}

\begin{proof}
	The preservation of the $\precor^F$ under gluing implies that the symmetric Frobenius algebra $F$ describing the open sector needs to be special: Indeed, if we glue the $\precor^F$ for the disk and the cylinder together, 
	we find a vector in $\FAA(D;\bzero(F))\cong\Ao(D;F)\cong \cat{A}(I,F)$
	that, as a brief calculation using the recipe in \cite[Section~9]{microcosm} shows,
	is the handle element $h$ of $F$ from Remark~\ref{remh}. If the $\precor^F$ are supposed to be preserved by gluing, 
	we need $h$ to be the unit: $h=\eta$. Via~\eqref{eqnh}, this implies that $F$ is special if we ask 
	$\varepsilon\circ\eta \neq 0$, i.e.\ if we impose the correlator for $\mathbb{S}^2$ to be non-zero.
	\end{proof}

The correlator for the torus is the \emph{(internal) character} in the sense of \cite{shimizucf}
of the bulk object $\bulk(F)$,
\begin{align}
	I \ra{ \text{coevaluation}} \bulk(F)^\vee \otimes \bulk(F) \to \int^{X\in\bar{\cat{A}} \boxtimes \cat{A}} X^\vee \otimes X \ ; \label{eqncharacter}
\end{align}
this is a consequence of the sewing constraints, see \cite[Proposition~4.11]{jfcs} or use the gluing property of the underlying ansular correlator \cite[Theorem 10.6]{microcosm}. 

\begin{corollary}
		Under the assumption from Corollary~\ref{corclassification}, the correlator for the torus is always non-zero: $\xi_{\mathbb{T}^2}^F\neq 0$. 
	\end{corollary}
\begin{proof}
The character~\eqref{eqncharacter} is non-zero by \cite[Corollary 4.2]{shimizucf} and because $\bulk(F)\neq 0$ (if $\bulk(F)=0$, condition (S) cannot hold).\end{proof}

\subsection{A comment on model-independence\label{secmodel}}
Theorem~\ref{thmmainshort} and Corollary~\ref{corclassification} treat correlators for the conformal field theory whose monodromy data is a modular category, with the modular functor built following the Lyubashenko construction.
This seems to be a statement that is model-dependent: It relies on a certain prescription for the construction of spaces of conformal blocks.
From a physical perspective, this is a problem, see also the comments in \cite[Section~3.3]{csrcft} and \cite[Section~1.2]{damioliniwoike} on this matter.

But actually, the results of this article are \emph{not} tied to models: They apply to any conformal field theory whose monodromy data forms a modular functor with values in $\Lexf$ 
 that is reflection equivariant relative to a rigid duality in the sense of \cite{reflection}. By \cite[Theorem~5.9 \& 5.11]{reflection} we obtain:

\begin{corollary}
	For any conformal field theory whose monodromy data can be described as modular functor with values in $\Lexf$ 
	 and that is reflection equivariant relative to a rigid duality, the correlator construction and classification from Theorem~\ref{thmmainshort} and Corollary~\ref{corclassification} apply. 
	\end{corollary}

In this formulation, the result does not even mention modular categories anymore. Instead, we just ask for a more geometric condition for the behavior under orientation reversal. \enlargethispage*{0.5cm}
This condition can be used to replace the a priori stronger \cite[Assumption~1]{fspivotal}.

\spaceplease
\part{Consequences and applications}
Having finished the construction part, we will now explore the features of the non-semisimple correlator construction.

\section{Holographic principle}
The correlators constructed above admit a holographic description, i.e.\ they have a three-dimensional origin. This will be worked out in this section.

\subsection{A reminder on module categories}
Let $\cat{A}$ be a finite tensor category. Then it is in particular an algebra 
up to coherent isomorphism, so we may consider the action on some finite category $\cat{M}$ that we then call a \emph{module category} (meaning here \emph{left} module; right module categories can be treated analogously). Of course, the usual axioms for a module action have to be relaxed up to coherent isomorphism, see \cite[Chapter~7]{egno} for details.
Since the action $\act : \cat{A}\boxtimes \cat{M}\to \cat{M}$ will be assumed to be right exact, the functor $-\act m : \cat{A}\to \cat{M}$ is right exact and therefore has a right adjoint (which is automatically left exact) that we denote by $\HOM_\cat{M}(m,-):\cat{M}\to\cat{A}$ and that is determined by natural isomorphisms
\begin{align} \cat{M}(a \act m,n)\cong \cat{A}(a,\HOM_\cat{M}(m,n))\label{eqnadjunction} \end{align} for $m,n\in \cat{M}$ and $a\in \cat{A}$.
The functor
\begin{align}\HOM_\cat{M}(-,-):\cat{M}^\op \boxtimes \cat{M} \to \cat{A} \end{align}
is called the \emph{internal hom} of the $\cat{A}$-module $\cat{M}$ and has a composition
\begin{align}
	\HOM_\cat{M}(m,n) \otimes \HOM_\cat{M}(\ell,m)\to\HOM_\cat{M}(\ell,n) \end{align} for $m,n,\ell \in\cat{M}$.

\subsection{A reminder on factorization homology\label{secfh}}	Let $\cat{A}$ be a finite ribbon category (or more generally a framed $E_2$-algebra).
Then one may define for each surface $\Sigma$
the factorization homology $\int_\Sigma \cat{A}$.
This concept was introduced in \cite{higheralgebra,AF}  inspired by \cite{bdca}.
The category $\int_\Sigma \cat{A}$ is obtained by `integrating' $\cat{A}$ over $\Sigma$.
Formally, it is obtained as the homotopy colimit of all $\cat{A}^{\boxtimes m}$ indexed by all oriented embeddings 
$\varphi : \sqcup_{m} \mathbb{D}^2 \to \Sigma$  of $m$ disks into $\Sigma$.
For an introduction tailored towards quantum algebra applications, we refer to \cite[Section~1.1]{bzbj}.
By \cite[Section~5.1]{bzbj} the embedding $\emptyset \to \Sigma$ induces an object $\qss \in \int_\Sigma \cat{A}$
that is a homotopy fixed point under oriented diffeomorphisms of $\Sigma$. This object is called the \emph{quantum structure sheaf}.
One calls its endomorphism algebra
\begin{align} \SkAlg_\cat{A}(\Sigma):= \End_{\int_\Sigma \cat{A}}(\qss) \label{eqnskeinalgebras}
\end{align}  \emph{skein algebra} of $\cat{A}$ at $\Sigma$, see \cite[Definition~2.6]{skeinfin}.
For a ribbon fusion category, this recovers the classical definition of the skein algebra \cite{cooke}.

\subsection{The factorization homology description of the open modular functor of a finite ribbon category}
The open modular functor $\Ao$ associated to a pivotal finite tensor $\cat{A}$ category can be described explicitly using ribbon graphs, see \cite[Eq.~(4.1)]{envas} and \cite[Section~9]{microcosm}. 
Here we will prove that it has a  factorization homology description in the case $\cat{A}$ is actually ribbon.
To this end, we will denote for a surface $\Sigma$ with $n\ge 0$ marked boundary intervals the induced $\cat{A}^{\boxtimes n}$-action on factorization homology \cite[Section~5.2]{bzbj} by
$\act : \cat{A}^{\boxtimes n} \boxtimes \int_\Sigma \cat{A}\to\int_\Sigma \cat{A}$.

\begin{theorem}\label{thmao}
	Let $\cat{A}$ be a finite ribbon category.
	For any surface $\Sigma$ with at least one boundary component per connected component and $n\ge 0$ intervals embedded in its boundary, there is a canonical isomorphism
	\begin{align}
		\label{eqnAo}	\Ao(\Sigma;X_1,\dots,X_n) \cong \Hom_{\int_\Sigma \cat{A}} \left(    \qss , (X_1 \boxtimes \dots \boxtimes X_n)\act\qss  \right)     
	\end{align} that is mapping class group equivariant and compatible with the gluing along intervals.
In fact, \eqref{eqnAo} can be seen as an equivalence of $\Lexf$-valued open modular functors.
\end{theorem}

\begin{proof}
	The strategy is to prove that the right hand side of~\eqref{eqnAo}
	gives us a $\Lexf$-valued open modular functor and to argue then via uniqueness. We begin by observing that the
	 functor 
	$\Hom_{\int_\Sigma \cat{A}} \left(   \qss, - \act \qss  \right)  : \cat{A}^{\boxtimes n} \to \vect$  carries a representation of the mapping class group of $\Sigma$. Indeed, any diffeomorphism $f: \Sigma \to \Sigma$ (for which we always assume that it preserves the orientation and the boundary parametrization) acts by an $\cat{A}^{\boxtimes n}$-module map $f_* : \int_\Sigma \cat{A}\to\int_\Sigma \cat{A}$. But then $ (X_1 \boxtimes \dots \boxtimes X_n) \act \qss$ is a homotopy fixed point because $\qss$ is already a homotopy fixed point, as was recalled in Section~\ref{secfh}.
	This endows $\Hom_{\int_\Sigma \cat{A}} \left( \qss, (X_1 \boxtimes \dots \boxtimes X_n) \act \qss  \right)$ with an action of the diffeomorphism group of $\Sigma$ (as topological group)
	 that clearly descends to the mapping class group because the action is on an object in a 1-category.

	Next we treat the compatibility of the right hand side of~\eqref{eqnAo} 
	 under gluing: Fix two marked boundary intervals for $\Sigma$ and denote by $\Sigma'$ the surface obtained by gluing these two intervals together; this requires $n\ge 2$ of course. 
	By the excision property for factorization homology, $\int_{\Sigma'} \cat{A}$ is the `trace category' of the 
	$\cat{A}$-bimodule $\int_\Sigma \cat{A}$. After seeing an $\cat{A}$-bimodule as module over the enveloping algebra $\cat{A}^\catf{e} = \cat{A}^{\otimes \op} \boxtimes \cat{A}$
	(here $\cat{A}^{\otimes \op}$ is the same category as $\cat{A}$,
	but with monoidal product $X \stackrel{\text{op}}{\otimes} Y=Y\otimes X$), this means that the gluing provides us with an equivalence
	$G:\cat{A}\boxtimes_{\cat{A}^\catf{e}}      \int_\Sigma \cat{A} \ra{\simeq}\int_{\Sigma'} \cat{A}$. 
	This equivalence is compatible with the diffeomorphism group actions, the remaining $\cat{A}^{\boxtimes (n-2)}$-action, and sends $I\boxtimes \qss$ to $\qssp$.
	Therefore, $G$ induces an isomorphism
	\begin{align}
		&	\Hom_{\int_{\Sigma'} \cat{A}} \left( \qssp, (X_1 \boxtimes \dots \boxtimes X_{n-2}) \act \qssp   \right)\\ \cong \quad &\Hom_{       \cat{A}\boxtimes_{\cat{A}^\catf{e}}      \int_\Sigma \cat{A}    }   (  I\boxtimes \qss, I\boxtimes (X_1 \boxtimes \dots \boxtimes I \boxtimes \dots \boxtimes I \boxtimes \dots \boxtimes X_{n-2})\act \qss)    \ , 
	\end{align}
	where the appearance of monoidal units in $X_1 \boxtimes \dots \boxtimes I \boxtimes \dots \boxtimes I \boxtimes \dots \boxtimes X_{n-2}$ is such  that the slots affected by the gluing are filled with $I$.
	With \cite[Theorem 3.3~(3)]{dss}, this gives us 
	an 	isomorphism
	\begin{align}
		&	\Hom_{\int_{\Sigma'} \cat{A}} \left( \qssp, (X_1 \boxtimes \dots \boxtimes X_{n-2}) \act \qssp    \right) \\ \cong \quad &\Hom_{      \cat{A}^\catf{e}    }   ( I\boxtimes I,   \END_\cat{A}(I) \otimes    \HOM_{\int_\Sigma \cat{A}}( \qss,	(X_1 \boxtimes \dots \boxtimes I \boxtimes \dots \boxtimes I \boxtimes \dots \boxtimes X_{n-2})\act \qss ) )  \, \label{eqnreins}
	\end{align} with the respective $\cat{A}^{\catf{e}}$-valued internal homs.
	By \cite[Lemma 4.4]{shimizuunimodular} $\END_\cat{A}(I) \cong \int^{X \in \cat{A}} X^\vee \boxtimes X$
	 is the coevaluation object $\Delta$ for the cyclic algebra $\cat{A}$, see \eqref{eqncoev}. We write as usual $\Delta = \Delta' \boxtimes \Delta''$ in Sweedler notation. 
	
	Denote the $\cat{A}^\text{e}$-action by $\unrhd$.
	Since internal homs are module functors, see~\cite[Section~3.2]{ostrik03} and \cite[Section~2.8]{shimizuunimodular},
	\begin{align} &\END_\cat{A}(I) \otimes     \HOM_{\int_\Sigma \cat{A}}(\qss, (X_1 \boxtimes \dots \boxtimes I \boxtimes \dots \boxtimes I \boxtimes \dots \boxtimes X_{n-2})\act \qss)\\ \cong\quad&    
		 \HOM_{\int_\Sigma \cat{A}}(\qss,\END_\cat{A}(I) \unrhd (X_1 \boxtimes \dots \boxtimes I \boxtimes \dots \boxtimes I \boxtimes \dots \boxtimes X_{n-2})\act \qss)
		 \\ \cong\quad&    
		 \HOM_{\int_\Sigma \cat{A}}(\qss, (X_1 \boxtimes \dots \boxtimes \Delta' \boxtimes \dots \boxtimes \Delta'' \boxtimes \dots \boxtimes X_{n-2})\act \qss) \ . 
	\end{align}
	If we re-insert in~\eqref{eqnreins}, we arrive at
	\begin{align}
		&\Hom_{\int_{\Sigma'} \cat{A}} \left( \qssp, (X_1 \boxtimes \dots \boxtimes X_{n-2}) \act \qssp    \right) \\ \cong \quad &\Hom_{      \cat{A}^\catf{e}    }   ( I\boxtimes I, \HOM_{\int_\Sigma \cat{A}}(\qss, (X_1 \boxtimes \dots \boxtimes \Delta' \boxtimes \dots \boxtimes \Delta'' \boxtimes \dots \boxtimes X_{n-2})\act \qss) )  \\
		\cong \quad &
		\Hom_{   \int_\Sigma \cat{A}    }   ( \qss, (X_1 \boxtimes \dots \boxtimes \Delta' \boxtimes \dots \boxtimes \Delta'' \boxtimes \dots \boxtimes X_{n-2})\act\qss)	\ . 
	\end{align}
	Consider now the case  in which $\Sigma$ is a disk $D$:
	Then $\int_D \cat{A}\simeq \cat{A}$, $\cat{O}_D^\cat{A}=I$ and $ (X_1 \boxtimes \dots \boxtimes X_n) \act I=X_1 \otimes \dots \otimes X_n$. Therefore,
	\begin{align} 
		\Hom_{\int_\Sigma \cat{A}} \left(  \qss, (X_1 \boxtimes \dots \boxtimes X_n) \act \qss   \right)=\cat{A}(I,X_1\otimes\dots\otimes X_n) \ , \label{eqnrecoverA} \end{align} with the cyclic symmetry being the usual one.
	From this and the gluing property, we conclude that the right hand side 
	of~\eqref{eqnAo} is left exact, and from the straightforward compatibility of the mapping class group action and the gluing, we conclude  that we obtain actually a $\Lexf$-valued open modular functor whose restriction to marked disks recovers the pivotal finite tensor category $\cat{A}$ thanks to~\eqref{eqnrecoverA}.
	Now the statement follows by uniqueness~\cite[Theorem~2.2]{envas} and~\eqref{eqndisk}.
\end{proof}

\subsection{A reminder on moduli algebras\label{secmodulialgebras}}
Let $\Sigma$ be a surface with $n$ marked boundary intervals.
Recall that in this situation,
$\int_\Sigma \cat{A}$ becomes
a module over $\cat{A}^{\boxtimes n}$  and gives us a \emph{moduli algebra}
\begin{align}
	\mathfrak{a}_\Sigma := \END_{\int_\Sigma \cat{A}} (\qss) \in \cat{A}^{\boxtimes n}  
\end{align} defined as usual as the internal endomorphism algebra of the quantum structure sheaf, see \cite[Section~5.2]{bzbj}, where the relation  to the works \cite{alekseevmoduli,agsmoduli,asmoduli} is explained.

The algebra $\moduli_\Sigma$ carries a mapping class group action~\cite[Section~5.2]{bzbj}, and one can easily observe that this makes the isomorphisms
\begin{align}
	\label{eqnmorequiv}	\Hom_{\int_\Sigma \cat{A}} (\qss,X\act \qss ) \cong \Hom_{\cat{A}^{\boxtimes n} } (  I, X\otimes  \moduli_\Sigma     )\cong \Hom_{\cat{A}^{\boxtimes n} } (  X^\vee,   \moduli_\Sigma     )
\end{align} for $X\in       \cat{A}^{\boxtimes n}$ mapping class group equivariant, where the origin of the mapping class group representation on the left hand side is again the one explained in the proof of Theorem~\ref{thmao}. 

\subsection{Holographic interpretation of the correlator construction\label{secintholprinciple}}
Let $F$ be a special symmetric Frobenius algebra in a modular category $\cat{A}$.
Moreover, let $\Sigma$ be a surface with at least one boundary component per connected component.
As in the proof of Theorem~\ref{thmsym}, we will see it as an operation in the open surface operad. 
Denote as 
in Section~\ref{secholographicintro} of the introduction by $\bar \Sigma \cupx_{\partial \Sigma} \Sigma$
the result of gluing $\Sigma$ and $\bar\Sigma$ together along spheres with three holes, see Figure~\ref{figdouble2} for an example.

\begin{figure}[h]
	\begin{tikzpicture}[scale=0.5]
		\begin{pgfonlayer}{nodelayer}
			\node [style=none] (0) at (3.75, -0.25) {};
			\node [style=none] (1) at (4.75, -0.25) {};
			\node [style=none] (2) at (3.5, 0.25) {};
			\node [style=none] (3) at (5, 0.25) {};
			\node [style=none] (4) at (1.5, 2) {};
			\node [style=none] (5) at (2.5, 0.5) {};
			\node [style=none] (6) at (1.5, -2) {};
			\node [style=none] (7) at (-1.5, 2) {};
			\node [style=none] (8) at (-2.5, 0.5) {};
			\node [style=none] (9) at (-1.5, -2) {};
			\node [style=none] (10) at (-4.75, -0.25) {};
			\node [style=none] (11) at (-3.75, -0.25) {};
			\node [style=none] (12) at (-5, 0.25) {};
			\node [style=none] (13) at (-3.5, 0.25) {};
			\node [style=none] (14) at (-1.5, -0.25) {};
			\node [style=none] (15) at (1.5, -0.25) {};
			\node [style=none] (16) at (-0.5, -1.25) {};
			\node [style=none] (17) at (0.5, -1.25) {};
			\node [style=none] (19) at (-0.5, 1.25) {};
			\node [style=none] (20) at (0.5, 1.25) {};
			\node [style=none] (28) at (-18, 2) {};
			\node [style=none] (29) at (-18, 0.75) {};
			\node [style=none] (30) at (-18, -2) {};
			\node [style=none] (31) at (-21.25, -0.25) {};
			\node [style=none] (32) at (-20.25, -0.25) {};
			\node [style=none] (33) at (-21.5, 0.25) {};
			\node [style=none] (34) at (-20, 0.25) {};
			\node [style=none] (35) at (-18, -0.75) {};
			\node [style=none] (36) at (-15, 2) {};
			\node [style=none] (37) at (-15, 0.75) {};
			\node [style=none] (38) at (-15, -2) {};
			\node [style=none] (39) at (-12.75, -0.25) {};
			\node [style=none] (40) at (-11.75, -0.25) {};
			\node [style=none] (41) at (-13, 0.25) {};
			\node [style=none] (42) at (-11.5, 0.25) {};
			\node [style=none] (43) at (-15, -0.75) {};
			\node [style=none] (44) at (-14.75, 6.5) {};
			\node [style=none] (45) at (-14.75, 5) {};
			\node [style=none] (46) at (-17.75, 6.5) {};
			\node [style=none] (47) at (-17.75, 5) {};
			\node [style=none] (50) at (-16.75, 5.75) {};
			\node [style=none] (51) at (-15.75, 5.75) {};
			\node [style=none] (52) at (-16.5, 4.5) {};
			\node [style=none] (53) at (-16.5, 1.5) {};
			\node [style=none] (54) at (-16.5, -1.5) {};
			\node [style=none] (55) at (-14.75, 3.25) {glue in};
			\node [style=none] (56) at (-9, 0) {};
			\node [style=none] (57) at (-7.5, 0) {};
			\node [style=none] (58) at (-13, -1) {$\Sigma$};
			\node [style=none] (59) at (-20, -1) {$\bar\Sigma$};
			\node [style=none] (60) at (2.5, -1.25) {$\bar \Sigma \cupx_{\partial \Sigma} \Sigma$};
		\end{pgfonlayer}
		\begin{pgfonlayer}{edgelayer}
			\draw [bend left=90, looseness=1.25] (0.center) to (1.center);
			\draw [bend right=90, looseness=1.50] (2.center) to (3.center);
			\draw [bend right=90, looseness=4.50] (6.center) to (4.center);
			\draw [in=180, out=-180, looseness=4.50] (7.center) to (9.center);
			\draw [bend left=90, looseness=1.25] (10.center) to (11.center);
			\draw [bend right=90, looseness=1.50] (12.center) to (13.center);
			\draw (7.center) to (4.center);
			\draw (9.center) to (6.center);
			\draw [bend right=90, looseness=0.75] (8.center) to (5.center);
			\draw [bend left=90, looseness=0.75] (14.center) to (15.center);
			\draw [bend left=90] (16.center) to (17.center);
			\draw [bend right=90] (16.center) to (17.center);
			\draw [bend left=90] (19.center) to (20.center);
			\draw [bend right=90] (19.center) to (20.center);
			\draw [in=180, out=-180, looseness=4.25] (28.center) to (30.center);
			\draw [bend left=90, looseness=1.25] (31.center) to (32.center);
			\draw [bend right=90, looseness=1.50] (33.center) to (34.center);
			\draw [bend right=90, looseness=2.00] (29.center) to (35.center);
			\draw [bend left=90, looseness=1.50] (35.center) to (30.center);
			\draw [bend left=90, looseness=1.50] (28.center) to (29.center);
			\draw [in=0, out=0, looseness=4.25] (36.center) to (38.center);
			\draw [bend left=90, looseness=1.25] (39.center) to (40.center);
			\draw [bend right=90, looseness=1.50] (41.center) to (42.center);
			\draw [bend left=90, looseness=2.00] (37.center) to (43.center);
			\draw [bend right=90, looseness=1.75] (36.center) to (37.center);
			\draw [bend right=90, looseness=1.75] (43.center) to (38.center);
			\draw [bend left=90, looseness=1.50] (43.center) to (38.center);
			\draw [bend left=90, looseness=1.50] (36.center) to (37.center);
			\draw [style=mydotsblack, bend right=90, looseness=1.75] (28.center) to (29.center);
			\draw [style=mydotsblack, bend right=90, looseness=1.50] (35.center) to (30.center);
			\draw [bend left=90, looseness=1.50] (44.center) to (45.center);
			\draw [bend right=90, looseness=1.75] (46.center) to (47.center);
			\draw [bend left=90, looseness=1.50] (46.center) to (47.center);
			\draw [style=mydotsblack, bend right=90, looseness=1.75] (44.center) to (45.center);
			\draw (46.center) to (44.center);
			\draw (47.center) to (45.center);
			\draw [bend left=90] (50.center) to (51.center);
			\draw [bend right=90] (50.center) to (51.center);
			\draw [style=end arrow, bend right=15] (52.center) to (53.center);
			\draw [style=end arrow, bend left=15] (52.center) to (54.center);
			\draw [style=end arrow] (56.center) to (57.center);
		\end{pgfonlayer}
	\end{tikzpicture}
	\caption{The construction of $\bar \Sigma \cupx_{\partial \Sigma} \Sigma$ for a genus one surface with two boundary components.}
	\label{figdouble2}
\end{figure}
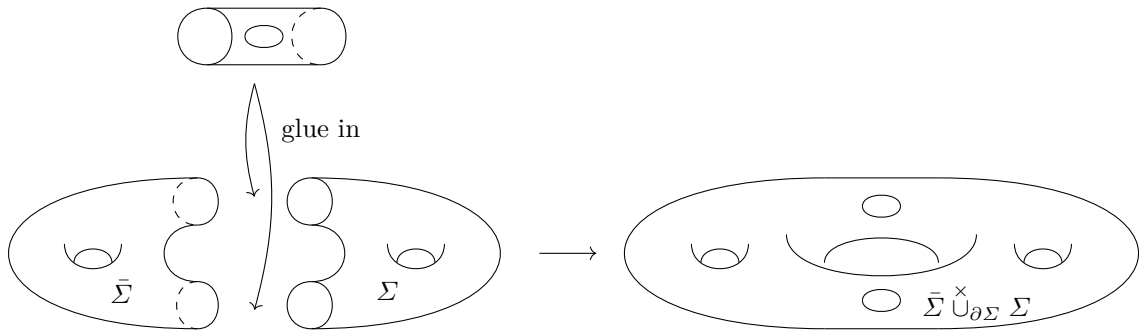

With excision for spaces of conformal blocks~\eqref{eqnexcision}, we 
obtain
\begin{align} \mathfrak{F}_{\bar{\cat{A}}\boxtimes\cat{A}}(\Sigma;\bzero(F),\dots,\bzero(F))\cong \FA(     \bar{\Sigma} \cupx_{\partial \Sigma} \Sigma;F^{\boxtimes n}    )\ ,  \end{align}
and we may see the vectors
\begin{align} \lambda_\Sigma^F \in \mathfrak{F}_{\bar{\cat{A}}\boxtimes\cat{A}}(\Sigma;\bzero(F),\dots,\bzero(F))    \label{eqnlambdas}\end{align}
from Theorem~\ref{thmsym},
that the correlators featuring in Theorem~\ref{thmmainlong} were built from,
 as vectors 
\begin{align} \lambda_\Sigma^F \in 	\FA(     \bar{\Sigma} \cupx_{\partial \Sigma} \Sigma;F^{\boxtimes n}    ) \  \end{align}
that we denote by the same symbol by a slight abuse of notation.
By combining~\eqref{eqnAoZ} and Theorem~\ref{thmao},
we obtain a canonical isomorphism
\begin{align}
	\FA(     \bar{\Sigma} \cupx_{\partial \Sigma} \Sigma;F^{\boxtimes n}    )\cong \Hom_{\int_\Sigma \cat{A}}(\qss,F^{\boxtimes n} \act \qss) 
	\end{align}
The right hand side can be rewritten using the admissible skein modules $\skA$  of $\cat{A}$ on a three-dimensional manifold which are defined in \cite{asm} by applying generalized skein-theoretic methods to the tensor ideal $\Proj \cat{A}$ of projective objects. 
The skein modules $\skA(\Sigma\times[0,1];-)$ with appropriate boundary conditions describe the morphism spaces in the skein category for $\cat{A}$ on $\Sigma$ whose finite free cocompletion is $\int_\Sigma \cat{A}$~\cite[Theorem~3.10]{brownhaioun}.
The morphism spaces of $\int_\Sigma \cat{A}$ can be described by  the finitely cocompleted version   $\SkA$ of the admissible skein modules~\cite{brownhaioun,mwskein}. 
This means that $\Hom_{\int_\Sigma \cat{A}}(\qss,F^{\boxtimes n} \act \qss)$ is the value of $\SkA$ 
on $\Sigma \times [0,1]$ with $F$ placed $n$ times on a disk near the boundary components as the target boundary condition
and the empty initial boundary condition.
Let us illustrate this for the example in Figure~\ref{figdouble2}:
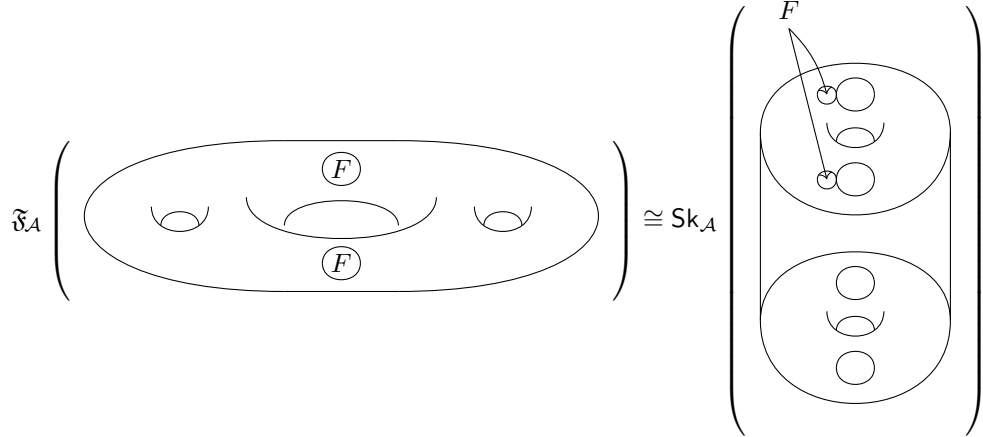
\begin{align}
	\FA\left(\!\!\!\!\!\!\!\!\!\!\!\!\!\!\! \begin{array}{c}\begin{tikzpicture}[scale=0.5]
			\begin{pgfonlayer}{nodelayer}
				\node [style=none] (0) at (3.75, -0.25) {};
				\node [style=none] (1) at (4.75, -0.25) {};
				\node [style=none] (2) at (3.5, 0.25) {};
				\node [style=none] (3) at (5, 0.25) {};
				\node [style=none] (4) at (1.5, 2) {};
				\node [style=none] (5) at (2.5, 0.5) {};
				\node [style=none] (6) at (1.5, -2) {};
				\node [style=none] (7) at (-1.5, 2) {};
				\node [style=none] (8) at (-2.5, 0.5) {};
				\node [style=none] (9) at (-1.5, -2) {};
				\node [style=none] (10) at (-4.75, -0.25) {};
				\node [style=none] (11) at (-3.75, -0.25) {};
				\node [style=none] (12) at (-5, 0.25) {};
				\node [style=none] (13) at (-3.5, 0.25) {};
				\node [style=none] (14) at (-1.5, -0.25) {};
				\node [style=none] (15) at (1.5, -0.25) {};
				\node [style=none] (16) at (-0.5, -1.25) {};
				\node [style=none] (17) at (0.5, -1.25) {};
				\node [style=none] (18) at (0, -1.25) {$F$};
				\node [style=none] (19) at (-0.5, 1.25) {};
				\node [style=none] (20) at (0.5, 1.25) {};
				\node [style=none] (21) at (0, 1.25) {$F$};
			\end{pgfonlayer}
			\begin{pgfonlayer}{edgelayer}
				\draw [bend left=90, looseness=1.25] (0.center) to (1.center);
				\draw [bend right=90, looseness=1.50] (2.center) to (3.center);
				\draw [bend right=90, looseness=4.50] (6.center) to (4.center);
				\draw [in=180, out=-180, looseness=4.50] (7.center) to (9.center);
				\draw [bend left=90, looseness=1.25] (10.center) to (11.center);
				\draw [bend right=90, looseness=1.50] (12.center) to (13.center);
				\draw (7.center) to (4.center);
				\draw (9.center) to (6.center);
				\draw [bend right=90, looseness=0.75] (8.center) to (5.center);
				\draw [bend left=90, looseness=0.75] (14.center) to (15.center);
				\draw [bend left=90, looseness=1.50] (16.center) to (17.center);
				\draw [bend right=90, looseness=1.50] (16.center) to (17.center);
				\draw [bend left=90, looseness=1.50] (19.center) to (20.center);
				\draw [bend right=90, looseness=1.50] (19.center) to (20.center);
			\end{pgfonlayer}
		\end{tikzpicture}
		\end{array}\!\!\!\!\!\!\!\!\!\!\!\!\!\!\! \right) 
	\cong
	\catf{Sk}_\cat{A} \left( \begin{array}{c}\begin{tikzpicture}[scale=0.5]
			\begin{pgfonlayer}{nodelayer}
				\node [style=none] (0) at (-2.5, -2) {};
				\node [style=none] (1) at (2.5, -2) {};
				\node [style=none] (2) at (-0.5, -2.25) {};
				\node [style=none] (3) at (0.5, -2.25) {};
				\node [style=none] (4) at (-0.75, -1.75) {};
				\node [style=none] (5) at (0.75, -1.75) {};
				\node [style=none] (6) at (-2.5, 3) {};
				\node [style=none] (7) at (2.5, 3) {};
				\node [style=none] (8) at (-0.5, 2.75) {};
				\node [style=none] (9) at (0.5, 2.75) {};
				\node [style=none] (10) at (-0.75, 3.25) {};
				\node [style=none] (11) at (0.75, 3.25) {};
				\node [style=none] (12) at (-0.5, -1) {};
				\node [style=none] (13) at (0.5, -1) {};
				\node [style=none] (14) at (-1.75, 6.25) {$F$};
				\node [style=none] (15) at (-0.5, -3.25) {};
				\node [style=none] (16) at (0.5, -3.25) {};
				\node [style=none] (18) at (-0.5, 4) {};
				\node [style=none] (19) at (0.5, 4) {};
				\node [style=none] (21) at (-0.5, 1.75) {};
				\node [style=none] (22) at (0.5, 1.75) {};
				\node [style=none] (27) at (0, -1.5) {};
				\node [style=none] (28) at (-1, 1.75) {};
				\node [style=none] (29) at (-0.5, 1.75) {};
				\node [style=none] (30) at (-1, 4) {};
				\node [style=none] (31) at (-0.5, 4) {};
				\node [style=none] (32) at (-1.75, 5.75) {};
				\node [style=none] (33) at (-0.75, 4) {};
				\node [style=none] (34) at (-0.75, 1.75) {};
			\end{pgfonlayer}
			\begin{pgfonlayer}{edgelayer}
				\draw [bend right=90, looseness=1.50] (0.center) to (1.center);
				\draw [bend left=90, looseness=1.25] (0.center) to (1.center);
				\draw [bend left=90, looseness=1.25] (2.center) to (3.center);
				\draw [bend right=90, looseness=1.50] (4.center) to (5.center);
				\draw [bend right=90, looseness=1.50] (6.center) to (7.center);
				\draw [bend left=90, looseness=1.25] (6.center) to (7.center);
				\draw [bend left=90, looseness=1.25] (8.center) to (9.center);
				\draw [bend right=90, looseness=1.50] (10.center) to (11.center);
				\draw (6.center) to (0.center);
				\draw (7.center) to (1.center);
				\draw [bend left=90, looseness=1.50] (12.center) to (13.center);
				\draw [bend right=90, looseness=1.50] (12.center) to (13.center);
				\draw [bend left=90, looseness=1.50] (15.center) to (16.center);
				\draw [bend right=90, looseness=1.50] (15.center) to (16.center);
				\draw [bend left=90, looseness=1.50] (18.center) to (19.center);
				\draw [bend right=90, looseness=1.50] (18.center) to (19.center);
				\draw [bend left=90, looseness=1.50] (21.center) to (22.center);
				\draw [bend right=90, looseness=1.50] (21.center) to (22.center);
				\draw [bend left=90, looseness=1.50] (28.center) to (29.center);
				\draw [bend right=90, looseness=1.75] (28.center) to (29.center);
				\draw [bend left=90, looseness=1.50] (30.center) to (31.center);
				\draw [bend right=90, looseness=1.75] (30.center) to (31.center);
				\draw [style=end arrow] (32.center) to (34.center);
				\draw [style=end arrow, bend left=15] (32.center) to (33.center);
			\end{pgfonlayer}
		\end{tikzpicture}
		\end{array}\right) \label{eqnholprin}
\end{align}
By summarizing all this
and taking additionally into account the description
in terms of moduli algebras from Section~\ref{secmodulialgebras}, we find the following:

\begin{corollary}[Holographic principle]\label{corholp}
	The vectors~\eqref{eqnlambdas} inducing the correlators can equivalently be seen
	\begin{itemize}
		\item as morphisms $\qss \to F^{\boxtimes n}\act \qss$ in factorization homology $\int_\Sigma \cat{A}$,
		\item or as morphisms $F^{\boxtimes n}\to \moduli_\Sigma$ in $\cat{A}^{\boxtimes n}$
		\end{itemize}
	that are fixed by the respective actions of the mapping class group of $\Sigma$. They amount to vectors inside the cocompleted admissible skein module for the cylinder over $\Sigma$.
	\end{corollary}

This tells us that the correlators for the double of $\Sigma$ originate
from dimension three because they can be seen as morphisms in factorization homology. 
The latter is a category assigned in dimension two, 
but its morphisms have a three-dimensional interpretation.

\begin{remark}[A holographic principle involving the four-dimensional (generalized) Crane-Yetter-Kauffman topological field theory]	
Corollary~\ref{corholp} a priori does not use the language of topological field theory. Nonetheless, the admissible skein modules for cylinders over surfaces belong to the three-dimensional
part of the four-dimensional skein topological field theory constructed in \cite{skeintft} generalizing the work of Crane-Yetter-Kauffman~\cite{CY,CYK}.
This means that the holographic principle afforded by Corollary~\ref{corholp} can actually be phrased in terms of topological field theory, but more naturally a four-dimensional one.
Thanks to \cite{wrtcy}, the holographic principle from Corollary~\ref{corholp}  necessarily
implies a similar principle in terms of the non-compact topological field theory from \cite{rggpmr}.
However, we do not claim that it is obvious how a possible translation of Corollary~\ref{corholp} into the framework of \cite{rggpmr,skeintft,wrtcy} would look like (and also, we do not need it).
In any case, \cite{mwskein} would be the dictionary to be used.
\end{remark}

\subsection{Class functions\label{secclassfunctions}}
The following considerations offer a slightly different perspective on Corollary~\ref{corholp} that will help us with the computations in Section~\ref{sectorus}:
Let $\Sigma$ be a connected surface without marked intervals in the boundary, 
but at least one boundary component (without parametrization). 
For a symmetric Frobenius algebra $F\in \cat{A}$ in a finite ribbon category, \eqref{omicrocosmeqn} gives via Theorem~\ref{thmao} a
$\Map(\Sigma)$-invariant map
$\lambda_\Sigma^F:\qss\to\qss$, i.e.\ a mapping class group invariant element of the skein algebra $\SkAlg_\cat{A}(\Sigma)$. Now fix an interval on one of the boundary components and denote the resulting surface by $\Sigma_+$. The surfaces $\Sigma$ and $\Sigma_+$ have the same factorization homology and quantum structure sheaf (the additional marked interval is just relevant for the module structure).
This tells us that $\lambda_\Sigma^F$ gives us a $\Map(\Sigma_+)$-invariant map $I \to \mathfrak{a}_{\Sigma_+}$ with $\mathfrak{a}_{\Sigma_+} = \END(\cat{O}_{\Sigma_+}^\cat{A})\in\cat{A}$. 
 Actually, the sewing constraints for the open correlator $\lambda^F$ tell us that this map
 is the open correlator 
	$\lambda_{\Sigma_+}^F:\cat{O}_{\Sigma_+}^\cat{A} \to F \act \cat{O}_{\Sigma_+}^\cat{A}$ after post-composing with the counit and using~\eqref{eqnmorequiv}. 
	
	If $\Sigma$ has genus $g$ and $r\ge 1$ unparametrized boundary components, then \cite[Corollary~6.11]{bzbj} tells us
	$\mathfrak{a}_{\Sigma_+}\cong \mathbb{F}^{\otimes (2g+r-1)}$ in $\cat{A}$ for the coend $\mathbb{F}=\int^{X\in\cat{A}} X^\vee \otimes X$ which allows us to see the map
	$I \to \mathfrak{a}_{\Sigma_+}$ just constructed
	as a morphism
	\begin{align}
		\zeta^F_{g,r}: \mathbb{A}^{\otimes (2g+r-1)} \to I 
		\end{align} for the end $\mathbb{A}=\int_{X\in\cat{A}} X^\vee \otimes X$. 
	Under the action of $\Map(\Sigma_+)$ on $\mathbb{A}^{\otimes (2g+r-1)}$ from \cite[Section~5.2]{bzbj} (see \cite{brochierjordane} for the case of the torus),
	the map $\zeta^F_{g,r}$ is invariant. 
	
	We refer to maps of tensor powers of $\mathbb{A}$ to $I$ as \emph{(generalized) class functions}; the case of maps $\mathbb{A}\to I$ recovers the classical notion \cite{shimizucf}. Indeed, if $\cat{A}$ is given by finite-dimensional modules over a finite-dimensional ribbon Hopf algebra $H$, we have $\mathbb{A}\cong H_\text{adj}$ (this is $H$ with the adjoint action)~\cite[Theorem 7.4.13]{kl}, which means that $H$-module maps $H_\text{adj}\to k$ are really class functions in the usual sense, i.e.\ linear forms $f:H\to k$ with $f(x_{(1)}yS(x_{(2)}))=\varepsilon(x)f(y)$ for $x,y\in H$, the antipode $S:H\to H$, the counit $\varepsilon:H\to k$ and Sweedler notation $\Delta (x)=x_{(1)}\otimes x_{(2)}$ for the coproduct of $H$. 
In the special case $g=r=1$, we will refer to a class function $\mathbb{A}^{\otimes 2}\to I$ as an \emph{elliptic class function}
because it points at an element in the elliptic double~\cite{brochierjordane}. 	

Let us summarize:

\begin{corollary}\label{corclassfunctions}
	Any symmetric Frobenius algebra $F\in\cat{A}$ in a finite ribbon category	 gives canonically rise to
mapping class group invariant class functions
	\begin{align}
		\zeta^F_{g,r}: \mathbb{A}^{\otimes (2g+r-1)}\to I\ ,  \quad r\ge 1
		\end{align} at all genera $g$ provided that we have $r\ge 1$ punctures.
\end{corollary}

\begin{example}\label{exellclassfunction}
From \cite[Section~9]{microcosm} and the description of the factorization homology of the torus with one boundary component via the elliptic double \cite[Section~6.5]{bzbj}, see also \cite{brochierjordane}, it follows that 
the elliptic class function $\zeta^F_{1,1}:\mathbb{A}^{\otimes 2} \to I$ is dual to the
 map $I \to \cat{D}_\cat{A}$ is given by
\begin{align}
		I \ra{v_4} F^{\otimes 4} \ra{\id \otimes c_{F,F}^{-1} \otimes \id} F^{\otimes 4} \ra{ \text{self-duality}} (F^\vee \otimes F)^{\otimes 2} \ra{\substack{\text{structure map} \\ \text{of the coend}}} \mathbb{F}^{\otimes 2}=\cat{D}_\cat{A} \ .     \label{eqninvinD}
\end{align}
\end{example}

	\section{Torus partition function\label{sectorus}}
In this section,
we demonstrate that the correlators from Theorem~\ref{thmmainlong} are  explicitly computable by extracting the torus partition function.

	\subsection{Preparation: Traces on factorization homology of closed surfaces\label{sectraces}}
The results on the factorization homology of modular categories obtained in \cite{brochierwoike,reflection}
can be used to define traces for endomorphisms in factorization homology.

Let $\cat{A}$ be a finite ribbon category and $\Sigma$ a closed surface.
If $P\in \int_\Sigma \cat{A}$ is projective, the image of an endomorphism $f$ of $P$ under
$\End_{\int_\Sigma \cat{A}}(P) \to HH_0(\int_\Sigma \cat{A})=\int^{P\in\Proj \int_\Sigma \cat{A}} \End_{\int_\Sigma \cat{A}}(P)$ is the \emph{Hattori-Stallings trace}~\cite{hattori,stallings}
that we denote by $\trace_\catf{HS}(f) \in HH_0(\int_\Sigma \cat{A})$. 
Recall that $HH_0$ is the zeroth Hochschild homology: For a finite linear category $\cat{C}$, it is  defined via $HH_0(\cat{C}) = \int^{P\in \Proj \cat{C}} \cat{C}(P,P)$; for more on Hochschild homology in the context of quantum algebra, see e.g.~\cite[Section~2]{dva}.

Using the $\Rexf$-valued ansular functor $\widehat{\cat{A}}$ associated to $\cat{A}$~\cite{cyclic,mwansular}
($\Rexf$ is obtained from $\Lexf$ by replacing left exact functors with right exact ones), it is proved in \cite[Section~4]{brochierwoike}
that for a handlebody $H$ with $n$ embedded disks and boundary surface $\Sigma=\partial H$ with $n$ boundary components (the embedded disks are converted into boundary components of the boundary surface), one may associate a right exact functor
\begin{align}
	\PhiA(H):\int_\Sigma \cat{A}\to \cat{A}^{\boxtimes n} \ , 
\end{align}
the so-called \emph{generalized skein module} for $H$.
By \cite[Theorem~4.2]{brochierwoike} 
\begin{align}
	\PhiA(H)\circ \qss \cong \widehat{\cat{A}}(H)      \label{eqnphiaqss}
\end{align}
which makes $\widehat{\cat{A}}(H)$, seen here as object in $\cat{A}^{\boxtimes n}$ via the cyclic structure,	a module over the skein algebra~\eqref{eqnskeinalgebras}.
This construction agrees with the classical handlebody skein modules in the semisimple case and, after a suitable cocompletion,
 with the admissible skein modules~\cite{asm} 
in the non-semisimple case as is established in \cite{mwskein}.

If $\cat{A}$ is modular, then all functors $\PhiA$ are equivalences~\cite[Corollary~6.6]{reflection}, and
we can use, for a closed surface, the equivalence \begin{align}
	\label{phaequiveqn}	\PhiA(H):\int_\Sigma \cat{A}\ra{\simeq} \vect\end{align} 
for a handlebody $H$ with $\partial H=\Sigma$ to see $\trace_\catf{HS}(f)$ as an element in $k$ under the canonical isomorphism
$HH_0(\vect)=\int^{V \in \vect}V^*\otimes V\cong k$ given by the trace of matrices.
In other words, we identify the Hattori-Stallings trace in the modular case with the number
\begin{align}
	\label{eqntracenumber}	\trace _\catf{HS}(f) = \catf{tr}\,  \PhiA(H)f\in k \quad \text{for} \quad f \in \End_{\int_\Sigma \cat{A}}(X) \ , 
\end{align}
where $\catf{tr}$ is the matrix trace.
Note that we do not need to ask $X$ to be projective since all objects of $\int_\Sigma \cat{A}\simeq \vect$ are automatically projective.

\begin{proposition}\label{proptrace}
	Let $\cat{A}$ be a modular category.
	Then for a closed surface $\Sigma$ the trace $\trace_\catf{HS}$
	does not depend on the choice of $H$ in \eqref{eqntracenumber},  and $\trace_\catf{HS}$ is cyclic and non-degenerate.
	Let 
	   $f\in \End_{\int_\Sigma \cat{A}}(X)$ be an idempotent, i.e.\ $f\circ f=f$, 
	   then $\trace_\catf{HS}(f)$ 
	    is a non-negative integer,
	    namely the dimension of the subspace of the skein module for $H$ onto which $\PhiA(H)f$ projects.
\end{proposition}

\begin{proof}
	For a different handlebody $H'$ with $\partial H'=\Sigma$,
	we have a natural isomorphism $\alpha : \PhiA(H)\ra{\cong} \PhiA(H')$ by connectedness of $\cat{A}$~\cite[Corollary~8.3]{brochierwoike}.
	For any endomorphism $f: X \to X$ in $\int_\Sigma \cat{A}$, this implies $\alpha_X\circ \PhiA(H)f = (\PhiA(H')f)\circ  \alpha_X$.
	The cyclicity of the matrix trace implies now the independence of $H$. 
	By~\eqref{eqnphiaqss} the cyclicity and non-degeneracy of the matrix trace carry over  to $\trace_\catf{HS}$.
	Finally, if $f$ is an idempotent, $ \PhiA(H)f$ is an idempotent matrix whose trace is a non-negative integer. 
	\end{proof}

Here `non-negative integer' means that the element is in the image of the standard map $\mathbb{Z}_{\ge 0}\to k$. Without assumption on the characteristic ok $k$, this map is not necessarily injective of course.

\subsection{The notion of the partition function in the non-semisimple case\label{sectoruspartdef}}
As was explained in the introduction,
the notion of a torus partition function is not a priori defined beyond the semisimple situation.
We will first have to introduce a reasonable notion. 
We will see 
that the proposed definition is, in hindsight, in line with existing definitions \cite{frs1,cardycartan}. This will not be obvious from the definition, but will follow from Theorem~\ref{thmpartition} below.

For a special symmetric Frobenius algebra $F\in\cat{A}$ in a modular category $\cat{A}$, we will first consider the vector  \begin{align}  \lambda_{\mathbb{T}^2_1}^F\in  \Ao(  \mathbb{T}_1^2 ; F)\cong \FAA(\mathbb{T}_1^2;\bzero(F))\label{eqnlambdatorus}
	\end{align} that 
 $F$ gives rise to by Theorem~\ref{thmsym}, and that
induces the correlator for $\mathbb{T}_1^2$ in Theorem~\ref{thmmainlong}.

By  construction 
 this vector in $\Ao(  \mathbb{T}_1^2 ; F)$
is just the vector that 
 $F$ gives  us through the modular microcosm principle.
This allows us to use the recipe in \cite[Section~9]{microcosm} to calculate it: We start with a disk with five marked intervals in its boundary that each carry the label $F$. The associated space of conformal blocks is $\cat{A}(I,F^{\otimes 5})$. For this disk, the (open) correlator is  a vector in $\cat{A}(I,F^{\otimes 5})$, namely the total arity five operation $I\to F^{\otimes 5}$ from~\eqref{eqntarn} of $F$ that we visualize by a red graph drawn on the surface, see Figure~\ref{figcorrelatortorus}. Then we glue two pairs of intervals together to obtain $\mathbb{T}_1^2$, again with the correlator drawn in red.

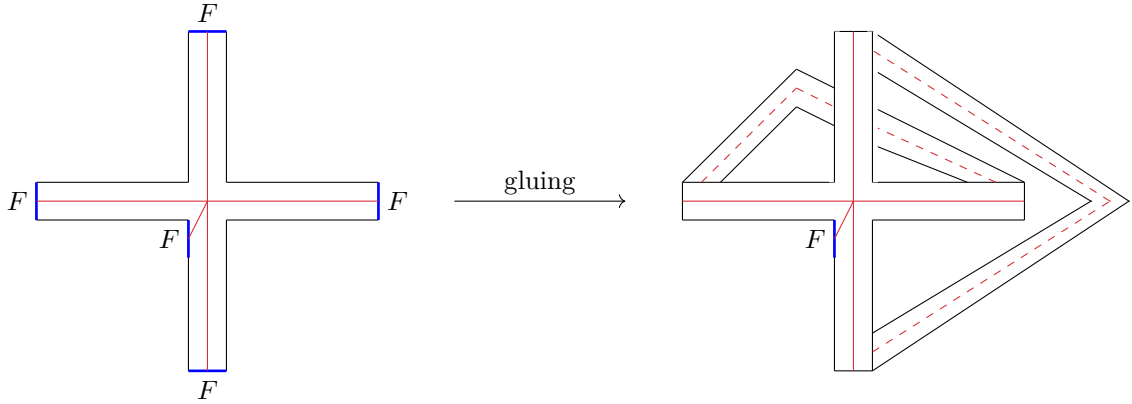
\begin{figure}[h]
\begin{tikzpicture}[scale=0.5]
	\begin{pgfonlayer}{nodelayer}
		\node [style=none] (0) at (4, 0) {};
		\node [style=none] (1) at (5, 0) {};
		\node [style=none] (2) at (4, 4) {};
		\node [style=none] (3) at (5, 4) {};
		\node [style=none] (4) at (4, -5) {};
		\node [style=none] (5) at (5, -5) {};
		\node [style=none] (6) at (4, -1) {};
		\node [style=none] (7) at (5, -1) {};
		\node [style=none] (8) at (9, 0) {};
		\node [style=none] (9) at (9, -1) {};
		\node [style=none] (10) at (0, 0) {};
		\node [style=none] (11) at (0, -1) {};
		\node [style=none] (12) at (-13, 0) {};
		\node [style=none] (13) at (-12, 0) {};
		\node [style=none] (14) at (-13, 4) {};
		\node [style=none] (15) at (-12, 4) {};
		\node [style=none] (16) at (-13, -5) {};
		\node [style=none] (17) at (-12, -5) {};
		\node [style=none] (18) at (-13, -1) {};
		\node [style=none] (19) at (-12, -1) {};
		\node [style=none] (20) at (-8, 0) {};
		\node [style=none] (21) at (-8, -1) {};
		\node [style=none] (22) at (-17, 0) {};
		\node [style=none] (23) at (-17, -1) {};
		\node [style=none] (24) at (-13, -2) {};
		\node [style=none] (25) at (4, -2) {};
		\node [style=none] (26) at (7.5, 0) {};
		\node [style=none] (27) at (1, 0) {};
		\node [style=none] (28) at (10.75, -0.5) {};
		\node [style=none] (29) at (11.75, -0.5) {};
		\node [style=none] (30) at (10.75, -0.5) {};
		\node [style=none] (31) at (11.75, -0.5) {};
		\node [style=none] (32) at (5, -4) {};
		\node [style=none] (33) at (5, 3) {};
		\node [style=none] (34) at (3, 2) {};
		\node [style=none] (35) at (5, 1) {};
		\node [style=none] (36) at (3, 3) {};
		\node [style=none] (37) at (5, 2) {};
		\node [style=none] (38) at (-12.5, -0.5) {};
		\node [style=none] (39) at (4.5, -0.5) {};
		\node [style=none] (40) at (-12.5, -5) {};
		\node [style=none] (41) at (-8, -0.5) {};
		\node [style=none] (42) at (-12.5, 4) {};
		\node [style=none] (43) at (-17, -0.5) {};
		\node [style=none] (44) at (-13, -1.5) {};
		\node [style=none] (45) at (4.5, -5) {};
		\node [style=none] (46) at (4.5, 4) {};
		\node [style=none] (47) at (9, -0.5) {};
		\node [style=none] (48) at (0, -0.5) {};
		\node [style=none] (49) at (4, -1.5) {};
		\node [style=none] (50) at (3, 2.5) {};
		\node [style=none] (51) at (5, 1.5) {};
		\node [style=none] (52) at (0.5, 0) {};
		\node [style=none] (53) at (8.25, 0) {};
		\node [style=none] (54) at (11.25, -0.5) {};
		\node [style=none] (55) at (11.25, -0.5) {};
		\node [style=none] (56) at (5, 3.5) {};
		\node [style=none] (57) at (5, -4.5) {};
		\node [style=none] (58) at (-17.5, -0.5) {$F$};
		\node [style=none] (59) at (3.5, -1.5) {$F$};
		\node [style=none] (60) at (-12.5, 4.5) {$F$};
		\node [style=none] (61) at (-7.5, -0.5) {$F$};
		\node [style=none] (62) at (-12.5, -5.5) {$F$};
		\node [style=none] (63) at (-13.5, -1.5) {$F$};
		\node [style=none] (64) at (-6, -0.5) {};
		\node [style=none] (65) at (-1.5, -0.5) {};
		\node [style=none] (66) at (-3.75, 0) {gluing};
		\node [style=none] (67) at (4, 2.5) {};
		\node [style=none] (68) at (4, 2) {};
		\node [style=none] (69) at (4, 1.5) {};
	\end{pgfonlayer}
	\begin{pgfonlayer}{edgelayer}
		\draw (0.center) to (10.center);
		\draw (6.center) to (11.center);
		\draw (6.center) to (4.center);
		\draw (7.center) to (5.center);
		\draw (7.center) to (9.center);
		\draw (1.center) to (8.center);
		\draw (13.center) to (15.center);
		\draw (12.center) to (14.center);
		\draw (12.center) to (22.center);
		\draw (18.center) to (23.center);
		\draw (18.center) to (16.center);
		\draw (19.center) to (17.center);
		\draw (19.center) to (21.center);
		\draw (13.center) to (20.center);
		\draw [style=open] (14.center) to (15.center);
		\draw [style=open] (20.center) to (21.center);
		\draw [style=open] (16.center) to (17.center);
		\draw [style=open] (22.center) to (23.center);
		\draw [style=open] (18.center) to (24.center);
		\draw (8.center) to (9.center);
		\draw (10.center) to (11.center);
		\draw (2.center) to (3.center);
		\draw (4.center) to (5.center);
		\draw [style=open] (6.center) to (25.center);
		\draw (28.center) to (30.center);
		\draw (29.center) to (31.center);
		\draw (5.center) to (29.center);
		\draw (32.center) to (28.center);
		\draw (30.center) to (33.center);
		\draw (31.center) to (3.center);
		\draw (27.center) to (34.center);
		\draw (10.center) to (36.center);
		\draw (35.center) to (26.center);
		\draw (37.center) to (8.center);
		\draw [style=RED] (40.center) to (38.center);
		\draw [style=RED] (38.center) to (41.center);
		\draw [style=RED] (38.center) to (42.center);
		\draw [style=RED] (44.center) to (38.center);
		\draw [style=RED] (38.center) to (43.center);
		\draw [style=RED] (39.center) to (45.center);
		\draw [style=RED] (49.center) to (39.center);
		\draw [style=RED] (39.center) to (47.center);
		\draw [style=RED] (39.center) to (48.center);
		\draw [style=RED] (39.center) to (46.center);
		\draw [style=REDdashed] (52.center) to (50.center);
		\draw [style=REDdashed] (53.center) to (51.center);
		\draw [style=mover] (1.center) to (3.center);
		\draw [style=mover] (2.center) to (0.center);
		\draw [style=REDdashed] (57.center) to (54.center);
		\draw [style=REDdashed] (54.center) to (55.center);
		\draw [style=REDdashed] (55.center) to (56.center);
		\draw [style=end arrow] (64.center) to (65.center);
		\draw (36.center) to (67.center);
		\draw (34.center) to (69.center);
		\draw [style=REDdashed] (50.center) to (68.center);
	\end{pgfonlayer}
\end{tikzpicture}
	\caption{For the calculation of the correlator for the torus with one boundary component.}
	\label{figcorrelatortorus}
\end{figure}

By Theorem~\ref{thmao} this corresponds to a morphism \begin{align}  \cat{O}_{\mathbb{T}_1^2} ^\cat{A}\to  F\act \cat{O}_{\mathbb{T}_1^2}^\cat{A}\label{eqncorrelatortorus}\end{align} in the factorization homology $\int_{\mathbb{T}_1^2}\cat{A}$ that by postcomposing with the counit $\varepsilon:F\to I$ gives us an element in $\SkAlg_\cat{A}(\mathbb{T}_1^2)$
invariant under
$\Map(\mathbb{T}_1^2)\cong B_3$
 or, equivalently, a generalized  element $I\to \cat{D}_\cat{A}$ in the elliptic double $\cat{D}_\cat{A}:= \END (\cat{O}^\cat{A}_{\mathbb{T}_1^2})\in\cat{A}$, namely~\eqref{eqninvinD}.

 From the element in $\SkAlg_\cat{A}(\mathbb{T}_1^2)$, we extract the partition function of the closed torus:
We denote its image 
 under $\int_{\mathbb{T}_1^2}\cat{A}\to \int_{\mathbb{T}^2} \cat{A}$ induced by the embedding $\mathbb{T}_1^2 \to \mathbb{T}_1^2 \cup_{\mathbb{S}^1} \mathbb{D}^2=\mathbb{T}^2$ by $\mathcal{Z}^F \in \SkAlg_\cat{A}(\mathbb{T}^2)$.
By \eqref{phaequiveqn} and~\eqref{eqnphiaqss} $\mathcal{Z}^F$ can be equivalently seen as an endomorphism $\FA(\mathbb{T}^2)\to \FA(\mathbb{T}^2)$
(this would a priori apply to the $\Rexf$-version of $\FA(\mathbb{T}^2)$ which however is dual to the $\Lexf$-valued one, so that both versions of the space of conformal blocks have the same endomorphisms).

\begin{definition}\label{defpartition}
	We call the endomorphism $\mathcal{Z}^F$ of $\FA(\mathbb{T}^2)$ the \emph{torus partition function} of the open-closed system of correlators associated to the special symmetric Frobenius algebra $F\in\cat{A}$.
\end{definition}

\begin{remark}
	This construction of an element in the skein algebra  and the associated endomorphism of the space of conformal blocks does not only work for the torus, but all closed surfaces of course. Hence, there is in principle a factorization homology construction of higher genus partition functions, but we focus our attention on the torus in this article.
	\end{remark}

\subsection{Modular invariance}
Before calculating the torus partition function in more detail,
we establish the property of modular invariance known from the semisimple case~\cite[Theorem~5.1]{frs1}:

\begin{proposition}\label{propmodularinvariance}
	The torus partition function $\mathcal{Z}^F$ has the modular invariance property
	\begin{align} [\mathcal{Z}^F,\FA(f)]=0 \quad \text{for all}\quad f \in \SL(2,\mathbb{Z}) \ . \label{eqncommutator}
	\end{align}
\end{proposition}

Of course, $\FA(f)$ is only defined up to scalar, but the commutator being zero does not depend on that.

\begin{proof}
	The map~\eqref{eqncorrelatortorus} is invariant under the action of $\Map(\mathbb{T}_1^2)\cong B_3$, the braid group on three strands, by Theorem~\ref{thmao}, and after postcomposing with the counit, we obtain an invariant element in $\SkAlg_\cat{A}(\mathbb{T}_1^2)$. Its image under the $B_3$-equivariant map $\SkAlg_\cat{A}(\mathbb{T}_1^2)\to\SkAlg_\cat{A}(\mathbb{T}^2)$ is $B_3$-invariant as well, where $\SkAlg_\cat{A}(\mathbb{T}^2)$ carries the $B_3$-action obtained by restriction of the geometric $\SL(2,\mathbb{Z})$-action along the epimorphism $B_3 \to \SL(2,\mathbb{Z})$. This implies that $\mathcal{Z}^F$ is by construction the image of an $\SL(2,\mathbb{Z})$-invariant  element in $\SkAlg_\cat{A}(\mathbb{T}^2)$ under an isomorphism  
	$\SkAlg_\cat{A}(\mathbb{T}^2)\cong \End(\FA(\mathbb{T}^2))$ that is mapping class group equivariant by \cite[Theorem~6.1]{reflection}.
	Here $\End(\FA(\mathbb{T}^2))$ carries the adjoint action.
	This proves~\eqref{eqncommutator}.
\end{proof}

\subsection{Coefficients of the partition function\label{seccoeff}}
If $\cat{A}$ is semisimple, $\FA(\mathbb{T}^2)$ is the linear span of a basis of simple objects $x_i$.
The coefficients  $\mathcal{Z}^F$ with respect to that basis are the usual coefficients of the torus partition function.
The purpose of the rest of this section is to make, inspired by \cite{frs1,cardycartan}, a proposal how to extract coefficients from $\mathcal{Z}^F$ beyond semisimplicity and to use the gluing pattern in Figure~\ref{figcorrelatortorus} to calculate these coefficients.

Let us think of the torus $\mathbb{T}^2$ as obtained by gluing the two boundaries  of a cylinder together. This gives us an identification $\FA(\mathbb{T}^2)\cong \cat{A}(I,\mathbb{F})$, and $\mathcal{Z}^F$ can be  seen as an endomorphism of 
that vector space. Let $P \in \Proj \cat{A}$, then any endomorphism $f:P\to P$ 
 induces a map $I\to \mathbb{F}$. We denote the associated element in $\FA(\mathbb{T}^2)$ by $[f]$. 	

This gives us an opportunity to obtain a coefficient $\mathcal{Z}_{P,Q}^F \in k$ of the partition function, i.e.\
a `$(P,Q)$-matrix element' of the endomorphism $\mathcal{Z}^F$ for projective objects $P,Q\in\Proj \cat{A}$.
We should start with a warning: The coefficients that we extract follow a rather natural procedure prompted by the factorization homology description of $\mathcal{Z}^F$, but at the end of the day these coefficients are just numbers that one extracts. The identities at (indecomposable) projective objects seen as vectors $[\id_P]$ in the space of conformal blocks of the torus are \emph{not a basis}. We refer to Remark~\ref{remcardycartan} for a few more comments on this.

Fix $P,Q\in\Proj\cat{A}$. The vector $\mathcal{Z}^F[\id_Q]$ can be seen as an element in the skein module for a handlebody with $\partial H=\mathbb{T}^2$. This means by \cite{mwskein} that it can be described by an admissible skein  in $H$ in the sense of~\cite{asm}. The identity of $P^\vee$ winding around $\bar H$ once gives us similarly
an admissible skein in $\bar H$.
The union is an admissible skein in $\bar H\cup H\cong \mathbb{S}^2 \times \mathbb{S}^1$. 

The admissible skein module $\skA(\mathbb{S}^2 \times \mathbb{S}^1)$ is isomorphic to $k$ \cite[Theorem 5.8]{skeintft}, in fact canonically so because it is the Hochschild homology of the Müger center of $\cat{A}$ (which is opposite to $\int_{\mathbb{S}^2} \cat{A}$), see \cite[Lemma~4.5]{skeinfin} for the semisimple case and \cite[Corollary~5.2]{reflection}
for the non-semisimple case.
 This means that we obtain a number.
We define this number to be the coefficient $\mathcal{Z}_{P,Q}^F \in k$ of the partition function. This construction is in direct analogy to \cite[Section~5.3]{frs1}.

\spaceplease
\subsection{Calculating the coefficients of the	partition function using factorization homology}
We will first prove the following far-reaching generalization of~\cite[Lemma~5.2]{frs1}:

\begin{lemma}\label{lemmacoeff}
	With the notation from Section~\ref{seccoeff}, $\mathcal{Z}^F[f]=[z_f^F]$ for the endomorphism
	\begin{equation}\begin{array}{c}
	\begin{tikzpicture}[scale=0.5]
		\begin{pgfonlayer}{nodelayer}
			\node [style=none] (0) at (-1, -4) {};
			\node [style=none] (1) at (-1, 1) {};
			\node [style=none] (2) at (-1, 4) {};
			\node [style=none] (3) at (-2, -4) {};
			\node [style=none] (4) at (-2, -1) {};
			\node [style=none] (5) at (0, 1) {};
			\node [style=none] (6) at (-2, 3) {};
			\node [style=none] (8) at (-2, -4.75) {$F$};
			\node [style=none] (9) at (-1, -4.75) {$Q$};
			\node [style=wbox] (10) at (-1, -2.75) {$f$};
			\node [style=none] (11) at (-5, 0) {$=$};
			\node [style=none] (12) at (-6.5, 0) {$z_f^F$};
			\node [style=none] (13) at (-2, 4) {};
			\node [style=none] (14) at (-2, 1) {};
			\node [style=none] (15) at (-2, 1) {};
			\node [style=none] (16) at (0, 1) {};
		\end{pgfonlayer}
		\begin{pgfonlayer}{edgelayer}
			\draw (3.center) to (4.center);
			\draw [style=open] (4.center) to (3.center);
			\draw [style=open, bend left=90, looseness=1.75] (4.center) to (6.center);
			\draw [style=open] (4.center) to (15.center);
			\draw (0.center) to (1.center);
			\draw [style=bover, bend right=45] (4.center) to (16.center);
			\draw [style=open, bend right=45] (16.center) to (6.center);
			\draw [style=over] (1.center) to (2.center);
			\draw [style=open] (3.center) to (15.center);
			\draw [style=bover] (15.center) to (13.center);
		\end{pgfonlayer}
	\end{tikzpicture}\end{array}
	 \label{eqntorusaction}
	\end{equation}
	of the projective object $F\otimes Q$. Here $f$ is an endomorphism of $Q\in\Proj\cat{A}$.
\end{lemma}

In the graphical calculus for braided monoidal categories used here, the diagrams are to be read from bottom to top. The Frobenius algebra is depicted in blue. Any $n$-valent vertex of the blue line is the standard operation of $F$ of arity $n$.

\begin{proof}[\slshape Proof of Lemma~\ref{lemmacoeff}]
	As in Section~\ref{seccoeff}, we think of 
	the torus as obtained by gluing the ends of a cylinder $C:=\mathbb{S}^1 \times[0,1]$ together.
	Then the construction from Section~\ref{sectoruspartdef}, when expressed in formulae, is
	\begin{align}
		z_f^F = \Add(s^F_{C} )^Q f \ , 
	\end{align}
	where
	\begin{itemize}
		
		\item $\Add(-)^Q - $ is the $Q$-component of the skein action for the cylinder  given by
		
		\footnotesize
		\begin{equation}\label{eqncompskeinaction}
			\begin{tikzcd}
				\Hom_{\int_{C}\cat{A}}(\cat{O}_{ C   }^\cat{A},(F\boxtimes F) \act \cat{O}_{ C   }^\cat{A})  \ar[]{rrr}{\Add=\PhiA(H)\circ -}  \ar[swap]{dd}{\Add(-)^Q -} &&& 	\Hom (   \FA^{\Rexf}(C;-) , \FA^{\Rexf}(C;F\otimes-,F\otimes-  ))  \ar{ddlll}{\text{\cite[Theorem 3.4]{mwskein} and restriction to $Q$}} \\
				\\ \Hom ( \FA(C;Q,Q^\vee), \FA( C;F\otimes Q,F\otimes Q^\vee  )   )
			\end{tikzcd} 
		\end{equation}
		\normalsize
		
		(the skein action $\Add$ is the postcomposition with $\PhiA$, see \cite[Section~4.3]{brochierwoike}),
		\item and $s^F_C \in 	\Hom_{\int_{C}\cat{A}}(  \cat{O}_{ C   }^\cat{A},(F\boxtimes F)\act\cat{O}_{ C   }^\cat{A})$ is the image of
		the open correlator $\lambda_C^F \in \Ao(C;F,F)$ under the isomorphism \begin{align}
			\Ao(C;F,F)\cong \Hom_{\int_{C}\cat{A}}(\cat{O}_{ C   }^\cat{A},(F\boxtimes F) \act \cat{O}_{ C   }^\cat{A})
			\end{align}
		from Theorem~\ref{thmao}.
		
	\end{itemize}
	To find~\eqref{eqntorusaction}, we need to calculate all these quantities:
	We begin by using the calculation recipe in \cite[Section~9]{microcosm} to conclude
	$\Ao(C;F,F)\cong \cat{A}\left(F,\int^{X \in \cat{A}} X^\vee \otimes F \otimes X\right)$ 
	and that $\lambda_C^F$ is given by
	\begin{align}
		\lambda_C^F : F \ra{\text{total arity $4$ operation}} F^{\otimes 3} \ra{\text{self-duality of $F$ and structure map of the coend}} \int^{X \in \cat{A}} X^\vee \otimes F \otimes X \ . 
	\end{align}
	By braiding $F$ over the left dummy variable, we obtain an element in $\cat{A}(F,F\otimes \mathbb{F})\cong \cat{A}(F\otimes F,\mathbb{F})$ that in the graphical calculus is
	given by
	\begin{equation}
	\begin{array}{c}\begin{tikzpicture}[scale=0.5]
			\begin{pgfonlayer}{nodelayer}
				\node [style=none] (0) at (-1, 0) {};
				\node [style=none] (1) at (-1, 2) {};
				\node [style=none] (2) at (-2, 4) {};
				\node [style=none] (3) at (-1, 4) {};
				\node [style=none] (4) at (0, 4) {};
				\node [style=none] (5) at (-2, 6) {};
				\node [style=none] (6) at (-1, 6) {};
				\node [style=none] (7) at (0, 6) {};
				\node [style=none] (8) at (-3, 0) {};
			\end{pgfonlayer}
			\begin{pgfonlayer}{edgelayer}
				\draw [style=open] (0.center) to (1.center);
				\draw [style=open] (1.center) to (3.center);
				\draw [style=open, in=270, out=120, looseness=0.75] (1.center) to (2.center);
				\draw [style=open] (1.center) to (4.center);
				\draw [style=open, in=-90, out=90] (2.center) to (6.center);
				\draw [style=open] (4.center) to (7.center);
				\draw [style=bover, in=0, out=90, looseness=0.75] (3.center) to (5.center);
				\draw [style=open, in=90, out=-180, looseness=0.50] (5.center) to (8.center);
			\end{pgfonlayer}
		\end{tikzpicture}
	\end{array}  \ , \label{eqnskeincyl}
	\end{equation}
	where the structure map sorting $F\otimes F$ into the coend via the self-duality is suppressed for readability; in other words, we display the components of the map.
	This map $F\otimes F \to \mathbb{F}$ provides us with a right $\mathbb{F}$-module map $F\otimes F \otimes \mathbb{F}\to\mathbb{F}$ for Lyubashenko coend algebra~\cite[Section~2]{lyu}, where $F\otimes F \otimes\mathbb{F}$ is the free right $\mathbb{F}$-module on $F\otimes F$. 
	The equivalence $\int_{C} \cat{A}\simeq \rmod\mathbb{F}$~\cite[Theorem~5.14]{bzbj} and the isomorphism from Theorem~\ref{thmao} is constructed in such a way such that
	the $\mathbb{F}$-module map
	$F\otimes F \otimes \mathbb{F}\to\mathbb{F}$
	seen as an element in $\Hom_{\int_{C} \cat{A}}(  \cat{O}_{ C   }^\cat{A}, (F\boxtimes F) \act 	\cat{O}_{ C   }^\cat{A})$ is exactly $s_C^F$. 
	Now we need to let $s_C^F$, with its description as an $\mathbb{F}$-module morphism~\eqref{eqnskeincyl}, act --- via the skein action  --- on $[f]$. 
	
	The skein action is defined by postcomposition with $\PhiA(\mathbb{D}^2\times [0,1])$,
	and a concrete description can be extracted from 
	 the proof of~\cite[Theorem~3.20]{bjss}. The authors do   not express this via the $\PhiA$-map, but the connection is explained in 
	\cite[Section~7.2]{brochierwoike}. In \cite[Remark~6.4]{reflection} this is used to express $\PhiA(\mathbb{D}^2\times [0,1])$ 
	 via the Drinfeld map, the well-known algebra map $\mathbb{D}:\mathbb{F}\to \mathbb{A}$ from the coend to the end~\cite{drinfeld},
	by using facts from the proof of \cite[Theorem~3.20]{bjss}. With this description of the skein action, we find 
	\begin{equation}
	\begin{array}{c}	\begin{tikzpicture}[scale=0.5]
			\begin{pgfonlayer}{nodelayer}
				\node [style=none] (11) at (-1, 1) {$=$};
				\node [style=none] (12) at (-4.75, 1) {$\Add(s^F_{C} )^Q f$};
				\node [style=none] (17) at (4, -5) {};
				\node [style=none] (18) at (4, -3) {};
				\node [style=none] (19) at (3, -1) {};
				\node [style=none] (20) at (4, -1) {};
				\node [style=none] (21) at (5, -1) {};
				\node [style=none] (22) at (3, 1) {};
				\node [style=none] (23) at (4, 1) {};
				\node [style=none] (24) at (5, 1) {};
				\node [style=none] (25) at (2, -1) {};
				\node [style=none] (26) at (6, 0) {};
				\node [style=none] (27) at (7, 0) {};
				\node [style=wbox] (28) at (8, -1) {$f$};
				\node [style=none] (29) at (5, 3) {};
				\node [style=none] (30) at (6, 3) {};
				\node [style=none] (31) at (5, 5) {};
				\node [style=none] (32) at (6, 5) {};
				\node [style=none] (35) at (8, -5) {};
				\node [style=none] (36) at (8, 0) {};
				\node [style=none] (37) at (7, 0) {};
				\node [style=none] (38) at (1, -1) {};
				\node [style=none] (39) at (1, 5) {};
				\node [style=none] (40) at (4, 5) {};
				\node [style=none] (41) at (6, 1) {};
				\node [style=none] (42) at (1, 7) {};
				\node [style=none] (43) at (6, 7) {};
				\node [style=none] (44) at (3.5, 6) {};
				\node [style=none] (45) at (7.5, 5) {};
				\node [style=none] (46) at (3.5, 1.25) {};
				\node [style=none] (47) at (7.25, -3.25) {};
			\end{pgfonlayer}
			\begin{pgfonlayer}{edgelayer}
				\draw [style=open] (17.center) to (18.center);
				\draw [style=open] (18.center) to (20.center);
				\draw [style=open, in=270, out=120, looseness=0.75] (18.center) to (19.center);
				\draw [style=open] (18.center) to (21.center);
				\draw [style=open, in=-90, out=90] (19.center) to (23.center);
				\draw [style=open] (21.center) to (24.center);
				\draw [style=bover, in=0, out=90, looseness=0.75] (20.center) to (22.center);
				\draw [style=open, in=90, out=-180, looseness=0.50] (22.center) to (25.center);
				\draw [style=open, in=-90, out=90, looseness=0.75] (30.center) to (31.center);
				\draw [bend left=90, looseness=19.75] (27.center) to (36.center);
				\draw (36.center) to (35.center);
				\draw [bend right=90, looseness=6.75] (26.center) to (37.center);
				\draw [style=open] (38.center) to (39.center);
				\draw [style=open, bend right=90, looseness=2.75] (38.center) to (25.center);
				\draw [style=open, bend left=90, looseness=2.00] (40.center) to (31.center);
				\draw [style=open] (40.center) to (23.center);
				\draw (26.center) to (41.center);
				\draw [in=-90, out=90] (41.center) to (29.center);
				\draw [style=bover, in=-90, out=90] (24.center) to (30.center);
				\draw (32.center) to (43.center);
				\draw [style=open] (39.center) to (42.center);
				\draw [style=over, in=-90, out=90, looseness=0.75] (29.center) to (32.center);
				\draw [style=mydots] (44.center) to (45.center);
				\draw [style=mydots] (44.center) to (46.center);
				\draw [style=mydots] (46.center) to (47.center);
				\draw [style=mydots] (47.center) to (45.center);
			\end{pgfonlayer}
		\end{tikzpicture}
	\end{array}  \ , \label{eqntorusaction0}
	\end{equation}
	with the components on the Drinfeld map highlighted through the dashed box.
	This gives us~\eqref{eqntorusaction}.
\end{proof}

This shows us that
the skein describing $\mathcal{Z}_{P,Q}^F$ is obtained by closing up the following diagram in $\mathbb{S}^2 \times [0,1]$:
\begin{equation}
	\begin{array}{c}
		\begin{tikzpicture}[scale=0.3]
			\begin{pgfonlayer}{nodelayer}
				\node [style=none] (2) at (0, 7) {};
				\node [style=none] (3) at (-1, -3.5) {};
				\node [style=none] (7) at (-1, 7) {};
				\node [style=none] (8) at (-1.65, -2.5) {$F$};
				\node [style=none] (9) at (0.55, -2.5) {$Q$};
				\node [style=none] (10) at (-5, -4) {};
				\node [style=none] (11) at (2, -4) {};
				\node [style=none] (12) at (0, -3.5) {};
				\node [style=none] (13) at (-0.25, -5.75) {};
				\node [style=none] (14) at (-0.25, -5.75) {$\mathbb{S}^2$};
				\node [style=none] (15) at (-4, -3.5) {};
				\node [style=none] (16) at (-4, 7) {};
				\node [style=none] (17) at (-3.25, -2.5) {$P$};
				\node [style=none] (19) at (0, 0.25) {};
				\node [style=none] (20) at (0, 3) {};
				\node [style=none] (21) at (0, 7) {};
				\node [style=none] (22) at (-1, 0.25) {};
				\node [style=none] (23) at (-1, 1) {};
				\node [style=none] (24) at (1, 3) {};
				\node [style=none] (25) at (-1, 5) {};
				\node [style=none] (29) at (-1, 7) {};
				\node [style=none] (30) at (-1, 3) {};
				\node [style=none] (31) at (-1, 3) {};
				\node [style=none] (32) at (1, 3) {};
			\end{pgfonlayer}
			\begin{pgfonlayer}{edgelayer}
				\draw [bend left=90, looseness=1.75] (10.center) to (11.center);
				\draw [bend right=90, looseness=1.75] (10.center) to (11.center);
				\draw [style=mydots, bend right=15] (10.center) to (11.center);
				\draw [style=mover] (15.center) to (16.center);
				\draw (22.center) to (23.center);
				\draw [style=open] (23.center) to (22.center);
				\draw [style=open, bend left=90, looseness=1.75] (23.center) to (25.center);
				\draw [style=open] (23.center) to (31.center);
				\draw (19.center) to (20.center);
				\draw [style=bover, bend right=45] (23.center) to (32.center);
				\draw [style=open, bend right=45] (32.center) to (25.center);
				\draw [style=over] (20.center) to (21.center);
				\draw [style=open] (22.center) to (31.center);
				\draw [style=bover] (31.center) to (29.center);
				\draw [style=bover] (3.center) to (22.center);
				\draw [style=over] (12.center) to (19.center);
			\end{pgfonlayer}
		\end{tikzpicture}
		\end{array}
		 \label{eqntorusactiononsphere}
\end{equation}
Of course, this is the familiar diagram from~\cite[eq.~(5.30)]{frs1}.

For $X\in \cat{A}$, denote by $\spr{X}$ the resulting object in $\int_{\mathbb{S}^2}\cat{A}$ that we obtain placing $X$ on a disk in the sphere.
By its definition in Section~\ref{seccoeff}, $\mathcal{Z}_{P,Q}^F$ is the trace of the endomorphism $\varphi(F;P,Q)$ of $\spr{P \otimes F\otimes Q}$ in $ \int_{\mathbb{S}^2}\cat{A}$ described by the skein~\eqref{eqntorusactiononsphere} 
 in the sense of Section~\ref{sectraces}.
We are now in position to prove the main result of this section:

\begin{theorem}\label{thmpartition} Let $F\in \cat{A}$ be a special Frobenius algebra in a modular category $\cat{A}$. 
	The coefficient $\mathcal{Z}_{P,Q}^F$ 
	of the
	torus partition function
	for $P,Q \in \Proj \cat{A}$ is a non-negative integer, and $\mathcal{Z}_{P,Q}^F \le \dim\, \cat{A}(P^\vee,F\otimes P)$. In the Cardy case $F=I$, the torus partition function for $P,Q \in \Proj \cat{A}$ is the entry $C_{P^\vee,Q}=\dim \, \cat{A}(P^\vee,Q)$ of the Cartan matrix of $\cat{A}$.  
\end{theorem}

\begin{proof}
	The morphism $\varphi(F;P,Q):\spr{P \otimes F\otimes Q}\to \spr{P \otimes F\otimes Q}$
	in $\int_{\mathbb{S}^2}\cat{A}$ mentioned above
	is an idempotent which follows from the description derived in Lemma~\ref{lemmacoeff} and same calculation as for \cite[Lemma~5.2]{frs1}.
	Actually, it is also clear abstractly, without the explicit calculation: 
	The part of~\eqref{eqntorusactiononsphere} without $P$ is obtained through the action of a skein algebra element for the cylinder (called $s_C^F$ in the proof of Lemma~\ref{lemmacoeff})
	 on a vector in the skein module. 
	 Both the skein algebra element and this vector in the skein module are idempotent with respect to the composition of cylinders.
	For the skein algebra element, this is a consequence of Proposition~\ref{propgluingspecial}
	and Theorem~\ref{thmao}.
	For the vector in the skein module, it is clear (it is the identity).
	This shows that
	$\varphi(F;P,Q)$
	 is an idempotent in $\int_{\mathbb{S}^2} \cat{A}$.

	The trace of $\varphi(F;P,Q)$ in the sense of Section~\ref{sectraces} is $\mathcal{Z}_{P,Q}^F$ by construction. Now Proposition~\ref{proptrace} implies that $\mathcal{Z}^F_{P,Q}$ is a non-negative integer. 
	
	Under  the equivalence $\PhiA(\mathbb{B}^3):\int_{\mathbb{S}^2} \cat{A}\ra{\simeq}\vect$, the object $ \spr{P \otimes F\otimes Q}$ is sent to $\cat{A}(I,P\otimes F \otimes Q)$. By the construction in Section~\ref{sectraces} $\mathcal{Z}^F_{P,Q}$
	is the trace of an idempotent on $\cat{A}(I,P\otimes F \otimes Q)$, which tells us that us that it is bounded by $\dim\, \cat{A}(P^\vee,F\otimes Q)$.
	
	In the Cardy case, we have clearly $\mathcal{Z}^F=\id$ as endomorphism of $\FA(\mathbb{T}^2)$ because the corresponding element in the skein algebra of $\mathbb{T}^2$ would be the unit, see also Lemma~\ref{lemmacoeff}. Therefore, $\mathcal{Z}^I_{P,Q}$ is the trace of the identity of the vector space $\cat{A}(P^\vee,Q)$. This gives us the entries of the Cartan matrix as the coefficients
	of the partition function.
\end{proof}

\begin{example}
	If $\cat{A}$ is given by finite-dimensional modules over a ribbon factorizable Hopf algebra $L$, then for any special symmetric Frobenius algebra $F$ in $L$-modules, we can consider the coefficient
	$\mathcal{Z}^F_{L^\vee,L}$ for the projective generator $L$.
	Since $\Hom_L(L,F\otimes L)=F\otimes L$ as vector spaces, Theorem~\ref{thmpartition} tells us $\mathcal{Z}^F_{L^\vee,L}\le \dim\, F \cdot \dim\, L$. 
	For the special Frobenius algebra $k[N] \in Z(H\catf{-mod})\simeq D(H)\catf{-mod}$, with notation $H=K\otimes k[N]$ and assumptions as in Example~\ref{exspecialsymfrob}, this implies
	\begin{align}
		\mathcal{Z}^{k[N]}_{D(H)^\vee,D(H)}\le |N| \cdot \dim\, D(H) = |N| \cdot (\dim\, K \cdot |N|)^2 = |N|^3 \cdot (\dim\, 	 K)^2 \ . 
		\end{align}
\end{example}

\begin{remark}\label{remcardycartan}
	Theorem~\ref{thmpartition} tells us that, in the Cardy case,
	we recover as coefficients of the
	partition function the Cardy-Cartan modular invariant that was proposed as a non-semisimple generalization of the partition function in \cite{cardycartan}.
	However, we need to accompany this by several warnings even though they just echo what was said in Section~\ref{seccoeff}: In \cite[Theorem~3]{cardycartan}, the coefficients of the Cartan matrix can be used
	 for a chiral decomposition of the partition function in the Hopf-algebraic Cardy case. 
	 It is unclear to us how such a decomposition in terms of characters is related to the endomorphism of $\FA(\mathbb{T}^2)$
	 that we understand as partition function, and even if we could extract such an endomorphism from the object in \cite[Theorem~3]{cardycartan},
	 it is not obvious to us what it would do to vectors in $\FA(\mathbb{T}^2)$ that are \emph{not} characters. 
	 For the relation of the subspace of characters 
	  to $\FA(\mathbb{T}^2)$, see  \cite{shimizucf,gainutdinovrunkel}.
\end{remark}

\subsection{The elliptic class function in the Hopf-algebraic special case\label{secellicf}}
The calculation of coefficients for the torus partition function at projective objects in Theorem~\ref{thmpartition} is certainly helpful
and showcases that many of the principle behind \cite{frs1} actually remain valid way beyond the framework that they were conceived in.
Nonetheless, we have highlighted amply that the coefficients are really just a family of numbers that one extracts. The more fundamental object is the object $\mathcal{Z}^F$ from Definition~\ref{defpartition} built in terms of factorization homology.
Its calculation is not terribly difficult and was accomplished in Lemma~\ref{lemmacoeff}.
In this subsection, we want to discuss an even more tangible piece of algebraic information that is \emph{enough} to know $\mathcal{Z}^F$, namely the elliptic class function $\zeta^F_{1,1}: \mathbb{A}^{\otimes 2}\to I$, see Section~\ref{secclassfunctions}.
This should be essentially understood as a partition function for the torus with an unparametrized boundary component (the punctured torus).

\begin{lemma}
	If for two consistent systems of open-closed
	correlators for a modular category subject to the conditions (S) and~(O) from Section~\ref{secclass}
	  the elliptic class functions agree, then so do their torus partition functions.
	\end{lemma}

\begin{proof}
By Corollary~\ref{corclassification} this amounts to observing that $\zeta_{1,1}^F=\zeta_{1,1}^{F'}$ for two special symmetric Frobenius algebras $F$ and $F'$ in $\cat{A}$ implies $\mathcal{Z}^F=\mathcal{Z}^{F'}$.
But this is clear because, as remarked before Definition~\ref{defpartition}, the partition function is constructed from	the $B_3$-invariant element \eqref{eqninvinD} in the elliptic double that is just dual to the elliptic class function by Remark~\ref{exellclassfunction}.
	\end{proof}

The elliptic class function (also those of higher genus) can be calculated using the tools from Section~\ref{secclassfunctions}. To make the calculation even more explicit,
suppose that our modular category $\cat{A}$
is given by the finite-dimensional modules over a ribbon factorizable Hopf algebra $H$. 
Let $F\in\cat{A}$ be a special symmetric Frobenius algebra, and denote by 
$\spr{-,-}:F\otimes F\to k$
 its pairing.
Moreover, denote by $v_{(1)} \otimes v_{(2)} \otimes v_{(3)} \otimes v_{(4)} \in F^{\otimes 4}$ (in Sweedler notation)
the image of $1\in k$ under the total arity four operation $v_4$. 
With the $R$-matrix $R=R_1 \otimes R_2$ of $H$ with inverse $R^{-1}=SR_1 \otimes R_2$,
we can express
Remark~\ref{exellclassfunction} in the Hopf-algebraic special case using e.g.\ the textbook~\cite[Section~XIV.6]{kassel}:

\begin{corollary}\label{corellcfhopf}
	The elliptic class function for the full conformal field theory described by a special symmetric Frobenius algebra $F$ in the category of finite-dimensional modules over a  ribbon factorizable Hopf algebra $H$ with pivot $\omega$ is the map $\zeta^F_{1,1}: H_{\text{adj}}\otimes H_{\text{adj}}\to k$ of $H$-modules given by
	\begin{align}
		\zeta^F_{1,1}(	h_1\otimes h_2) =	\spr{ v_{(1)} ,((h_1 R_2).v_{(3)}} \spr{    v_{(2)}, (\omega R_1\omega^{-1} h_2).v_{(4)}}  \quad \text{for}\quad h_1,h_2 \in H     \label{eqnellclass} 
	\end{align}
	that is a modular invariant, i.e.\ invariant under the Brochier-Jordan $B_3$-action on $H_{\text{adj}}\otimes H_{\text{adj}}$. 
	Moreover, all elliptic class functions arising from full conformal field theories with monodromy data given by the category of finite-dimensional $H$-modules subject to the conditions (S) and~(O) from Section~\ref{secclass} are of the form~\eqref{eqnellclass}. 
\end{corollary}

The elliptic class function is a tractable algebraic object that determines the torus partition function.
Of course, one could now, as an approach to correlators in the logarithmic setting, try to find all modular invariant elliptic class functions.
This would be the logarithmic analogue of the classification program of modular invariants.
This comes with a well-known caveat, see e.g.\ the introduction of \cite{frs1} for similar discussions in the rational case, that also applies here:
Even if we find a modular invariant elliptic class function, it could be `unphysical' in the sense that it is not part of a consistent system of correlators.
One of the most conceptual reasons for this is that by Corollary~\ref{corclassfunctions} we can produce modular invariant elliptic class functions from symmetric Frobenius algebras that are not special.

\section{Local operators and Hochschild cohomology of the category of boundary conditions}
Using the consistent system of correlators associated to a special symmetric Frobenius algebra $F$ in modular category $\cat{A}$,
we construct in this section an algebra of local operators, a differential graded $E_2$-algebra whose multiplication is induced by the operator product expansion. For this algebra of local operators, we prove in Theorem~\ref{thmhhfmod}
a conjecture of Kapustin-Rozansky~\cite{kapustinrozansky} in the context of logarithmic conformal field theories.

\subsection{Algebra of local operators via homotopy invariants of the bulk field\label{sechi}}
Let $F\in\cat{A}$ be a special symmetric Frobenius algebra in a modular category $\cat{A}$.
The consistent system of correlators that $F$ gives rise to by Theorem~\ref{thmmainshort} produces, after restriction to genus zero, a symmetric
commutative Frobenius algebra in $\Fbar\in\bar{\cat{A}}\boxtimes \cat{A}$. We saw this 
 in the proof of Proposition~\ref{propextclosed}, but it is also a general necessity~\cite{jfcs,microcosm}.
Under the equivalence $G: \bar{\cat{A}}\boxtimes\cat{A}\ra{\simeq} Z(\cat{A})$ from~\eqref{equivG}, we obtain a symmetric 
commutative Frobenius algebra $\Fhb \in Z(\cat{A})$. So far, we have said rather little 
 about this Frobenius algebra, but we will derive a concrete description in Section~\ref{secrelbulk} and~\ref{secinternalend}.
From this algebra, we will produce \emph{an algebra of local operators} in chain complexes over $k$ by taking homotopy invariants in the appropriate sense:

\begin{proposition}\label{prophi} Let $F\in\cat{A}$ be a special symmetric Frobenius algebra in a modular category $\cat{A}$.
	Through the consistent system of correlators associated to $F$, the homotopy invariants $\mathbb{R}\Hom_\cat{A}(I,U\Fhb)$
	inherit the structure of differential graded algebra over the framed $E_2$-operad.
	In particular, $\Ext_\cat{A}(I,U\Fhb)$ is a Batalin-Vilkovisky algebra.
\end{proposition}

\begin{proof}
	By construction the algebra $U\Fhb$ in $\cat{A}$
	comes with a lift to a  commutative algebra in $Z(\cat{A})$ whose ribbon automorphism is trivial; the latter is a consequence of \cite[Proposition~2.25]{correspondences}.
	Now it follows from \cite[Theorem~7.3]{homotopyinvariants} that $\mathbb{R}\Hom_\cat{A}(I,U\Fhb)$ is a framed $E_2$-algebra in chain complexes.
	The fact that framed $E_2$-algebras produce Batalin-Vilkovisky algebras after taking cohomology is a result of \cite{getzlerbv}.
\end{proof}

We denote the framed $E_2$-algebra from Proposition~\ref{prophi} by $\LO(\cat{A};F)$ and refer to it as the \emph{algebra of local operators (on $\mathbb{S}^1$)} for the full conformal field theory given by the pair
$(\cat{A},\cor^F)$ of the monodromy data $\cat{A}$ 
and the consistent system of correlators $\xi^F$ built from $F$.
The terminology is motivated by the fact that the multiplication on the boundary object $F$ and its induced multiplication on the bulk object describe the operator product expansion for operators on the interval and the circle, respectively, see e.g.\
\cite{algcften} or the
introduction of~\cite{cardycartan}.
The fact that such an algebra of local operators should have a framed $E_2$-structure is, from a physical perspective, expected in \cite[Section~1.2]{bbzbdn}.

\subsection{The pivotal module category of boundary conditions\label{secreminderpiv}}
Given a special symmetric Frobenius algebra $F\in \cat{A}$ in a modular category $\cat{A}$, the category $\Fmod$ of $F$-modules in $\cat{A}$ is a so-called \emph{pivotal module category} over $\cat{A}$ according to \cite[Theorem~5.20]{schaumannpiv}, see also \cite[Section~3.6]{relserre}.
(Since we are working with left $F$-modules, this is a \emph{right} $\cat{A}$-module category.)
This means 
 that the internal homs come with isomorphisms
\begin{align}\HOM_F(M,N)^\vee \cong \HOM_F(N,M) \end{align} for $M,N \in \Fmod$ that square to the identity
 relative to the pivotal structure. In particular, the internal homs are exact.
Moreover, using \cite[Theorem 4.26]{fss} and \cite[Section~3]{shibatashimizu}, we can see 
 that $\Fmod$ is a projective Calabi-Yau category and therefore self-injective.

By \cite[Theorem~3.14]{relserre} 
the internal endomorphism algebras $\HOM_F(M,M)\in\cat{A}$ are symmetric Frobenius algebras, with the multiplication coming from the composition of internal endomorphisms.
By \cite[Corollary 39]{internal} the end \begin{align} \int_{M\in\Fmod} \HOM_F(M,M)\label{eqninternalend}\end{align} inherits from this the structure of a symmetric commutative Frobenius algebra in $Z(\cat{A})$. 
Informed by the experience with the semisimple case, the pivotal module category $\Fmod$ is suggested in~\cite{fspivotal,internal} as a \emph{category of boundary conditions} for a potential system of correlators built from $F$.

\subsection{Proof of a conjecture of Kapustin-Rozansky in the context of logarithmic conformal field theories}
Kapustin-Rozansky conjecture in \cite{kapustinrozansky} that the algebra of local operators 
agrees with the Hochschild cohomology of the category of boundary conditions.
Such a principle is for instance realized for the differential graded topological conformal field theories built from Calabi-Yau categories~\cite{costellotcft}.

The following result is a proof of this conjecture in the context of full logarithmic conformal field theory, with the algebra of local operators built in Section~\ref{sechi} and the category of boundary condition proposed in \cite{fspivotal,internal}, as recalled in Section~\ref{secreminderpiv}:

\begin{theorem}\label{thmhhfmod}
	Let $F$ be a special symmetric Frobenius algebra in a modular category $\cat{A}$.
	Then there is an equivalence 
	\begin{align} \LO(\cat{A};F) \simeq CH^*(\Fmod)
	\end{align} of differential graded framed $E_2$-algebras
	between the local operators for the conformal field theory $\cat{A}$ with correlator built from $F$ and the Hochschild cochains of the category $\Fmod$ of $F$-modules in $\cat{A}$, equipped with the framed $E_2$-structure coming from the cyclic Deligne conjecture.
\end{theorem}

The (cyclic) Deligne conjecture is of course a theorem, 
with so many different proofs that no reasonable introduction 
of several lines could do this subject any justice. For more background, we refer to the introduction of \cite{brdeligne} and to \cite{costellotcft} for a connection to conformal field theory.

As postulated in \cite{kapustinrozansky}, Theorem~\ref{thmhhfmod} relates the closed part of the conformal field theory (here represented by the derived algebra of local operators) to the deformations of the open part, namely to the Hochschild cohomology of the category of boundary conditions.
The proof of this result will occupy most of the rest of the article and will be finished on page~\pageref{secproofloop}.

\subsection{Relations characterizing the bulk field\label{secrelbulk}}
As a first step towards the proof of Theorem~\ref{thmhhfmod}, we explicitly calculate the bulk object $\bulk(F)$:
For a special symmetric Frobenius algebra $F\in \cat{A}$ for a modular category $\cat{A}$, let us see $\bzero(F)$ as an object in the Drinfeld center $Z(\cat{A})$ under the equivalence $\bar{\cat{A}}\boxtimes\cat{A}\ra{\simeq} Z(\cat{A})$. 
Then $\bzero(F)=LF$ for the left adjoint $L:\cat{A}\to Z(\cat{A})$ to the forgetful functor $U:Z(\cat{A})\to \cat{A}$, see Section~\ref{secopencor}.
Note that the right adjoint $R$ for $U$ is, thanks to reflection equivariance in the sense of \cite{reflection} that we fix,  canonically equivalent to $L$.
This follows from~\cite[Theorem~4.10]{shimizuunimodular} because the reflection equivariance amounts to a trivialization of the distinguished invertible object of $\cat{A}$.

Here we prefer to work with the model $\bzero(F)=RF$.
The underlying object in $\cat{A}$ is $\int_{X \in \cat{A}} X^\vee \otimes F \otimes X$. 
With the calculation methods in  \cite[Section~9]{microcosm} (we need to dualize appropriately), we find that the cylinder idempotent $C_F:\bzero(F)\to\bzero(F)$ from Section~\ref{seccyl}
is characterized by the fact that for the structure map $\pi_X : \bzero(F)\to X^\vee \otimes F \otimes X$ the $X$-component $\pi_X \circ C_F$ is given by
\begin{align}
	\int_{X \in \cat{A}} X^\vee \otimes F \otimes X \ra{\pi_{F \otimes X} \ \text{and self-duality of $F$}} X^\vee \otimes F^{\otimes 3} \otimes X \ra{\text{multiplication of $F$}} X^\vee \otimes F \otimes X\ . 
\end{align}
The fact that an endomorphism of an end can be specified this way follows from the universal property of the end.
We will write this endomorphism as
\begin{equation}
	\begin{array}{c}\begin{tikzpicture}[scale=0.5]
			\begin{pgfonlayer}{nodelayer}
				\node [style=none] (0) at (-1, 0) {};
				\node [style=none] (1) at (-1, 4) {};
				\node [style=none] (2) at (0, 0) {};
				\node [style=none] (3) at (1, 0) {};
				\node [style=none] (4) at (2, 0) {};
				\node [style=none] (5) at (3, 0) {};
				\node [style=none] (6) at (3, 4) {};
				\node [style=none] (7) at (1, 2) {};
				\node [style=none] (8) at (1, 4) {};
				\node [style=none] (9) at (-3, 2) {$=$};
				\node [style=none] (10) at (-4, 2) {$C_F$};
				\node [style=none] (11) at (-1, 2) {};
				\node [style=none] (12) at (3, 2) {};
			\end{pgfonlayer}
			\begin{pgfonlayer}{edgelayer}
				\draw [style=open] (2.center) to (7.center);
				\draw [style=open] (3.center) to (7.center);
				\draw [style=open] (4.center) to (7.center);
				\draw [style=open] (7.center) to (8.center);
				\draw [style=end arrow] (1.center) to (11.center);
				\draw [style=end arrow] (5.center) to (12.center);
				\draw (12.center) to (6.center);
				\draw (11.center) to (0.center);
			\end{pgfonlayer}
		\end{tikzpicture}
	 \end{array}\label{eqnCF}
\end{equation}
in the graphical calculus for monoidal categories.
The black lines are the dummy variables, with the upward pointing line for $X$ and the downward pointing one for $X^\vee$.
The thick blue lines are for the Frobenius algebra $F$. No direction is indicated because $F$ is self-dual. The vertex with four incident lines is the multiplication operation in arity 3, or total arity 4.
Beware that on the right hand side of \eqref{eqnCF} we indicate the \emph{components} of $C_F$, not $C_F$ itself. 
We prefer to not include the structure maps of the end into the graphical calculus because it makes the notation rather clumsy.
Note that the double strand of dummy variables of the end satisfies
\begin{equation}
	\begin{array}{c}	\begin{tikzpicture}[scale=0.5]
			\begin{pgfonlayer}{nodelayer}
				\node [style=none] (0) at (-3, 0) {};
				\node [style=none] (1) at (-3, 4) {};
				\node [style=none] (2) at (-1, 4) {};
				\node [style=none] (3) at (-1, 0) {};
				\node [style=none] (4) at (-3, 2) {};
				\node [style=none] (5) at (-1, 2) {};
				\node [style=none] (6) at (-1, 3) {};
				\node [style=none] (7) at (-1, 2.75) {};
				\node [style=none] (8) at (-1, 4) {};
				\node [style=wbox] (9) at (-1, 1) {$f$};
				\node [style=none] (10) at (0, 2) {$=$};
				\node [style=none] (11) at (1, 0) {};
				\node [style=none] (12) at (1, 4) {};
				\node [style=none] (13) at (3, 4) {};
				\node [style=none] (14) at (3, 0) {};
				\node [style=none] (15) at (1, 2.75) {};
				\node [style=none] (16) at (3, 2) {};
				\node [style=none] (19) at (3, 4) {};
				\node [style=wbox] (20) at (1, 1) {$f^\vee$};
				\node [style=none] (21) at (-2, 4) {};
				\node [style=none] (22) at (-2, 0) {};
				\node [style=none] (23) at (2, 4) {};
				\node [style=none] (24) at (2, 0) {};
			\end{pgfonlayer}
			\begin{pgfonlayer}{edgelayer}
				\draw [style=end arrow] (1.center) to (4.center);
				\draw (4.center) to (0.center);
				\draw [style=end arrow] (5.center) to (7.center);
				\draw (7.center) to (8.center);
				\draw (3.center) to (5.center);
				\draw [style=end arrow] (12.center) to (15.center);
				\draw (15.center) to (11.center);
				\draw [style=open] (21.center) to (22.center);
				\draw [style=open] (24.center) to (23.center);
				\draw [style=end arrow] (14.center) to (16.center);
				\draw (16.center) to (19.center);
			\end{pgfonlayer}
		\end{tikzpicture}
	\end{array} \tag{B1}\label{eqnend}
\end{equation}
for all morphisms in $\cat{A}$ because $\bzero(F)=\int_{X \in \cat{A}} X^\vee \otimes F \otimes X$ is the subobject of $\prod_{X \in \cat{A}} X^\vee \otimes F \otimes F$ satisfying this relation. 
With this notation, we find
\begin{equation}
	\begin{array}{c}\begin{tikzpicture}[scale=0.5]
			\begin{pgfonlayer}{nodelayer}
				\node [style=none] (0) at (-2, 0) {};
				\node [style=none] (1) at (-2, 6) {};
				\node [style=none] (2) at (0, 0) {};
				\node [style=none] (3) at (1, 0) {};
				\node [style=none] (4) at (2, 0) {};
				\node [style=none] (5) at (4, 0) {};
				\node [style=none] (6) at (4, 6) {};
				\node [style=none] (7) at (1, 2) {};
				\node [style=none] (8) at (1, 6) {};
				\node [style=none] (11) at (-2, 3) {};
				\node [style=none] (12) at (4, 3) {};
				\node [style=none] (13) at (-1, 0) {};
				\node [style=none] (14) at (3, 0) {};
				\node [style=none] (17) at (1, 4) {};
				\node [style=none] (18) at (6, 0) {};
				\node [style=none] (19) at (6, 6) {};
				\node [style=none] (20) at (8, 0) {};
				\node [style=none] (21) at (9, 0) {};
				\node [style=none] (22) at (10, 0) {};
				\node [style=none] (23) at (12, 0) {};
				\node [style=none] (24) at (12, 6) {};
				\node [style=none] (25) at (9, 2) {};
				\node [style=none] (26) at (9, 6) {};
				\node [style=none] (27) at (6, 3) {};
				\node [style=none] (28) at (12, 3) {};
				\node [style=none] (29) at (7, 0) {};
				\node [style=none] (30) at (11, 0) {};
				\node [style=none] (31) at (9, 4) {};
				\node [style=none] (32) at (5, 3) {$=$};
				\node [style=none] (33) at (-3, 3) {$=$};
				\node [style=none] (34) at (-4, 3) {$C_F^2$};
				\node [style=none] (35) at (7.5, 2) {};
				\node [style=none] (36) at (10.5, 2) {};
				\node [style=none] (37) at (14, 0) {};
				\node [style=none] (38) at (14, 6) {};
				\node [style=none] (42) at (18, 0) {};
				\node [style=none] (43) at (18, 6) {};
				\node [style=none] (45) at (16, 6) {};
				\node [style=none] (46) at (14, 3) {};
				\node [style=none] (47) at (18, 3) {};
				\node [style=none] (48) at (15, 0) {};
				\node [style=none] (53) at (13, 3) {$=$};
				\node [style=none] (54) at (15, 1) {};
				\node [style=none] (55) at (15, 3) {};
				\node [style=none] (56) at (16, 4) {};
				\node [style=none] (57) at (17, 0) {};
				\node [style=none] (58) at (16, 0) {};
			\end{pgfonlayer}
			\begin{pgfonlayer}{edgelayer}
				\draw [style=open] (2.center) to (7.center);
				\draw [style=open] (3.center) to (7.center);
				\draw [style=open] (4.center) to (7.center);
				\draw [style=open] (7.center) to (8.center);
				\draw [style=end arrow] (1.center) to (11.center);
				\draw [style=end arrow] (5.center) to (12.center);
				\draw (12.center) to (6.center);
				\draw (11.center) to (0.center);
				\draw [style=open] (13.center) to (17.center);
				\draw [style=open] (17.center) to (14.center);
				\draw [style=open] (21.center) to (25.center);
				\draw [style=open] (25.center) to (26.center);
				\draw [style=end arrow] (19.center) to (27.center);
				\draw [style=end arrow] (23.center) to (28.center);
				\draw (28.center) to (24.center);
				\draw (27.center) to (18.center);
				\draw [style=open] (29.center) to (35.center);
				\draw [style=open] (35.center) to (20.center);
				\draw [style=open] (22.center) to (36.center);
				\draw [style=open] (36.center) to (30.center);
				\draw [style=open] (35.center) to (31.center);
				\draw [style=open] (31.center) to (36.center);
				\draw [style=end arrow] (38.center) to (46.center);
				\draw [style=end arrow] (42.center) to (47.center);
				\draw (47.center) to (43.center);
				\draw (46.center) to (37.center);
				\draw [style=open] (48.center) to (54.center);
				\draw [style=open, bend right=90] (54.center) to (55.center);
				\draw [style=open, bend left=90] (54.center) to (55.center);
				\draw [style=open] (56.center) to (57.center);
				\draw [style=open] (56.center) to (45.center);
				\draw [style=open] (55.center) to (56.center);
				\draw [style=open] (58.center) to (56.center);
			\end{pgfonlayer}
		\end{tikzpicture}
	 \end{array}\label{CFinv}
\end{equation}
and see immediately that $C_F^2=C_F$ if $F$ is special, which we already know from Section~\ref{seccyl}.
Actually, the converse is true: If on the right hand side of~\eqref{CFinv},
 we project to the unit component, we see that $C^2_F=C_F$ implies that the blue diagram on the very right of \eqref{CFinv} needs to agree with the blue part of \eqref{eqnCF}. This implies that $F$ needs to be special.

By definition $\bulk(F)$ is the subobject of $\bzero(F)$ with one more relation, namely
\begin{equation}
	\begin{array}{c}\begin{tikzpicture}[scale=0.5]
			\begin{pgfonlayer}{nodelayer}
				\node [style=none] (0) at (-5, 0) {};
				\node [style=none] (1) at (-5, 4) {};
				\node [style=none] (2) at (-2, 4) {};
				\node [style=none] (3) at (-2, 0) {};
				\node [style=none] (4) at (-5, 2) {};
				\node [style=none] (5) at (-2, 2) {};
				\node [style=none] (6) at (-2, 3) {};
				\node [style=none] (8) at (-2, 4) {};
				\node [style=none] (10) at (-1, 2) {$=$};
				\node [style=none] (11) at (0, 0) {};
				\node [style=none] (12) at (0, 4) {};
				\node [style=none] (13) at (1, 4) {};
				\node [style=none] (14) at (1, 0) {};
				\node [style=none] (15) at (0, 2) {};
				\node [style=none] (16) at (1, 2) {};
				\node [style=none] (19) at (1, 4) {};
				\node [style=none] (21) at (-3.5, 4) {};
				\node [style=none] (22) at (-3.5, 0) {};
				\node [style=none] (23) at (0.5, 4) {};
				\node [style=none] (24) at (0.5, 0) {};
				\node [style=none] (25) at (-3.5, 2) {};
				\node [style=none] (26) at (-4.5, 0) {};
				\node [style=none] (27) at (-2.5, 0) {};
			\end{pgfonlayer}
			\begin{pgfonlayer}{edgelayer}
				\draw [style=end arrow] (1.center) to (4.center);
				\draw (4.center) to (0.center);
				\draw [style=end arrow] (12.center) to (15.center);
				\draw (15.center) to (11.center);
				\draw [style=open] (21.center) to (22.center);
				\draw [style=open] (24.center) to (23.center);
				\draw [style=end arrow] (14.center) to (16.center);
				\draw (16.center) to (19.center);
				\draw [style=end arrow] (3.center) to (5.center);
				\draw (5.center) to (8.center);
				\draw [style=open] (26.center) to (25.center);
				\draw [style=open] (25.center) to (27.center);
			\end{pgfonlayer}
		\end{tikzpicture}
	\end{array} \tag{B2}\label{eqnrelbulk}
\end{equation}

\subsection{The connection to the end over internal endomorphisms and Morita invariance\label{secinternalend}} Let $F$ be a special symmetric Frobenius algebra in a modular category $\cat{A}$.
The object \begin{align} \int_{M\in\Fmod} \HOM_F(M,M)\end{align} from~\eqref{eqninternalend} is proposed as a bulk field candidate in \cite{fspivotal,internal}. 
From the correlator construction given in this article, such a description is not at all obvious
because it does not even use the pivotal module category $\Fmod$.
The following result reconciles both ideas and forms the core of the proof of Theorem~\ref{thmhhfmod}:

\begin{proposition}\label{propbulkend}
	For a special symmetric Frobenius algebra $F$ in a modular category $\cat{A}$, the following two symmetric Frobenius algebras are canonically isomorphic:
	\begin{pnum}
		\item The bulk Frobenius algebra $\Fhb \in Z(\cat{A})$ with the geometric cyclic framed $E_2$-algebra structure in $\cat{A}$ underlying the consistent system of correlators from Theorem~\ref{thmmainlong}.
		\item The symmetric commutative Frobenius algebra $\int_{M\in\Fmod} \HOM_F(M,M)\in Z(\cat{A})$ obtained from the natural endotransformations of the identity of $\Fmod$.
	\end{pnum}
\end{proposition}

\begin{proof}
	As recalled in Section~\ref{secreminderpiv}, $\HOM_F(-,-)$ is exact, and therefore we can use \cite[Proposition~5.1.7]{kl} to write $\int_{M\in\Fmod} \HOM_F(M,M)$ as $\int_{M\in\Proj \Fmod} \HOM_F(M,M)$. (Note that this proposition applies to coends; in order to apply it to ends, we would need to replace projective objects with injective ones, but $\Fmod$ is self-injective.)
	In fact, a further reduction is possible:
	If $P$ is a projective $F$-module in $\cat{A}$, then the free $F$-module $F\otimes \widetilde P$ for a projective object $\widetilde P$ in $\cat{A}$ with an epimorphism
	$\widetilde P \to P$ in $\cat{A}$
	 comes with an epimorphism $F\otimes \widetilde P \to P$ of $F$-modules that splits by projectivity of $P$ and hence exhibits $P$ as direct summand of free modules $F\otimes \widetilde P$ on projective objects of $\cat{A}$.
	With the methods of \cite[Section~2.2]{dva}, the end 
	$\int_{M\in\Fmod} \HOM_F(M,M)$
	is actually canonically isomorphic to $\int_{\substack{\text{free $F$-modules} \\ F\otimes P}} \HOM_F(F\otimes P,F\otimes P)$, where it is understood that the $P$ is projective as object in $\cat{A}$. With $\HOM_F(F\otimes P,F\otimes P)\cong (F\otimes P)^\vee \otimes _F (F\otimes P)\cong P^\vee \otimes F \otimes P$, we see in combination
	with Section~\ref{secrelbulk}
	that $\bulk(F)$ and $\int_{\substack{\text{free $F$-modules} \\ F\otimes P}} \HOM_F(F\otimes P,F\otimes P)$, for the moment both seen as objects in $\cat{A}$, are both subobjects of $\prod_{X \in \cat{A}} X^\vee \otimes F \otimes X$, but with a priori
	different relations:
	For the bulk object $\bulk(F)$, these are \eqref{eqnend} and \eqref{eqnrelbulk} while for $\int_{\substack{\text{free $F$-modules} \\ F\otimes P}} \HOM_F(F\otimes P,F\otimes P)$ the relation is
	\begin{align}
		\begin{array}{c}	\begin{tikzpicture}[scale=0.5]
				\begin{pgfonlayer}{nodelayer}
					\node [style=none] (0) at (-3, 0) {};
					\node [style=none] (1) at (-3, 4) {};
					\node [style=none] (2) at (-1, 4) {};
					\node [style=none] (3) at (-1, 0) {};
					\node [style=none] (4) at (-3, 2) {};
					\node [style=none] (5) at (-1, 2) {};
					\node [style=none] (6) at (-1, 3) {};
					\node [style=none] (7) at (-1, 2.75) {};
					\node [style=none] (8) at (-1, 4) {};
					\node [style=none] (10) at (0, 2) {$=$};
					\node [style=none] (11) at (1, 0) {};
					\node [style=none] (12) at (1, 4) {};
					\node [style=none] (13) at (3, 4) {};
					\node [style=none] (14) at (3, 0) {};
					\node [style=none] (15) at (1, 2.75) {};
					\node [style=none] (16) at (3, 2) {};
					\node [style=none] (19) at (3, 4) {};
					\node [style=none] (21) at (-2, 4) {};
					\node [style=none] (22) at (-2, 0) {};
					\node [style=none] (23) at (2, 4) {};
					\node [style=none] (24) at (2, 0) {};
					\node [style=bigbox] (25) at (-1.5, 1) {$g$};
					\node [style=bigbox] (26) at (1.5, 1) {$g^\vee$};
				\end{pgfonlayer}
				\begin{pgfonlayer}{edgelayer}
					\draw [style=end arrow] (1.center) to (4.center);
					\draw (4.center) to (0.center);
					\draw [style=end arrow] (5.center) to (7.center);
					\draw (7.center) to (8.center);
					\draw (3.center) to (5.center);
					\draw [style=end arrow] (12.center) to (15.center);
					\draw (15.center) to (11.center);
					\draw [style=open] (21.center) to (22.center);
					\draw [style=open] (24.center) to (23.center);
					\draw [style=end arrow] (14.center) to (16.center);
					\draw (16.center) to (19.center);
				\end{pgfonlayer}
			\end{tikzpicture}
		\end{array}\label{eqnrelationend}
	\end{align}
	for morphisms $g$ of free modules. The latter are exactly  of the form
	\begin{align}
\begin{tikzpicture}[scale=0.5]
	\begin{pgfonlayer}{nodelayer}
		\node [style=none] (2) at (-1, 4) {};
		\node [style=none] (3) at (-1, 0) {};
		\node [style=none] (5) at (-1, 2) {};
		\node [style=none] (6) at (-1, 3) {};
		\node [style=none] (7) at (-1, 2.75) {};
		\node [style=none] (8) at (-1, 4) {};
		\node [style=none] (10) at (0, 2) {$=$};
		\node [style=none] (13) at (3, 4) {};
		\node [style=none] (14) at (3, 0) {};
		\node [style=none] (16) at (3, 2) {};
		\node [style=none] (19) at (3, 4) {};
		\node [style=none] (21) at (-2, 4) {};
		\node [style=none] (22) at (-2, 0) {};
		\node [style=none] (23) at (2, 4) {};
		\node [style=none] (24) at (2.5, 1) {};
		\node [style=bigbox] (25) at (-1.5, 1) {$g$};
		\node [style=bigbox] (26) at (2.75, 1) {$h$};
		\node [style=none] (27) at (1, 0) {};
		\node [style=none] (28) at (2.25, 2.5) {};
	\end{pgfonlayer}
	\begin{pgfonlayer}{edgelayer}
		\draw [style=end arrow] (5.center) to (7.center);
		\draw (7.center) to (8.center);
		\draw (3.center) to (5.center);
		\draw [style=open] (21.center) to (22.center);
		\draw [style=open] (24.center) to (23.center);
		\draw [style=end arrow] (14.center) to (16.center);
		\draw (16.center) to (19.center);
		\draw [style=open] (27.center) to (28.center);
	\end{pgfonlayer}
\end{tikzpicture}
	\end{align} with $h$ being just a morphism in $\cat{A}$, so that \eqref{eqnrelationend} becomes
	\begin{align}
		\begin{array}{c}\begin{tikzpicture}[scale=0.5]
				\begin{pgfonlayer}{nodelayer}
					\node [style=none] (0) at (0, 0) {};
					\node [style=none] (1) at (0, 4) {};
					\node [style=none] (2) at (2, 4) {};
					\node [style=none] (3) at (2, 0) {};
					\node [style=none] (4) at (0, 2) {};
					\node [style=none] (5) at (2, 2) {};
					\node [style=none] (6) at (2, 3) {};
					\node [style=none] (7) at (2, 3) {};
					\node [style=none] (8) at (2, 4) {};
					\node [style=none] (9) at (3.5, 2) {$=$};
					\node [style=none] (10) at (5, 0) {};
					\node [style=none] (11) at (5, 4) {};
					\node [style=none] (12) at (7, 4) {};
					\node [style=none] (13) at (7, 0) {};
					\node [style=none] (14) at (5, 3.25) {};
					\node [style=none] (15) at (7, 2) {};
					\node [style=none] (16) at (7, 4) {};
					\node [style=none] (17) at (1, 4) {};
					\node [style=none] (18) at (1.5, 1) {};
					\node [style=none] (19) at (5.5, 2.25) {};
					\node [style=none] (20) at (6, 0) {};
					\node [style=bigbox] (21) at (1.75, 1) {$h$};
					\node [style=bigbox] (22) at (5, 2.25) {$h^\vee$};
					\node [style=none] (23) at (6, 4) {};
					\node [style=none] (24) at (1, 0) {};
					\node [style=none] (25) at (1.25, 2.5) {};
					\node [style=none] (26) at (5.75, 1) {};
				\end{pgfonlayer}
				\begin{pgfonlayer}{edgelayer}
					\draw [style=end arrow] (1.center) to (4.center);
					\draw (4.center) to (0.center);
					\draw [style=end arrow] (5.center) to (7.center);
					\draw (7.center) to (8.center);
					\draw (3.center) to (5.center);
					\draw [style=end arrow] (11.center) to (14.center);
					\draw (14.center) to (10.center);
					\draw [style=open] (17.center) to (18.center);
					\draw [style=open] (20.center) to (19.center);
					\draw [style=end arrow] (13.center) to (15.center);
					\draw (15.center) to (16.center);
					\draw [style=open, bend left, looseness=1.25] (24.center) to (25.center);
					\draw [style=open, bend left=15, looseness=1.25] (23.center) to (26.center);
				\end{pgfonlayer}
			\end{tikzpicture}
		\end{array}\label{eqnrelationend2}
	\end{align}
	For this reason, we need to prove
	\begin{align}
	\label{eqnproofcomp}	\begin{array}{c}	\eqref{eqnend} \quad \&\quad  \eqref{eqnrelbulk} \quad \Longleftrightarrow \quad \eqref{eqnrelationend2} \ . \end{array}
	\end{align}
	Suppose that \eqref{eqnend} and  \eqref{eqnrelbulk} hold.
	Then
	\begin{align}
		\begin{tikzpicture}[scale=0.5]
			\begin{pgfonlayer}{nodelayer}
				\node [style=none] (0) at (0, 0) {};
				\node [style=none] (1) at (0, 4) {};
				\node [style=none] (2) at (2, 4) {};
				\node [style=none] (3) at (2, 0) {};
				\node [style=none] (4) at (0, 2) {};
				\node [style=none] (5) at (2, 2) {};
				\node [style=none] (6) at (2, 3) {};
				\node [style=none] (7) at (2, 3) {};
				\node [style=none] (8) at (2, 4) {};
				\node [style=none] (9) at (3.5, 2) {$\stackrel{\eqref{eqnrelbulk}}{=}$};
				\node [style=none] (10) at (28.25, 0) {};
				\node [style=none] (11) at (28.25, 4) {};
				\node [style=none] (12) at (30, 4) {};
				\node [style=none] (13) at (30, 0) {};
				\node [style=none] (14) at (28.25, 3.25) {};
				\node [style=none] (15) at (30, 2) {};
				\node [style=none] (16) at (30, 4) {};
				\node [style=none] (17) at (1, 4) {};
				\node [style=none] (18) at (1.5, 1) {};
				\node [style=none] (19) at (28.75, 2.25) {};
				\node [style=none] (20) at (29.25, 0) {};
				\node [style=bigbox] (21) at (1.75, 1) {$h$};
				\node [style=bigbox] (22) at (28.25, 2.25) {$h^\vee$};
				\node [style=none] (23) at (29.25, 4) {};
				\node [style=none] (24) at (1, 0) {};
				\node [style=none] (25) at (1.25, 2.5) {};
				\node [style=none] (26) at (29, 1) {};
				\node [style=none] (27) at (4.5, 0) {};
				\node [style=none] (28) at (4.5, 4) {};
				\node [style=none] (29) at (8.25, 4) {};
				\node [style=none] (30) at (8.25, 0) {};
				\node [style=none] (31) at (4.5, 2) {};
				\node [style=none] (32) at (8.25, 2) {};
				\node [style=none] (33) at (8.25, 3) {};
				\node [style=none] (34) at (8.25, 3) {};
				\node [style=none] (35) at (8.25, 4) {};
				\node [style=none] (36) at (6.25, 4) {};
				\node [style=none] (37) at (7.75, 1) {};
				\node [style=bigbox] (38) at (8, 1) {$h$};
				\node [style=none] (39) at (6.25, 0) {};
				\node [style=none] (40) at (7, 2.5) {};
				\node [style=none] (41) at (5.25, 0) {};
				\node [style=none] (42) at (7.25, 0) {};
				\node [style=none] (43) at (6.5, 0.75) {};
				\node [style=none] (44) at (10, 0) {};
				\node [style=none] (45) at (10, 4) {};
				\node [style=none] (46) at (13.5, 4) {};
				\node [style=none] (47) at (13.5, 0) {};
				\node [style=none] (48) at (10, 2) {};
				\node [style=none] (49) at (13.5, 2) {};
				\node [style=none] (50) at (13.5, 3) {};
				\node [style=none] (51) at (13.5, 3) {};
				\node [style=none] (52) at (13.5, 4) {};
				\node [style=none] (53) at (11.5, 4) {};
				\node [style=none] (54) at (13, 1) {};
				\node [style=bigbox] (55) at (13.25, 1) {$h$};
				\node [style=none] (56) at (11.5, 0) {};
				\node [style=none] (57) at (11.75, 3.5) {};
				\node [style=none] (58) at (10.5, 0) {};
				\node [style=none] (59) at (12.5, 0) {};
				\node [style=none] (61) at (9.25, 2) {$=$};
				\node [style=none] (62) at (12.25, 2.5) {};
				\node [style=none] (63) at (11, 1) {};
				\node [style=none] (64) at (15, 2) {$\stackrel{\eqref{eqnend}}{=}$};
				\node [style=none] (65) at (16.5, 0) {};
				\node [style=none] (66) at (16.5, 4) {};
				\node [style=none] (67) at (20, 4) {};
				\node [style=none] (68) at (20, 0) {};
				\node [style=none] (69) at (16.5, 3) {};
				\node [style=none] (70) at (20, 2) {};
				\node [style=none] (71) at (20, 3) {};
				\node [style=none] (72) at (20, 3) {};
				\node [style=none] (73) at (20, 4) {};
				\node [style=none] (74) at (18.5, 4) {};
				\node [style=none] (77) at (18.5, 0) {};
				\node [style=none] (79) at (17.5, 0) {};
				\node [style=none] (80) at (19.5, 0) {};
				\node [style=none] (84) at (17.5, 0.5) {};
				\node [style=none] (85) at (17, 1.75) {};
				\node [style=bigbox] (86) at (16.75, 1.75) {$h^\vee$};
				\node [style=none] (87) at (18, 1.5) {};
				\node [style=none] (88) at (18.5, 3) {};
				\node [style=none] (89) at (21, 2) {$=$};
				\node [style=none] (90) at (22, 0) {};
				\node [style=none] (91) at (22, 4) {};
				\node [style=none] (92) at (25.5, 4) {};
				\node [style=none] (93) at (25.5, 0) {};
				\node [style=none] (94) at (22, 3) {};
				\node [style=none] (95) at (25.5, 2) {};
				\node [style=none] (96) at (25.5, 3) {};
				\node [style=none] (97) at (25.5, 3) {};
				\node [style=none] (98) at (25.5, 4) {};
				\node [style=none] (99) at (24, 4) {};
				\node [style=none] (100) at (24, 0) {};
				\node [style=none] (101) at (23, 0) {};
				\node [style=none] (102) at (25, 0) {};
				\node [style=none] (103) at (23, 0.5) {};
				\node [style=none] (104) at (22.5, 1.75) {};
				\node [style=bigbox] (105) at (22.25, 1.75) {$h^\vee$};
				\node [style=none] (106) at (24, 3) {};
				\node [style=none] (107) at (24, 3) {};
				\node [style=none] (108) at (26.5, 2) {$\stackrel{\eqref{eqnrelbulk}}{=}$};
			\end{pgfonlayer}
			\begin{pgfonlayer}{edgelayer}
				\draw [style=end arrow] (1.center) to (4.center);
				\draw (4.center) to (0.center);
				\draw [style=end arrow] (5.center) to (7.center);
				\draw (7.center) to (8.center);
				\draw (3.center) to (5.center);
				\draw [style=end arrow] (11.center) to (14.center);
				\draw (14.center) to (10.center);
				\draw [style=open] (17.center) to (18.center);
				\draw [style=open] (20.center) to (19.center);
				\draw [style=end arrow] (13.center) to (15.center);
				\draw (15.center) to (16.center);
				\draw [style=open, bend left, looseness=1.25] (24.center) to (25.center);
				\draw [style=open, bend left=15, looseness=1.25] (23.center) to (26.center);
				\draw [style=end arrow] (28.center) to (31.center);
				\draw (31.center) to (27.center);
				\draw [style=end arrow] (32.center) to (34.center);
				\draw (34.center) to (35.center);
				\draw (30.center) to (32.center);
				\draw [style=open] (36.center) to (37.center);
				\draw [style=open] (39.center) to (40.center);
				\draw [style=open] (41.center) to (43.center);
				\draw [style=open] (43.center) to (42.center);
				\draw [style=end arrow] (45.center) to (48.center);
				\draw (48.center) to (44.center);
				\draw [style=end arrow] (49.center) to (51.center);
				\draw (51.center) to (52.center);
				\draw (47.center) to (49.center);
				\draw [style=open] (53.center) to (54.center);
				\draw [style=open, bend left=45] (59.center) to (62.center);
				\draw [style=open] (58.center) to (63.center);
				\draw [style=open] (63.center) to (56.center);
				\draw [style=open] (63.center) to (57.center);
				\draw [style=end arrow] (66.center) to (69.center);
				\draw (69.center) to (65.center);
				\draw [style=end arrow] (70.center) to (72.center);
				\draw (72.center) to (73.center);
				\draw (68.center) to (70.center);
				\draw [style=open] (79.center) to (84.center);
				\draw [style=open] (84.center) to (85.center);
				\draw [style=open] (84.center) to (87.center);
				\draw [style=open] (87.center) to (77.center);
				\draw [style=open] (87.center) to (88.center);
				\draw [style=open] (88.center) to (74.center);
				\draw [style=open] (88.center) to (80.center);
				\draw [style=end arrow] (91.center) to (94.center);
				\draw (94.center) to (90.center);
				\draw [style=end arrow] (95.center) to (97.center);
				\draw (97.center) to (98.center);
				\draw (93.center) to (95.center);
				\draw [style=open] (101.center) to (103.center);
				\draw [style=open] (103.center) to (104.center);
				\draw [style=open] (103.center) to (106.center);
				\draw [style=open] (106.center) to (100.center);
				\draw [style=open] (106.center) to (107.center);
				\draw [style=open] (107.center) to (99.center);
				\draw [style=open] (107.center) to (102.center);
			\end{pgfonlayer}
		\end{tikzpicture}
		 \ . 
	\end{align}
	Conversely, assume that \eqref{eqnrelationend2} holds. In order to prove \eqref{eqnend} for a morphism $f:X\to Y$ in $\cat{A}$, one uses simply \eqref{eqnrelationend2} for $h : X\ra{f} Y \ra{\text{unit of $F$}} F \otimes Y$. It remains to prove~\eqref{eqnrelbulk}. This is accomplished as follows:
	\begin{align}
	\begin{tikzpicture}[scale=0.5]
		\begin{pgfonlayer}{nodelayer}
			\node [style=none] (109) at (-2, -1) {};
			\node [style=none] (110) at (-2, 3) {};
			\node [style=none] (111) at (1, 3) {};
			\node [style=none] (112) at (1, -1) {};
			\node [style=none] (113) at (-2, 1) {};
			\node [style=none] (114) at (1, 1) {};
			\node [style=none] (115) at (1, 2) {};
			\node [style=none] (116) at (1, 3) {};
			\node [style=none] (117) at (2, 1) {$=$};
			\node [style=none] (118) at (13, -1) {};
			\node [style=none] (119) at (13, 3) {};
			\node [style=none] (120) at (14, 3) {};
			\node [style=none] (121) at (14, -1) {};
			\node [style=none] (122) at (13, 1) {};
			\node [style=none] (123) at (14, 1) {};
			\node [style=none] (124) at (14, 3) {};
			\node [style=none] (125) at (-0.5, 3) {};
			\node [style=none] (126) at (-0.5, -1) {};
			\node [style=none] (127) at (13.5, 3) {};
			\node [style=none] (128) at (13.5, -1) {};
			\node [style=none] (129) at (-0.5, 1) {};
			\node [style=none] (130) at (-1.5, -1) {};
			\node [style=none] (131) at (0.5, -1) {};
			\node [style=none] (132) at (3, -1) {};
			\node [style=none] (133) at (3, 3) {};
			\node [style=none] (134) at (6, 3) {};
			\node [style=none] (135) at (6, -1) {};
			\node [style=none] (136) at (3, 1) {};
			\node [style=none] (137) at (6, 1) {};
			\node [style=none] (138) at (6, 2) {};
			\node [style=none] (139) at (6, 3) {};
			\node [style=none] (140) at (4.5, 3) {};
			\node [style=none] (141) at (4.5, -1) {};
			\node [style=none] (142) at (4.5, 1) {};
			\node [style=none] (143) at (3.5, -1) {};
			\node [style=none] (144) at (5.5, -1) {};
			\node [style=none] (145) at (5, 0) {};
			\node [style=none] (146) at (7, 1) {$\stackrel{\eqref{eqnrelationend2}}{=}$};
			\node [style=none] (147) at (8, -1) {};
			\node [style=none] (148) at (8, 3) {};
			\node [style=none] (149) at (11, 3) {};
			\node [style=none] (150) at (11, -1) {};
			\node [style=none] (151) at (8, 1) {};
			\node [style=none] (152) at (11, 1) {};
			\node [style=none] (153) at (11, 2) {};
			\node [style=none] (154) at (11, 3) {};
			\node [style=none] (155) at (9.5, 3) {};
			\node [style=none] (157) at (9.5, 1.5) {};
			\node [style=none] (158) at (9.5, 0) {};
			\node [style=none] (159) at (9.5, -1) {};
			\node [style=none] (160) at (12, 1) {$=$};
		\end{pgfonlayer}
		\begin{pgfonlayer}{edgelayer}
			\draw [style=end arrow] (110.center) to (113.center);
			\draw (113.center) to (109.center);
			\draw [style=end arrow] (119.center) to (122.center);
			\draw (122.center) to (118.center);
			\draw [style=open] (125.center) to (126.center);
			\draw [style=open] (128.center) to (127.center);
			\draw [style=end arrow] (121.center) to (123.center);
			\draw (123.center) to (124.center);
			\draw [style=end arrow] (112.center) to (114.center);
			\draw (114.center) to (116.center);
			\draw [style=open] (130.center) to (129.center);
			\draw [style=open] (129.center) to (131.center);
			\draw [style=end arrow] (133.center) to (136.center);
			\draw (136.center) to (132.center);
			\draw [style=end arrow] (135.center) to (137.center);
			\draw (137.center) to (139.center);
			\draw [style=open] (143.center) to (142.center);
			\draw [style=open] (141.center) to (145.center);
			\draw [style=open] (145.center) to (144.center);
			\draw [style=open] (145.center) to (142.center);
			\draw [style=open] (142.center) to (140.center);
			\draw [style=end arrow] (148.center) to (151.center);
			\draw (151.center) to (147.center);
			\draw [style=end arrow] (150.center) to (152.center);
			\draw (152.center) to (154.center);
			\draw [style=open] (157.center) to (155.center);
			\draw [style=open] (159.center) to (158.center);
			\draw [style=open, bend right=90, looseness=2.00] (158.center) to (157.center);
			\draw [style=open, bend left=90, looseness=2.00] (158.center) to (157.center);
		\end{pgfonlayer}
	\end{tikzpicture} 
	\end{align}
	This proves~\eqref{eqnproofcomp}. 
	
We conclude that the bulk object $\bulk(F)$ and $\int_{M\in \Fmod} \HOM_F(M,M)$ are canonically isomorphic as objects in $\cat{A}$ since they can be identified with the same subobject of $\int_{X\in \cat{A}} X^\vee \otimes F \otimes X$.
	Of course,   $\int_{X\in \cat{A}} X^\vee \otimes F \otimes X$ has a half braiding by virtue of being $RF$ for the right adjoint $R:\cat{A}\to Z(\cat{A})$ to the forgetful functor $U:Z(\cat{A})\to\cat{A}$. The half braiding on $\bulk(F)$ and $\int_{M\in \Fmod} \HOM_F(M,M)$ comes in both cases by restriction of this half braiding along the inclusion. For $\bulk(F)$, this holds by construction and the comments made in Section~\ref{secrelbulk}.
	For $\int_{M\in \Fmod} \HOM_F(M,M)$, this is a part of the construction in \cite{internal}.
	This identifies $\bulk(F)$ and $\int_{M\in \Fmod} \HOM_F(M,M)$ as objects in $Z(\cat{A})$.
	
	By definition $\Fhb$ is the bulk object, seen as object in $Z(\cat{A})$, with its topologically inherited multiplication and Frobenius structure.
	As one can see again with the tools from \cite[Section~9]{microcosm}, the multiplication is 
	induced by the map $\int_{X\in\cat{A}} X^\vee \otimes F \otimes X \otimes \int_{Y\in\cat{A}}Y^\vee \otimes F \otimes Y\to F$
	given by
	\begin{align}
		\begin{tikzpicture}[scale=0.5]
			\begin{pgfonlayer}{nodelayer}
				\node [style=none] (0) at (-5, 0) {};
				\node [style=none] (1) at (-4, 0) {};
				\node [style=none] (2) at (-3, 0) {};
				\node [style=none] (3) at (-2, 0) {};
				\node [style=none] (4) at (-1, 0) {};
				\node [style=none] (5) at (0, 0) {};
				\node [style=none] (6) at (-1, 1) {};
				\node [style=none] (7) at (-4, 1) {};
				\node [style=none] (8) at (-2.5, 2.5) {};
				\node [style=none] (9) at (-2.5, 4) {};
			\end{pgfonlayer}
			\begin{pgfonlayer}{edgelayer}
				\draw [style=open] (0.center) to (7.center);
				\draw [style=open] (7.center) to (1.center);
				\draw [style=open] (7.center) to (2.center);
				\draw [style=open] (3.center) to (6.center);
				\draw [style=open] (6.center) to (4.center);
				\draw [style=open] (6.center) to (5.center);
				\draw [style=open] (7.center) to (8.center);
				\draw [style=open] (8.center) to (6.center);
				\draw [style=open] (8.center) to (9.center);
			\end{pgfonlayer}
		\end{tikzpicture}
		\end{align}
		(in the first step, we project to the $F$-component of the end; this is  suppressed as usual).
		Next we need to turn this into a binary operation on $RF$ using the adjunction $U\dashv R$
	(the fact that this is how we obtain the corresponding morphism in $Z(\cat{A})$ follows because, by \cite[Theorem 4.2]{envas}, we can use the string-net model \cite{sn} which produces the morphisms in the center via this adjunction).
	 This resulting map $RF\otimes RF\to RF$ restricts to the bulk object and is there given by
	 	\begin{align}
	 	\begin{tikzpicture}[scale=0.5]
	 		\begin{pgfonlayer}{nodelayer}
	 			\node [style=none] (0) at (-6, 0) {};
	 			\node [style=none] (1) at (-5, 0) {};
	 			\node [style=none] (2) at (-4, 0) {};
	 			\node [style=none] (3) at (-1, 0) {};
	 			\node [style=none] (4) at (0, 0) {};
	 			\node [style=none] (5) at (1, 0) {};
	 			\node [style=none] (6) at (0, 1) {};
	 			\node [style=none] (7) at (-5, 1) {};
	 			\node [style=none] (8) at (-2.5, 2.5) {};
	 			\node [style=none] (9) at (-2.5, 4) {};
	 			\node [style=none] (10) at (-3, 0) {};
	 			\node [style=none] (11) at (-2, 0) {};
	 			\node [style=none] (12) at (-7, 0) {};
	 			\node [style=none] (13) at (-7, 4) {};
	 			\node [style=none] (14) at (2, 4) {};
	 			\node [style=none] (15) at (2, 0) {};
	 			\node [style=none] (16) at (3, 2) {$\stackrel{\text{(B2)}}{=}$};
	 			\node [style=none] (18) at (5, 0) {};
	 			\node [style=none] (21) at (8, 0) {};
	 			\node [style=none] (25) at (6.5, 2.5) {};
	 			\node [style=none] (26) at (6.5, 4) {};
	 			\node [style=none] (27) at (6, 0) {};
	 			\node [style=none] (28) at (7, 0) {};
	 			\node [style=none] (29) at (4, 0) {};
	 			\node [style=none] (30) at (4, 4) {};
	 			\node [style=none] (31) at (9, 4) {};
	 			\node [style=none] (32) at (9, 0) {};
	 		\end{pgfonlayer}
	 		\begin{pgfonlayer}{edgelayer}
	 			\draw [style=open] (0.center) to (7.center);
	 			\draw [style=open] (7.center) to (1.center);
	 			\draw [style=open] (7.center) to (2.center);
	 			\draw [style=open] (3.center) to (6.center);
	 			\draw [style=open] (6.center) to (4.center);
	 			\draw [style=open] (6.center) to (5.center);
	 			\draw [style=open] (7.center) to (8.center);
	 			\draw [style=open] (8.center) to (6.center);
	 			\draw [style=open] (8.center) to (9.center);
	 			\draw [bend left=90, looseness=3.00] (10.center) to (11.center);
	 			\draw (12.center) to (13.center);
	 			\draw (15.center) to (14.center);
	 			\draw [style=open] (25.center) to (26.center);
	 			\draw [bend left=90, looseness=3.00] (27.center) to (28.center);
	 			\draw (29.center) to (30.center);
	 			\draw (32.center) to (31.center);
	 			\draw [style=open] (18.center) to (25.center);
	 			\draw [style=open] (25.center) to (21.center);
	 		\end{pgfonlayer}
	 	\end{tikzpicture}
	 \end{align}
	In other words, the product on the bulk object is induced by the multiplication on $F$ and the lax monoidal structure on $R$.
	For $\int_{M\in \Fmod} \HOM_F(M,M)$, the multiplication comes from the composition on the internal endomorphisms. When described on the isomorphic end 
	$\int_{\substack{\text{free $F$-modules} \\ F\otimes P}} \HOM_F(F\otimes P,F\otimes P)$ over free modules, the internal endomorphisms are $P^\vee \otimes F \otimes P$, and the composition of internal endomorphisms uses the evaluation of objects in $\cat{A}$ and the multiplication of $F$. This is the restriction of the aforementioned product on $RF$.
	Therefore, the multiplications agree.
	
	A similar consideration applies to the Frobenius structure. We just give the main idea:
	The symmetric non-degenerate pairing on the bulk object was constructed in Section~\ref{seccyl} using reflection equivariance, which amounted to a trivialization of the distinguished invertible object, and self-duality of $F$. This is what trivializes the Serre functor for $\Fmod$~\cite[Theorem 4.26]{fss} and is used in \cite{internal} to obtain the Frobenius structure on $\int_{M\in \Fmod} \HOM_F(M,M)$. This way, one observes that the Frobenius structures agree.
\end{proof}

For the correlator construction in Theorem~\ref{thmmainlong}, it is not obvious at all that the closed part of the
construction depends on the special symmetric Frobenius algebra only up to Morita equivalence.
Thanks to
Proposition~\ref{propbulkend}
however, this is now clear:

\begin{corollary}[Morita invariance]\label{cormoritainvarianz}
	For a modular category $\cat{A}$, the underlying bulk field correlator associated to a special symmetric Frobenius algebra $F$ depends on $F$ only up to Morita equivalence.
\end{corollary}

\subsection{Cochains on the loop objects versus Hochschild cochains: The end of the proof of Theorem~\ref{thmhhfmod}\label{secproofloop}}
We are now finally in a position to prove the equivalence
\begin{align}
	\LO(\cat{A};F) \simeq CH^*(\Fmod)\label{eqnlohh}
\end{align} of framed $E_2$-algebras for any special symmetric Frobenius algebra $F\in\cat{A}$ in a modular category.

\begin{proof}[\slshape Proof of Theorem~\ref{thmhhfmod}]
	Through the construction in Proposition~\ref{prophi}, $\LO(\cat{A};F)=\mathbb{R}\Hom_\cat{A}(I,U\Fhb)$ as framed $E_2$-algebra.
	The framed $E_2$-structure and the functoriality of
	$\mathbb{R}\Hom_\cat{A}(I,U-)$ in the braided commutative (symmetric Frobenius) algebra in the Drinfeld center $Z(\cat{A})$ are a consequence of~\cite{homotopyinvariants}.
	
	The main point is that by
	Proposition~\ref{propbulkend} we have $\Fhb \cong \int_{M\in \Fmod} \HOM_F(M,M)$ as 
	braided commutative symmetric Frobenius algebras in $Z(\cat{A})$. 
	This gives us \begin{align} \LO(\cat{A};F) \simeq \mathbb{R}\Hom_\cat{A}\left(I,\int_{M\in \Fmod} \HOM_F(M,M)\right) \end{align} as differential graded framed $E_2$-algebras. Now the statement follows from \cite[Theorem~5.11 \& Section~7]{homotopyinvariants}.
\end{proof}

\begin{example}	In the Cardy case $F=I$, we have $\LO(\cat{A};F)=\mathbb{R}\Hom_\cat{A}(I,\mathbb{A})$ for the end $\mathbb{A}=\int_{X \in \cat{A}} X^\vee \otimes X$.
The right hand side of~\eqref{eqnlohh} reduces
to the Hochschild cochains of $\cat{A}$.
Suppose that $\cat{A}$ is given by finite-dimensional modules over a ribbon factorizable Hopf algebra $H$, then $\mathbb{A}$ is $H$ itself with the adjoint action, and \eqref{eqnlohh} reduces, at cohomology level, to the classical isomorphism $\Ext_H(k,H_\text{adj})\cong HH^*(H)$ that goes back to Cartan-Eilenberg~\cite{cartaneilenberg}, see e.g.~\cite[Proposition~2.1]{bichon}. 
The statement on the level of framed $E_2$-algebras is a consequence of the approach in~\cite{homotopyinvariants} to the cyclic Deligne conjecture via braided commutative algebras in the Drinfeld center.
This applies for example to the triplet at integral parameter thanks to the description via a suitable small quantum group in \cite{gannonnegron}. 
In that case,  $HH^*(\cat{A})$ as graded commutative algebra and hence also   $H^*\LO(\cat{A};I)$ can be explicitly described \cite{lq} in terms of generators and relations; in particular, this algebra is not concentrated in degree zero.
 The study of $\LO(\cat{A};F)$ in the general case lies beyond the scope of this article.
\end{example}

	\small	
\newcommand{\etalchar}[1]{$^{#1}$}

\vspace*{0.3cm} \noindent  \textsc{Université Bourgogne Europe, CNRS, IMB UMR 5584, F-21000 Dijon, France}

\end{document}

%% file: mydiagrams.tikzstyles
\tikzstyle{black dot}=[fill=black, draw=black, shape=circle, minimum size=3pt, inner sep=0pt]
\tikzstyle{black dot small}=[fill=black, draw=black, shape=circle, minimum size=2pt, inner sep=0pt]
\tikzstyle{fblack dot}=[fill=black, draw=red, shape=circle, minimum size=2pt, inner sep=0pt]
\tikzstyle{wbox}=[fill=white, draw=black, shape=rectangle, minimum height=0.5cm, minimum width=0.01cm]
\tikzstyle{bbox}=[fill=white, draw=blue, shape=rectangle, minimum height=0.5cm, minimum width=0.01cm]
\tikzstyle{rbox}=[fill=white, draw=red, shape=rectangle, minimum height=0.5cm, minimum width=0.01cm]
\tikzstyle{bwbox}=[draw=blue, shape=rectangle, minimum width=2cm, minimum height=0.5cm]
\tikzstyle{bbwbox}=[draw=blue, shape=rectangle, minimum width=1cm, minimum height=1cm]
\tikzstyle{big white circle}=[fill=white, draw=black, shape=circle, minimum width=0.75cm]
\tikzstyle{white dot big}=[fill=white, draw=black, shape=circle, inner sep=1pt]
\tikzstyle{white dot}=[fill=white, draw=black, shape=circle, minimum size=3pt, inner sep=0pt]
\tikzstyle{flat box}=[fill=white, draw=black, shape=rectangle, minimum width=1.3cm, minimum height=0.5cm,fill=morphismcolor]
\tikzstyle{square}=[fill=white, draw=black, shape=rectangle]
\tikzstyle{flat box 2}=[fill=white, draw=black, shape=rectangle, minimum height=0.5cm, minimum width=0.01cm,fill=morphismcolor]
\tikzstyle{bigbox}=[fill=white, draw=black, shape=rectangle, minimum height=0.5cm, minimum width=0.8cm,fill=white]
\tikzstyle{over }=[front]
\tikzstyle{theta}=[fill=blue, draw=blue, shape=ellipse, minimum height=6pt, minimum width=6pt, inner sep=0pt]
\tikzstyle{thetabig}=[fill=blue, draw=blue, shape=ellipse, minimum width=1cm, minimum height=0.01cm]
\tikzstyle{thetainv}=[fill=blue, draw=red, shape=ellipse, minimum height=6pt, minimum width=6pt, inner sep=0pt]
\tikzstyle{thetabinv}=[fill=blue, draw=red, shape=ellipse, minimum width=1cm, minimum height=0.01cm]
\tikzstyle{bigdisk}=[draw=black, shape=circle, minimum width=3cm]
\tikzstyle{wdisk}=[shape=circle, minimum width=0.48cm,fill=white]
\tikzstyle{bigdisk2}=[draw=black, fill=lightgray, shape=circle, minimum width=3cm]
\tikzstyle{little disk}=[fill=white, draw=black, shape=circle, minimum width=0.5cm]

\tikzstyle{mid arrow}=[-, postaction={on each segment={mid arrow}}]
\tikzstyle{end arrow}=[->]
\tikzstyle{mover}=[-, link]
\tikzstyle{string}=[-, draw=blue,postaction={on each segment={mid arrow}}]
\tikzstyle{stringd}=[-, dotted,draw=blue,postaction={on each segment={mid arrow}}]
\tikzstyle{red}=[-, dotted,draw=red]
\tikzstyle{mydots}=[-,dotted,dashed,draw=gray]
\tikzstyle{mydotsblack}=[-,dotted,dashed,draw=black]
\tikzstyle{open}=[-, line width=1pt,draw=blue]
\tikzstyle{thick}=[-,line width=1pt]
\tikzstyle{rarrow}=[->,draw=red]
\tikzstyle{red mid arrow}=[-, draw={rgb,255: red,214; green,42; blue,51}, postaction={on each segment={mid arrow}}, line width=1pt]
\tikzstyle{RED}=[-, draw={rgb,255: red,214; green,42; blue,51}]
\tikzstyle{REDdashed}=[-,dashed, draw={rgb,255: red,214; green,42; blue,51}]
\tikzstyle{REDarrow}=[->, draw={rgb,255: red,214; green,42; blue,51}]
\tikzstyle{darrow}=[->,dotted]
\tikzstyle{blue}=[-, draw=blue]
\tikzstyle{blue mid arrow}=[-, draw={rgb,255: red,23; green,37; blue,167}, postaction={on each segment={mid arrow}}, line width=1pt]
\tikzstyle{over}=[-, link]
\tikzstyle{bover}=[-, blink]
\tikzstyle{mover}=[-, link]
\tikzstyle{mapsto}=[{|->}]